\documentclass[final,3p,sort&compress]{elsarticle}

\usepackage{bm}
\usepackage{bbm}
\usepackage{mathrsfs}
\usepackage{amsmath}
\usepackage{amssymb}
\usepackage{amsthm}

\usepackage{graphicx}
\usepackage{xcolor}
\usepackage{array}
\usepackage{needspace}

\usepackage[all]{xy}
\usepackage[hidelinks,bookmarksopen=true,bookmarksnumbered=true]{hyperref}
\hypersetup{
    pdftitle={Uniform Approximation for Operator Learning on Compact Sets},
    pdfkeywords={Operator learning; Uniform approximation property; Nonlinear approximation; Stable linear recovery; Kolmogorov widths; Neural network approximation}
}

\newtheorem{theorem}{Theorem}[section]
\newtheorem{lemma}[theorem]{Lemma}
\newtheorem{proposition}[theorem]{Proposition}
\newtheorem{corollary}[theorem]{Corollary}
\newtheorem{assumption}[theorem]{Assumption}

\theoremstyle{definition}
\newtheorem{definition}[theorem]{Definition}
\newtheorem{example}[theorem]{Example}

\theoremstyle{remark}
\newtheorem{remark}[theorem]{Remark}

\numberwithin{equation}{section}

\allowdisplaybreaks
\makeatletter
\AddToHook{cmd/appendix/after}{
    
    \def\ACHA@prefix@section{\appendixname~}
    \def\@seccntformat#1{\csname ACHA@prefix@#1\endcsname\csname the#1\endcsname\@seccntDot\hskip 0.5em}
}
\makeatother
\journal{Applied and Computational Harmonic Analysis}

\begin{document}


\def\supp{\operatorname{supp}}
\def\Span{\operatorname{span}}
\def\cspan{\overline{\operatorname{span}}}
\def\dist{\operatorname{dist}}

\newcommand{\one}{\mathbbm{1}}
\newcommand{\rmd}{\mathrm{d}}
\newcommand{\N}{\mathbb{N}}
\newcommand{\bI}{\mathbb{I}}
\newcommand{\J}{\mathbb{J}}
\newcommand{\bP}{\mathbb{P}}
\newcommand{\Q}{\mathbb{Q}}
\newcommand{\Z}{\mathbb{Z}}
\newcommand{\R}{\mathbb{R}}
\newcommand{\C}{\mathbb{C}}
\newcommand{\K}{\mathbb{K}}
\newcommand{\T}{\mathbb{T}}
\newcommand{\G}{\mathbb{G}}
\newcommand{\cG}{\mathcal{G}}
\newcommand{\cB}{\mathcal{B}}
\newcommand{\cC}{\mathcal{C}}
\newcommand{\cF}{\mathcal{F}}
\newcommand{\cH}{\mathcal{H}}
\newcommand{\cK}{\mathcal{K}}
\newcommand{\cL}{\mathcal{L}}
\newcommand{\cM}{\mathcal{M}}
\newcommand{\cV}{\mathcal{V}}
\newcommand{\fM}{\mathfrak{M}}
\newcommand{\Lip}[1]{\|{#1}\|_{Lip}}
\newcommand{\inp}[2]{\langle#1,#2\rangle}
\newcommand{\intd}[2]{\int_{#1} #2 \mathrm{d} \mu(t)}
\newcommand{\sset}{\Omega_{n,j}}
\newcommand{\EE}[2]{\mathbb{E}\left[#1 | #2\right]}
\newcommand{\norm}[1]{\left\| #1 \right\|}
\newcommand{\abs}[1]{\left| #1 \right|}
\newcommand{\br}[1]{\left(#1 \right)}

    \begin{frontmatter}
        \title{Uniform Approximation for Operator Learning on Compact Sets}

        \author[nankai]{Rui Liu}
        \ead{ruiliu@nankai.edu.cn}
        \author[nankai]{Jie Shen}
        \ead{1710064@mail.nankai.edu.cn}
        \affiliation[nankai]{organization={School of Mathematical Sciences and LPMC, Nankai University},
            city={Tianjin},
            postcode={300071},
            country={P.R. China}}

        \begin{abstract}
            We study uniform approximation for operator learning on compact subsets of Banach spaces.
            We bound the recovery dimensions and the number of network parameters in terms of compact-set complexity, the target gauge, and recovery stability.
            For bounded linear operators with a separable domain, we characterize gauge approximability by coordinate decompositions with a common vanishing gauge and arbitrarily small additive errors.
            Their partial sums give a fixed family of nonlinear encoder--decoder pairs.
            We prove that nonlinear recovery with a common vanishing gauge and a rank bound depending only on the model dimension is equivalent to the uniform approximation property (UAP).
            We also obtain linear recovery estimates with the same rank bound and an operator version for the linear factors of nonlinear targets.
            For Barron operators taking values in a Banach space \(Y\) of type \(p\in(1,2]\) and admitting a Bochner-integrable density, sampling gives the rate \(N^{-(1-1/p)}\) for the expected \(L^p(\lambda;Y)\) error, where \(\lambda\) is a probability measure on the input compact set.
            For fixed finite-dimensional factors and a class of H\"older maps, we obtain matching upper and lower parameter bounds at fixed sufficiently large depth.
            Classical sampling and truncation examples provide explicit recovery maps, and a nonlinear integral operator illustrates the parameter bounds.
        \end{abstract}
    \end{frontmatter}

    \section{Introduction}\label{sec:introduction}

A neural network for operator learning acts on finitely many coordinates, while its inputs and outputs belong to Banach spaces.
For a continuous map \(T:X\to Y\), we consider models of the form \(\widehat T=R_q^Y N C_m^X\) as in \cite[Section~2.3]{LanthalerPCANet2023}, where \(C_m^X:X\to\R^m\) is an encoder, \(N:\R^m\to\R^q\) is a neural network, and \(R_q^Y:\R^q\to Y\) is a decoder.
The approximation error is measured on a compact set \(K\Subset X\), but the encoder and decoder are defined on their whole domains.
There are four different sources of difficulty in obtaining an accurate model with controlled dimensions and network size.
\begin{enumerate}
    \item \textbf{Set complexity.}
    The input and output compact sets may require many coordinates for accurate reconstruction.
    This is a question about the finite-dimensional approximation of the sets involved.
    \item \textbf{Operator gauge.}
    The operator gauge may be large at small scales or fail to vanish at zero.
    The resulting estimate may require more accurate reconstruction or fail to ensure convergence.
    \item \textbf{Uniform stable recovery.}
    A finite-dimensional model need not admit a stable recovery with a rank bound determined by the complexity of the compact set.
    Small reconstruction error on a special compact set does not provide a common stability bound for all compact sets with comparable complexity.
    \item \textbf{Network complexity.}
    Even after suitable reductions have been selected, the reduced nonlinear map may require a large network.
\end{enumerate}

These difficulties may occur together.
They lead us to ask when the complexity of the input and output compact sets can be converted into uniform bounds for stable recovery and neural network approximation.
Here uniformity concerns the complexity of compact sets: the recovery maps may depend on the compact set, but their stability and rank bounds should follow from common complexity bounds.

Neural approximation of nonlinear operators goes back to Chen and Chen \cite{ChenChen1995}.
Finite-dimensional input and output representations are used in DeepONets, PCA-based models, and spectral neural operators \cite{LuJinPangZhangKarniadakis2021,LanthalerPCANet2023,LiEtAl2021FNO}; see also \cite{KovachkiEtAl2023,LanthalerMishraKarniadakis2022}.
Once these representations have been fixed, a finite-dimensional universal approximation theorem gives qualitative approximation of the reduced map.
Also, Lanthaler, Li, and Stuart \cite{LanthalerLiStuart2025} prove universality for averaging neural operators, and G\"odeke and Fernsel \cite{GodekeFernsel2025} consider the requirement that one approximating sequence work on every compact set.
Zappala \cite{Zappala2024Projection} uses Leray--Schauder maps and polynomial projection models in \(L^p\), while De Feo \cite{DeFeo2026} studies approximation of differentiable operators and their derivatives on Banach spaces.
Quantitative estimates must also account for the dimensions, the stability of reconstruction, and the effect of the decoder on the network error.
Our purpose is to relate these quantities to compact-set geometry and approximation properties of Banach spaces and operators. 

The first difficulty requires a measure of how much finite-dimensional information is needed at a given accuracy.
We use widths, covering numbers, and the corresponding dimension requirements.
Widths describe approximation by low-dimensional spaces, whereas covering numbers describe approximation by finite sets.
These quantities provide the complexity bounds used throughout the paper, beginning in Section~\ref{sec:Compact sets-widths-reductions}.
Space selection and recovery from measurements also enter reduced-basis greedy algorithms and data assimilation in Banach spaces \cite{DeVorePetrovaWojtaszczyk2013,DeVorePetrovaWojtaszczyk2017}.
Reduced bases are used in neural-network bounds for parametric PDE solution maps \cite{KutyniokPetersenRaslanSchneider2022}.
Nonlinear model reduction is studied through stable manifold widths \cite{CohenDeVorePetrovaWojtaszczyk2022}, Lipschitz widths and entropy \cite{PetrovaWojtaszczyk2023}, and auto-encoder-based operator learning \cite{LiuDahalLaiLiao2025}.
In our setting, the geometric approximation error must be accompanied by a bound for realizing the chosen model through maps defined on the ambient space.

The second difficulty links this geometric error to the target operator.
If the input recovery has error at most \(\delta\) and the variation of \(T\) is bounded by a gauge \(\omega\), then input replacement contributes at most \(\omega(\delta)\) to the output error.
Thus the same compact-set approximation can lead to different accuracy requirements for different operators.
For fixed coders, a local modulus on the compact set containing the original and reconstructed inputs is sufficient.
We use a global vanishing gauge when selecting the coders for a prescribed output error.
Error estimates for fixed encoders and decoders are given in \cite{LanthalerPCANet2023,LanthalerMishraKarniadakis2022,Mhaskar2023LocalOperators}.
Cheng, Fan, Song, and Zhou \cite{ChengFanSongZhou2026} give approximation and generalization bounds for Fr\'echet differentiable operators with prespecified encoders and decoders.

The third difficulty is the main approximation-theoretic question of the paper.
Every compact subset of a Banach space admits arbitrarily accurate finite-dimensional Lipschitz reconstructions, but the Lipschitz constants may depend on the compact set and the accuracy.
Requiring one common bound leads to the Lipschitz bounded approximation property, which is equivalent to the bounded approximation property (BAP) by Godefroy and Kalton \cite{GodefroyKalton2003}.
Approximation with a common gauge was studied by Kalton \cite{Kalton2012UniformStructure}; see also Godefroy \cite{Godefroy2020}.
Liu, Shen, and Zheng \cite{LiuShenZheng2023} characterize operators with the Lipschitz BAP by Lipschitz decompositions.
For bounded linear operators with separable domain, Section~\ref{sec:nonlinear-encoder} characterizes the Lipschitz BAP and gauge approximability by coordinate decompositions.
The coordinate functions and reconstruction vectors can be fixed, giving encoder--decoder pairs whose compositions converge uniformly on every compact set.
However, these decompositions do not give a rank bound uniform over compact sets of a given complexity.

The connection between finite-rank approximation and coordinate representations also appears in Pe\l czy\'nski's complemented-basis theorem \cite{Pelczynski1971} and the Schauder-frame characterization of the BAP by Casazza, Han, and Larson \cite{CasazzaHanLarson1999}.
Beanland, Freeman, and Liu \cite{BeanlandFreemanLiu2015} study associated spaces of Schauder frames, while Han, Hu, Liu, and Liu \cite{HanHuLiuLiu2026} give a comprehensive account of Banach-space frame theory.
Bao, Han, Liu, and Shen \cite{BaoHanLiuShen2025} prove that a Lipschitz framing of a separable Banach space admits a linear framing with the same synthesis operator, and Bao, Liu, and Shen \cite{BaoLiuShen2027} characterize the unconditional BAP through Schauder-frame dilations.

With bounded linear encoders and decoders, the network is the only nonlinear component of the model.
Unrestricted nonlinear coders may themselves carry part of the nonlinear approximation, and the complexity of these coders is not included in the parameter count of the middle network.
For linear reconstruction, fixing a model subspace gives an error bound in terms of the distance to that subspace and the operator norm.
We therefore seek such recoveries with a common norm bound and rank controlled by the model dimension.
These bounds are supplied by the uniform approximation property (UAP), studied by Pe\l czy\'nski and Rosenthal \cite{PelczynskiRosenthal1975}.
We introduce the Gauge-UAP, which requires arbitrarily accurate nonlinear recovery on balls of finite-dimensional subspaces, with a common vanishing gauge and a rank bound depending only on the subspace dimension.
We prove that the Gauge-UAP is equivalent to the UAP.
In fact, Theorem~\ref{thm:common-gauge-recovery-UAP} gives linear recovery with the same rank bound and the following width estimates for every compact set:
\[ d_{f(n)}(K)\leq\mathcal E_{f(n)}^\omega(K)\leq\operatorname{slc}_{f(n)}^L(K;X)\leq(1+L)d_n(K). \]
Here \(d_n\), \(\mathcal E_m^\omega\), and \(\operatorname{slc}_m^L\) denote the Kolmogorov width, the nonlinear recovery error, and the stable linear recovery error, respectively. \(f\) is the uniform rank bound and \(L\) is the asymptotic slope of the gauge \(\omega\).

We use these recovery estimates to bound the error of \(\widehat T\).
Let \(\delta\) bound the error of the input recovery \(R_m^XC_m^X\) on \(K\), and let \(\varepsilon_{\rm out}\) bound the error of the output recovery \(R_q^YC_q^Y\) on \(T(K)\).
Let \(\eta\) bound the network error after decoding.
For a global gauge \(\omega_T\), the error decomposition in Section~\ref{subsec:operator-to-functional-reduction} gives
\[ \sup_{x\in K}\norm{T(x)-\widehat T(x)}_Y\leq\norm{R_q^YC_q^Y}\,\omega_T(\delta)+\varepsilon_{\rm out}+\eta. \]
The gauge determines the input accuracy, widths and stable recovery give norm and rank bounds, and the coordinate maps determine the network accuracy.
Section~\ref{sec:unit-ball-UAP-coders} and the appendices give measurement and reconstruction maps for point sampling, conditional expectations, and wavelet and Gabor truncations, together with their norm and approximation bounds.

For a target \(T_F=LFS\), where \(S,L\) are bounded linear maps and \(F\) is nonlinear, the input enters through \(S\) and the output is obtained by applying \(L\).
We therefore estimate finite-rank approximations of these two factors.
A bounded linear operator may admit such estimates even when its domain and codomain fail the UAP, as shown by the infinite-rank projection in Example~\ref{ex:operator-uap-without-space-uap}.
The operator rank functions defined in Section~\ref{sec:UAP} give the corresponding recovery bounds for \(S\) and \(L\).
We also obtain the corresponding Gauge-UAP characterization for bounded linear operators; see Corollary~\ref{cor:operator-gauge-UAP}(2).
For operator classes, a common gauge bound and compactness of \(\overline{\bigcup_{F\in\mathcal F}F(SK)}\) give common reductions whose output rank is independent of the input recovery.

The fourth difficulty remains after stable reductions have been obtained.
The required number of network parameters still depends on the dimension and the approximation properties of the reduced map.
For targets with a Barron representation, we obtain a sampling upper bound.
Following the sampling approach underlying Barron approximation \cite{Barron1993}, we select input functionals and output vectors directly from a representation of the operator with a Bochner-integrable density.
We describe the space of these operators as a Banach quotient of a nuclear operator space.
For scalar functions, finite restrictions determine the representation norm.
The sampling construction gives finite encoder--decoder realizations, with the representation mass entering the bound on the number of hidden neurons.
Yang and Shi \cite{YangShi2026} study variation spaces for operators between Hilbert spaces and obtain a dimension-independent finite-width term in Bochner norm.
Our sampling estimate applies to output Banach spaces of type \(p\in(1,2]\) and gives the rate \(N^{-(1-1/p)}\) in \(L^p(\lambda;Y)\), for a probability measure \(\lambda\) on the input compact set.
With bounded representation mass, input radius, and type constant, this term has no explicit dependence on a coordinate dimension.
The uniform error bound also contains a term involving the covering numbers of the input compact set.
This argument supplies a structured upper bound without first reconstructing the input compact set.

We also estimate the number of network parameters required for uniform approximation of operator classes.
Lanthaler \cite{Lanthaler2024Information} gives information-theoretic lower bounds for Lipschitz operators with model size measured in encoded bits.
Lanthaler and Stuart \cite{LanthalerStuart2023} establish a parametric complexity barrier for general Lipschitz and \(C^r\) operators, together with positive results for structured Hamilton--Jacobi solution operators.
In Section~\ref{sec:compact-factorization-lower-bound-gauges}, we count trainable real parameter slots in one fixed ReLU graph for the target class.
For fixed finite-dimensional factors and a class of H\"older maps, the upper and lower parameter bounds have the same order at any fixed sufficiently large depth.
The integral-operator example in Appendix~\ref{app:detailed-operator-learning-examples} gives explicit recovery and parameter bounds under assumptions on the linear factors.

The paper is organized as follows.
Section~\ref{sec:Compact sets-widths-reductions} introduces the widths and entropy of compact sets.
Section~\ref{sec:nonlinear-encoder} studies nonlinear approximation with a common gauge.
Section~\ref{sec:UAP} studies uniform stable recovery and its relation to the UAP.
Section~\ref{sec:unit-ball-UAP-coders} gives encoders and recovery maps from classical sampling and truncation methods.
Section~\ref{sec:operator-learning-nn-complexity} proves error estimates and parameter bounds for neural approximation of operators.
Section~\ref{sec:compact-factorization-lower-bound-gauges} gives common reductions and parameter lower bounds for operator classes, including matching bounds for H\"older maps with fixed finite-dimensional factors.
The appendices contain detailed constructions and proofs, together with a nonlinear integral-operator example.

    \section{Widths and entropy of compact sets}\label{sec:Compact sets-widths-reductions}

Without loss of generality, we work with real Banach spaces.
For complex subspaces and complex-linear maps, finite dimensions, codimensions, and ranks are multiplied by the fixed factor two.
Denote \(\N=\{0,1,2,\ldots\}\) and \(\N_+=\{1,2,3,\ldots\}\).

We first measure how much finite-dimensional information is needed to approximate the input and output compact sets.
Widths measure the error of a model subspace, while entropy measures the number of representatives required at a prescribed accuracy.
These quantities provide the geometric inputs for the recovery and network estimates below.
Example~\ref{ex:compact-cube-complexity} shows how the geometry of a compact set affects neural parameter lower bounds.

\begin{definition}\label{def:compact-profiles}
    Let \(K\Subset X\) and let \(\varepsilon>0\).
    \begin{enumerate}
        \item The covering number of \(K\) at \(\varepsilon\) is
        \[ N(K,\varepsilon B_X) = \min\left\{N\in\N_+:K\subseteq \bigcup_{j=1}^N(x_j+\varepsilon B_X) \text{ for some }x_1,\dots,x_N\in X\right\}. \]
        \item The Kolmogorov entropy number, indexed by \(m\in\N_+\), is
        \[ e_m(K;X) = \inf\{\varepsilon>0:N(K,\varepsilon B_X)\leq2^{m-1}\}. \]
        \item The Kolmogorov width, indexed by \(n\in\N\), is
        \[
            d_n(K;X) = \inf_{\substack{E\subseteq X\\
                    \dim E\leq n}} \sup_{x\in K}\dist(x,E).
        \]
        \item If \(K\) is symmetric and \(0\in K\), the Gelfand width, indexed by \(m\in\N_+\), is
        \[
            c_m(K;X) = \inf_{\substack{M\subseteq X\\
                    \operatorname{codim}M<m}} \sup_{x\in K\cap M}\norm{x}.
        \]
        For a general compact set \(K\), the same width can be applied to the difference set \(K-K\).
    \end{enumerate}
\end{definition}

The Gelfand width \(c_m\) measures the part of \(K\) that we cannot distinguish by fewer than \(m\) linear measurements.
When the space \(X\) is clear, we write \(N(K,\varepsilon)\), \(e_m(K)\), \(d_n(K)\), and \(c_m(K)\).
For a relatively compact set, we evaluate these quantities on its norm closure.
In particular, we use this convention for \(T(B_Y)\) when \(T\) is a compact operator.

\Needspace{6\baselineskip}

We record the following standard width--entropy estimates for later use.

\begin{proposition}\label{prop:entropy-width-elementary}
    Let \(K\Subset X\).
    \begin{enumerate}
        \item If \(K\subseteq \bigcup_{j=1}^{N}(x_j+\varepsilon B_X)\), then \(d_N(K)\leq\varepsilon\).
        In particular, \(d_{2^{m-1}}(K)\leq e_m(K).\)
        \item Suppose \(K\subseteq RB_X\), \(d_n(K)<\eta\), and \(\delta>0\).
        Then
        \[ N(K,\eta+\delta) \leq \left(1+\frac{2(R+\eta)}{\delta}\right)^n. \]
        Hence if \((m-1)\log 2\geq n\log\left(1+2(R+\eta)/\delta\right),\) then we have \(e_m(K)\leq \eta+\delta\).
    \end{enumerate}
\end{proposition}

\begin{proof}
    For the first assertion, let \(E=\Span\{x_1,\dots,x_N\}\).
    Then we have \(\dim E\leq N\), and \(\dist(x,E) \leq \varepsilon\) for all \(x \in K\).
    By taking \(N\leq2^{m-1}\) and passing to the infimum, we have \(d_{2^{m-1}}(K)\leq e_m(K)\).

    For the second assertion, we choose \(E\subseteq X\) with \(\dim E\leq n\) and \(\sup_{x\in K}\dist(x,E)<\eta\).
    For each \(x\in K\), choose \(y_x\in E\) such that \(\norm{x-y_x}<\eta\).
    Since \(K\subseteq RB_X\), we have \(\norm{y_x}<R+\eta\).
    If \(\{z_1,\ldots,z_M\}\) is a maximal \(\delta\)-separated subset of \((R+\eta)B_E\), then it is a \(\delta\)-net.
    The balls \(z_j+(\delta/2) B_E\) are pairwise disjoint and contained in \((R+\eta+\delta/2) B_E\).
    By comparing their volumes, we have
    \[ M\leq \left(\frac{R+\eta+\delta/2}{\delta/2}\right)^{\dim E} \leq \left(1+\frac{2(R+\eta)}{\delta}\right)^n . \]
    Thus we can cover \(K\) by the same number of balls of radius \(\eta+\delta\).
    Then we obtain the entropy estimate.
\end{proof}

\begin{remark}\label{rem:volume-log-loss}
    In particular, assume that \(K\subseteq RB_X\) and \(d_n(K)\lesssim n^{-s}\) for some \(s>0\).
    By taking \(\eta\) and \(\delta\) of order \(n^{-s}\) in Proposition \ref{prop:entropy-width-elementary}, we obtain the estimate \(e_m(K)\lesssim \left(\log m/m\right)^s .\)
    Thus the single-scale volume argument incurs a logarithmic loss.
\end{remark}

By the definitions and Proposition~\ref{prop:entropy-width-elementary}, we immediately obtain the following estimates.

\begin{proposition}\label{prop:polynomial-rate-inversion}
    Let \(K\Subset X\), let \(s, \varepsilon>0\), and let \(A,a>0\).
    \begin{enumerate}
        \item If \(d_j(K)\leq Aj^{-s}\) for every \(j\in\N_+\), then the integer \(n_\varepsilon=\lceil(A/\varepsilon)^{1/s}\rceil\) satisfies \(d_{n_\varepsilon}(K)\leq\varepsilon\).
        \item If \(e_j(K)\leq Aj^{-s}\) for every \(j\in\N_+\), then \(h_\varepsilon=1+\lceil(A/\varepsilon)^{1/s}\rceil\) satisfies \(\log N(K,\varepsilon)\leq(h_\varepsilon-1)\log2\).
        \item Suppose that \(d_j(K)\geq aj^{-s}\) for every \(j\in\N_+\).
        If \(\ell_\varepsilon=\lfloor(a/(2\varepsilon))^{1/s}\rfloor\geq1\), then \(N(K,\varepsilon)>\ell_\varepsilon\).
        \item Suppose that \(e_j(K)\geq aj^{-s}\) for every \(j\in\N_+\).
        If \(\ell_\varepsilon\geq1\) is as above, then \(N(K,\varepsilon)>2^{\ell_\varepsilon-1}\).
    \end{enumerate}
\end{proposition}

If \(K\) lies in a Hilbert space and \(0<\delta<1\), a Johnson--Lindenstrauss map sends an \(\eta\)-net into \(\R^m\) and preserves all pairwise distances within relative error \(\delta\), where \(m\lesssim\delta^{-2}\log N(K,\eta)\) \cite{JohnsonLindenstrauss1984,Naor2018}.
Thus the chosen net admits finite-dimensional coordinates with the stated distortion and dimension bounds.

For compact sets given as operator images, the same questions can be expressed through operator numbers.
This description allows us to use approximation estimates for the linear factors in an operator-learning problem.

\begin{definition}\label{def:operator-image-realization}
    Let \(T\in\cB(Y,X)\).
    The approximation numbers of \(T\), indexed by \(m\in\N_+\), are
    \[ a_m(T)=\inf\{\norm{T-A}:A\in\cF(Y,X),\ \operatorname{rank}A<m\}. \]
    The Kolmogorov numbers of \(T\), indexed by \(m\in\N\), are
    \[
        d_m(T) = \inf_{\substack{E\subseteq X\\\dim E\leq m}} \sup_{y\in B_Y}\dist(Ty,E)=d_m(\overline{T(B_Y)};X).
    \]
    The Gelfand numbers of \(T\), indexed by \(m\in\N_+\), are
    \[
        c_m(T)= \inf_{\substack{M\subseteq Y\\ \operatorname{codim}M<m}} \norm{T|_M}.
    \]
    If \(T\) is compact,  define its entropy numbers by \(e_m(T)=e_m(\overline{T(B_Y)})\).
\end{definition}

\begin{remark}\label{rem:subspace-versus-realizable}
    Every finite-dimensional subspace is complemented, so we obtain a finite-rank approximation.
    However, \(d_n(K)\) does not control the norm of the projection.
    In contrast, if \(a_m(T)\leq\varepsilon\), then for every \(\zeta>0\) there is a finite-rank operator \(A:Y\to X\) such that \(\operatorname{rank}A<m\) and \(\norm{T-A}<\varepsilon+\zeta\).
\end{remark}

\begin{proposition}\label{prop:operator-image-realizable}
    Let \(T\in\cB(Y,X)\) be compact.
    For every \(m\in\N_+\),
    \[ d_m(T)\leq a_m(T) \qquad\text{and}\qquad c_m(T)\leq a_m(T). \]
\end{proposition}

\begin{proof}
    Let \(A\in\cF(Y,X)\) with \(\operatorname{rank}A<m\).
    Since \(A(B_Y)\) is contained in a subspace of dimension at most \(m-1\), we have \(\sup_{y\in B_Y}\dist(Ty,A(Y)) \leq \sup_{y\in B_Y}\norm{Ty-Ay} \leq \norm{T-A}.\)
    By taking the infimum over all such \(A\), we have \(d_m(T)\leq a_m(T)\).
    For the Gelfand number, take \(M=\ker A\).
    Then \(\operatorname{codim}M\leq\operatorname{rank}A<m\), and \(\norm{T|_M}\leq\norm{T-A}\).
    By again taking the infimum over \(A\), we have \(c_m(T)\leq a_m(T)\).
\end{proof}

Carl's inequality relates operator widths to entropy numbers \cite{CarlStephani1990,HinrichsKolleckVybiral2015}.

\begin{lemma}\label{lem:carl-bridge}
    Let \(T\in\cB(Y,X)\) be compact.
    For every \(\alpha>0\) there is a constant \(\gamma_\alpha>0\), depending only on \(\alpha\), such that for every \(n\in\N_+\),
    \[ \sup_{1\leq k\leq n} k^\alpha e_k(T) \leq \gamma_\alpha \max\left\{\norm{T},\sup_{1\leq k\leq n} k^\alpha s_k(T)\right\}, \]
    where \(s_k(T)\) may be taken to be one of the \(a_k(T)\), \(c_k(T)\), or \(d_k(T)\).
\end{lemma}

By Lemma~\ref{lem:carl-bridge} with \(\alpha=s\), we immediately obtain the following polynomial decay estimate, without the logarithmic loss in the volume argument.

\begin{corollary}\label{cor:carl-polynomial}
    Let \(T\in\cB(Y,X)\) be compact, let \(s>0\), and let \(s_k(T)\) be one of \(a_k(T)\), \(c_k(T)\), or \(d_k(T)\).
    Suppose that \(s_k(T)\leq Ak^{-s}\) for \(k\geq1\).
    Then \(e_k(T)\leq\gamma_s\max\{A,\norm{T}\}k^{-s}\) for \(k\geq1\).
\end{corollary}

The following example shows why the recovery dimension and the number of network parameters need separate estimates.
\begin{example}\label{ex:compact-cube-complexity}
    Let \(X=c_0\), let \((e_j)\) be its standard basis, fix \(\alpha>1\), and put
    \[ K=\left\{\sum_{j=1}^{\infty}j^{-\alpha}t_je_j:0\leq t_j\leq1\right\}\subseteq X. \]
    For \(n\in\N_+\), the coordinate truncation \(x\mapsto\sum_{j=1}^n x_je_j\) has norm one and rank \(n\), and
    \[ \sup_{x\in K}\norm{x-\sum_{j=1}^n x_je_j}_\infty=(n+1)^{-\alpha}. \]
    Since \(K\) is closed and its coordinate tails tend to zero uniformly, \(K\) is compact.
    Thus \(d_n(K;X)\leq(n+1)^{-\alpha}\), and stable linear recovery to accuracy \(\delta\) uses \(O(\delta^{-1/\alpha})\) coordinates.
    The standard coordinate functionals have norm one, so \(K\) satisfies the infinite-dimensional cube assumption of Lanthaler and Stuart \cite[Definition~2.3]{LanthalerStuart2023}.
    Fix \(r\in\N_+\) and \(\eta>0\).
    By \cite[Theorem~2.11]{LanthalerStuart2023}, there are a fixed functional \(T\in C^r(X;\R)\) and constants \(c,\varepsilon_0>0\) with the following property.
    Let \(m\in\N_+\), let \(C_m^X:X\to\R^m\) be any bounded linear encoder, and let \(N:\R^m\to\R\) be a ReLU network.
    For \(0<\varepsilon\leq\varepsilon_0\), the estimate \(\sup_{x\in K}\abs{T(x)-N(C_m^Xx)}\leq\varepsilon\) implies that \(N\) has at least \(\exp\!\left(c\varepsilon^{-1/[r(\alpha+1+\eta)]}\right)\) nonzero weights and biases.
    Here \(m\), \(C_m^X\), and \(N\) may depend on \(\varepsilon\).
    Thus, on this compact set, the recovery dimension has a polynomial upper bound and the number of network parameters has an exponential lower bound.
\end{example}

The estimates above determine candidate model dimensions and finite nets from compact-set complexity.
For a width-based model, we need to construct a recovery map on the whole space and control its stability and rank.

    \section{Nonlinear approximation and gauge decompositions}\label{sec:nonlinear-encoder}
Nonlinear reconstruction gives a first answer to the realization question: every compact set admits finite-dimensional Lipschitz approximation.
The difficulty is to keep the stability bound independent of the compact set and the required accuracy.
We therefore study approximants and coordinate decompositions with a common Lipschitz or gauge bound.

Let \(T:X\to Y\) be continuous and let \(K\Subset X\).
For a finite-dimensional space \(Z_m\), let \(\mathcal C_m:X\to Z_m\) and \(\mathcal R_m:Z_m\to X\) be an encoder and a decoder, respectively.
We use the gauge of \(T\) to bound the output error caused by input reconstruction.
The gauge of the target controls error propagation, while a common gauge for the reconstruction maps controls the stability of the reduction itself.

\begin{definition}
    A function \(\omega:[0,\infty) \to [0,\infty)\) is a gauge if it is nondecreasing, subadditive, and \(\omega(0)=0\).
    It is vanishing if \(\omega(t)\to0\) as \(t\downarrow0\).
\end{definition}

\begin{definition}
    Let \(M_1\) and \(M_2\) be metric spaces.
    If \(T:M_1\to M_2\) is a map, then its global modulus is
    \[ \omega_T(t)=\sup\{\dist_{M_2}(Tx,Ty):\dist_{M_1}(x,y)\leq t\}. \]
    For every compact set \(K\subseteq X\), the local modulus is
    \[ \omega_{T,K}(t) = \sup\{\norm{Tx-Ty}:x,y\in K,\ \norm{x-y}\leq t\}. \]
\end{definition}

\begin{example}
    A map \(T:X\to Y\) is Lipschitz if and only if \(\omega_T(t)\leq Lt\) for some \(L>0\) and all \(t\geq0\).
    It is uniformly continuous if and only if \(\omega_T(t)\to0\) as \(t\downarrow0\).
\end{example}

Suppose that \(\mathcal R_m\mathcal C_m|_K\) is continuous.
Then \(K_m=K\cup\mathcal R_m\mathcal C_m(K)\) is compact, \(\omega_{T,K_m}\) vanishes at zero, and \( \sup_{x\in K}\norm{Tx-T\mathcal R_m\mathcal{C}_mx} \leq \omega_{T,K_m}\!\left(\sup_{x\in K}\norm{x-\mathcal R_m\mathcal{C}_mx}\right)\).

In applications, we also need bounds on the stability and dimensions of the encoders and decoders \cite{BenyaminiLindenstrauss2000,GodefroyKalton2003,GodefroyOzawa2014}.

\begin{definition}\label{AP}
    Let \(X,Y\) be Banach spaces and \(T\in\mathcal B(X,Y)\).
    The operator \(T\) has the approximation property (AP) if, for every compact \(K\subseteq X\) and every \(\varepsilon>0\), there is a bounded finite-rank linear operator \(S:X\to Y\) such that \(\sup_{x\in K}\norm{Tx-Sx}<\varepsilon.\)
    It has the Lipschitz approximation property (Lip AP) if \(S\) can instead be chosen Lipschitz with finite-dimensional range.
\end{definition}

In the next lemma, we construct a Lipschitz reconstruction from a finite net; see also \cite{Godefroy2020}.

\begin{lemma}\label{lem:lip-ap}
    Let \(X\) be a Banach space.
    Then \(I_X\) has Lip AP.
\end{lemma}
\begin{proof}
    Fix a compact \(K\subseteq X\) and \(\varepsilon>0\), and put \(\rho=\min\{\varepsilon/2,1\}\).
    Choose a finite \(\rho/2\)-net \(\{x_i\}_{i=1}^n\subseteq K\), and define
    \[ u_0(x)=\min\{\rho,\dist(x,K)\},\qquad u_i(x)=\max\{\rho-\norm{x-x_i},0\}\quad(1\le i\le n). \]
    Every \(u_i\) is \(1\)-Lipschitz.
    We claim that \(g=\sum_{i=0}^n u_i\geq\rho/4\) on \(X\).
    If \(\dist(x,K)\geq\rho/4\), then we have \(u_0\geq\rho/4\).
    Otherwise, by the choice of the \(\rho/2\)-net, we have \(u_i(x)>\rho/4\) for some \(i\).
    Thus the claim holds.
    Since the \(u_i\) are bounded and Lipschitz, the functions \(\phi_i=u_i/g\), \(1\le i\le n\), are  Lipschitz.
    The map \(Sx=\sum_{i=1}^n\phi_i(x)x_i\) is Lipschitz and has finite-dimensional range.
    If \(x\in K\), then \(u_0(x)=0\), \(\sum_{i=1}^n\phi_i(x)=1\), and \(\phi_i(x)>0\) implies \(\norm{x-x_i}<\rho\).
    Therefore, we have \(\norm{x-Sx}<\rho<\varepsilon\).
    Thus \(I_X\) has Lip AP.
\end{proof}

By Lemma \ref{lem:lip-ap}, we can choose finite-rank Lipschitz approximants on each compact set in every Banach space.
Their centers, rank, and Lipschitz constant may depend on \(K\) and \(\varepsilon\), so we do not obtain uniform stability or a uniform gauge.
In Definition \ref{BAP}, we require pointwise approximation by finite-dimensional maps with a common norm or Lipschitz bound \cite{GodefroyKalton2003}.

\begin{definition}\label{BAP}
    Let \(X,Y\) be Banach spaces, \(\lambda>0\), and \(T\in\mathcal B(X,Y)\).
    The operator \(T\) has the \(\lambda\)-bounded approximation property (\(\lambda\)-BAP) if, for every \(\varepsilon>0\), there is a net \(\{S_\gamma\}_{\gamma\in\Gamma}\subseteq\mathcal B(X,Y)\) such that
    \[ \sup_{\gamma\in\Gamma}\norm{S_\gamma}\leq\lambda+\varepsilon,\qquad \dim S_\gamma(X)<\infty,\qquad S_\gamma x\longrightarrow Tx\quad(x\in X). \]
    It has the BAP if this holds for some \(\lambda\).

    The operator \(T\) has the \(\lambda\)-Lipschitz BAP if, for every \(\varepsilon>0\), there is a net of Lipschitz maps \(S_\gamma:X\to Y\) such that
    \[ \sup_{\gamma\in\Gamma}\norm{S_\gamma}_{\mathrm{Lip}}\leq\lambda+\varepsilon,\qquad \dim\Span S_\gamma(X)<\infty,\qquad S_\gamma x\longrightarrow Tx\quad(x\in X). \]
    It has the Lipschitz BAP if this holds for some \(\lambda\).
\end{definition}

\begin{remark}
    A Banach space \(X\) has the BAP precisely when \(I_X\) has the BAP.
    If \(X\) is separable, then by equicontinuity and a diagonal argument on a countable dense set, we can replace the approximating net in Definition \ref{BAP} by a sequence.
    By passing to the pointwise limit, we have \(\norm{T}_{\mathrm{Lip}}\leq\lambda\) whenever \(T\) has the \(\lambda\)-Lipschitz BAP.
\end{remark}

By the AP, \(T\) admits arbitrarily accurate linear encoder--decoder approximations on each compact set.
The BAP gives a uniform norm bound, and the Lipschitz BAP gives the corresponding Lipschitz bound.
For separable \(X\), the decomposition below fixes a sequence of coordinate functions and reconstruction vectors.
Truncation then gives a fixed family of encoder--decoder pairs while retaining this bound.

\begin{definition}
    Let \(X\) be separable, let \(Y\) be a Banach space, and let \(T\in\mathcal B(X,Y)\).
    The operator \(T\) is \(\lambda\)-Lipschitz decomposable if, for every \(\varepsilon>0\), there are \(y_n\in Y\) and \(f_n\in\operatorname{Lip}(X;\R)\) such that
    \[ Tx=\sum_{n=1}^{\infty}f_n(x)y_n\quad(x\in X),\qquad \sup_{N\geq1}\norm{\sum_{n=1}^N f_n(\cdot)y_n}_{\mathrm{Lip}}\leq\lambda+\varepsilon. \]
    It is decomposable if the \(f_n\) may be chosen in \(X^*\).
\end{definition}

    If \(T\) is \(\lambda\)-Lipschitz decomposable, then  \(\norm{T}_{\mathrm{Lip}}\leq\lambda\) trivially.

\begin{theorem}
    Let \(X\) be a separable Banach space, let \(Y\) be a Banach space, and let \(T\in\mathcal B(X,Y)\).
    Then \(T\) is \(\lambda\)-Lipschitz decomposable if and only if \(T\) has the \(\lambda\)-Lipschitz BAP.
\end{theorem}
\begin{proof}
    If \(T\) is \(\lambda\)-Lipschitz decomposable, then we can take its finite partial sums as the required approximating sequence.
    Conversely, assume that \(T\) has the \(\lambda\)-Lipschitz BAP and fix \(\varepsilon>0\).
    Since \(X\) is separable, we can choose Lipschitz maps \(S_n:X\to Y\) whose images lie in finite-dimensional subspaces such that
    \[ S_nx\longrightarrow Tx\quad(x\in X),\qquad \sup_n\norm{S_n}_{\mathrm{Lip}}\leq\lambda+\frac{\varepsilon}{2}. \]
    By replacing \(S_n\) by \(S_n-S_n(0)\), we may also assume that \(S_n(0)=0\).
    Put \(S_0=0\), \(A_n=S_n-S_{n-1}\), \(E_n=\Span A_n(X)\), and \(d_n=\dim E_n\).
    Choose an Auerbach basis \((z_j^n,z_j^{n*})_{j=1}^{d_n}\) of \(E_n\), and put \( C_n=\max_{0\leq j\leq d_n} \norm{\sum_{k=1}^j z_k^{n*}\otimes z_k^n}. \)
    Choose \(M_n\in\N_+\) so that
    \(\frac{C_n}{M_n}\leq \min\left\{\frac{\varepsilon}{4\lambda+2\varepsilon}, \frac{2^{-n}}{2\lambda+\varepsilon}\right\}.\)
    Define
    \[ y_i=z_j^n,\qquad f_i=\frac{z_j^{n*}\circ A_n}{M_n} \qquad\text{if}\qquad i=\sum_{k=1}^{n-1}M_kd_k+rd_n+j, \]
    where \(r=0,\ldots,M_n-1\) and \(j=1,\ldots,d_n\).
    For every \(x\in X\), we have \[\sum_{r=0}^{M_n-1}\sum_{k=1}^{d_n} \frac{1}{M_n} z_k^{n*}(A_nx)z_k^n=A_nx.\]
    For \(1\leq j\leq d_n\), by the choice of \(M_n\), we also have \[\norm{M_n^{-1}\sum_{k=1}^jz_k^{n*}(A_nx)z_k^n}\leq(C_n/M_n)\norm{A_nx}\leq2^{-n}\norm{x}.\]
    Since \(S_{n-1}x,S_nx\to Tx\), we obtain pointwise convergence of the full series to \(T\).

    By the choice of \(M_n\), for every \(i=\sum_{k=1}^{n-1}M_kd_k+rd_n+j\), we have
    \begin{align*}
        \norm{\sum_{\ell=1}^{i}f_\ell\otimes y_\ell}_{\mathrm{Lip}}&=\norm{\left(1-\frac{r}{M_n}\right)S_{n-1}+\frac{r}{M_n}S_n+\frac{1}{M_n}\left(\sum_{k=1}^jz_k^{n*}\otimes z_k^n\right)A_n}_{\mathrm{Lip}}\\
        &\leq\left(1-\frac{r}{M_n}\right)\norm{S_{n-1}}_{\mathrm{Lip}}+\frac{r}{M_n}\norm{S_n}_{\mathrm{Lip}}+\frac{C_n}{M_n}\norm{A_n}_{\mathrm{Lip}}\\
        &\leq\lambda+\frac{\varepsilon}{2}+\frac{C_n}{M_n}(2\lambda+\varepsilon )\leq\lambda+\varepsilon.
    \end{align*}
    Therefore, we obtain a \(\lambda\)-Lipschitz decomposition of \(T\).
\end{proof}

In Kalton's theory, one uses vanishing gauges and almost Lipschitz bounds in place of the linear modulus \(Lt\) \cite{Kalton2012UniformStructure}.
It remains open whether every separable Banach space is approximable in the sense defined below.

\begin{definition}
    Let \(X\) be a Banach space and let \(\omega\) be a vanishing gauge.
    We say that \(X\) is \(\omega\)-approximable if, for every compact set \(K\) and \(\varepsilon>0\), there is a uniformly continuous \(\phi\) such that \(\phi(X)\) is relatively compact, \(\sup_{x \in K} \|x-\phi(x)\| <\varepsilon\) and \(\omega_{\phi} \leq \omega\).
    \(X\) is approximable if \(X\) is \(\omega\)-approximable for some vanishing gauge \(\omega\).
\end{definition}

For the almost Lipschitz version, we replace the Lipschitz gauge \(Lt\) by \(Lt+m\), where \(m>0\) is arbitrary.

\begin{definition}
    Let \(L>0\) and \(m>0\).
    We say that \(X\) is \((L,m)\)-almost Lipschitz approximable if, for every compact set \(K\subseteq X\) and every \(\delta>0\), there exists a uniformly continuous map \(\phi:X\to X\) such that \(\phi(X)\) is relatively compact,
    \( \sup_{x\in K}\|x-\phi(x)\|<\delta\) and \( \omega_\phi(t)\le Lt+m\), for all \( t\ge0\).
    \(X\) is \(L\)-almost Lipschitz approximable if \(X\) is \((L,m)\)-almost Lipschitz approximable for every \(m>0\), and \(X\) is almost Lipschitz approximable if it is \(L\)-almost Lipschitz approximable for some \(L>0\).
\end{definition}
By Proposition~4.3 of \cite{Kalton2012UniformStructure}, we obtain the following sequential approximation result.
\begin{lemma}\label{lem:gauge-sequential-approximation}
    Let \(X\) be a separable Banach space and let \(\omega\) be a vanishing gauge.
    If \(X\) is \(\omega\)-approximable, then for every \(\varepsilon>0\) there exists an equi-uniformly continuous sequence of continuous maps \(\{S_n\}_{n=1}^\infty\) with relatively compact images contained in finite-dimensional subspaces of \(X\) such that \(\{S_n\}_{n=1}^\infty\) converges pointwise to the identity operator and \(\omega_{S_n} \leq \omega+\varepsilon\).
\end{lemma}

The following definition requires a common gauge bound for the partial sums.

\begin{definition}
    Let \(X\) be a separable Banach space.
    Let \(\varepsilon,L>0\), and let \(\omega\) be a vanishing gauge.
    A sequence \((f_n,x_n)_{n\geq1}\) is an \((\omega+\varepsilon)\)-frame if each \(f_n\) is a continuous scalar function, each \(x_n\in X\), and the series \(\sum_{n=1}^{\infty}f_n(x)x_n\) converges to \(x\) uniformly on every compact subset of \(X\).
    For every \(N\in\N_+\), the map \(\sum_{n=1}^{N}f_n\otimes x_n\) is uniformly continuous, its image is relatively compact, and \(\omega_{\sum_{n=1}^{N}f_n\otimes x_n}\leq\omega+\varepsilon\).

    The sequence is an \((L,\varepsilon)\)-almost Lipschitz frame if each \(f_n\) is continuous, each \(x_n\in X\), and the series \(\sum_{n=1}^{\infty}f_n(x)x_n\) converges to \(x\) uniformly on every compact subset of \(X\).
    For every \(N\in\N_+\), the map \(\sum_{n=1}^{N}f_n\otimes x_n\) has relatively compact image and satisfies \(\omega_{\sum_{n=1}^{N}f_n\otimes x_n}(t)\leq Lt+\varepsilon\) for all \(t\geq0\).
\end{definition}
If the target operator has a bounded linear factor, we need to consider the approximability of that factor.
In the operator form below, we choose approximants with finite-dimensional images in \(Y\) which converge to \(T:X\to Y\).
The approximation condition is imposed on \(T\) itself.
We consider uniform rank estimates for linear realizations separately in Section~\ref{sec:UAP}.

\begin{definition}\label{def:operator-approximability}
    Let \(X\) be a separable Banach space, \(Y\) be a Banach space, \(T \in \mathcal{B}(X,Y)\), and  \(\omega\) be a vanishing gauge.
    \(T\) is \(\omega\)-approximable if for every compact set \(K\subseteq X\), every \(\delta>0\), and every \(\eta>0\), there exists a continuous map \(S:X\to Y\) such that \(S(X)\) is relatively compact and contained in a finite-dimensional subspace of \(Y\), \(\sup_{x\in K}\norm{Tx-Sx}<\delta\), and \(\omega_S\leq\omega+\eta.\)

    We say that \(T\) is \((L,\eta)\)-almost Lipschitz approximable if the same assertion holds with \(\omega_S(t)\leq Lt+\eta.\)
    And \(T\) is \(L\)-almost Lipschitz approximable if it is \((L,\eta)\)-almost Lipschitz approximable for every \(\eta>0\).
\end{definition}

\begin{definition}\label{def:operator-decomposition}
    Let \(X\) be a separable Banach space, \(Y\) be a Banach space, \(T \in \mathcal{B}(X,Y)\), \(\eta,L>0\), and let \(\omega\) be a vanishing gauge.
    A sequence \(\{(f_n,y_n)\}_{n=1}^{\infty}\) is an operator \((\omega+\eta)\)-decomposition for \(T\) if each \(f_n:X\to\R\) is continuous, each \(y_n\in Y\), and the series \(\sum_{n=1}^{\infty}f_n(x)y_n\) converges to \(Tx\) uniformly on every compact subset of \(X\).
    For every \(N\in\N_+\), the map \(\sum_{n=1}^{N}f_n\otimes y_n\) has relatively compact image contained in a finite-dimensional subspace of \(Y\) and satisfies \(\omega_{\sum_{n=1}^{N}f_n\otimes y_n}\leq\omega+\eta\).

    The sequence is an operator \((L,\eta)\)-almost Lipschitz decomposition for \(T\) if each \(f_n:X\to\R\) is continuous, each \(y_n\in Y\), and the series \(\sum_{n=1}^{\infty}f_n(x)y_n\) converges to \(Tx\) uniformly on every compact subset of \(X\).
    For every \(N\in\N_+\), the map \(\sum_{n=1}^{N}f_n\otimes y_n\) has relatively compact image contained in a finite-dimensional subspace of \(Y\) and satisfies \(\omega_{\sum_{n=1}^{N}f_n\otimes y_n}(t)\leq Lt+\eta\) for all \(t\geq0\).
\end{definition}

The next theorem characterizes operator approximability by decompositions with finite-dimensional blocks in \(Y\).
\begin{theorem}\label{thm:operator-omega-decomposition}
    Let \(X\) be a separable Banach space, \(Y\) a Banach space, and \(T\in\mathcal B(X,Y)\).
    The following statements are equivalent.
    \begin{enumerate}
        \item \(T\) is \(\omega\)-approximable for some vanishing gauge \(\omega\).
        \item There exists \(L>0\) such that \(T\) is \(L\)-almost Lipschitz approximable.
        \item There exists a vanishing gauge \(\omega\) such that, for every \(\eta>0\), \(T\) has an operator \((\omega+\eta)\)-decomposition.
        \item There exists \(L>0\) such that, for every \(\eta>0\), \(T\) has an operator \((L,\eta)\)-almost Lipschitz decomposition.
    \end{enumerate}
\end{theorem}

\begin{proof}
    By taking the vanishing gauge \(\Omega(t)=Lt\) in Definition \ref{def:operator-approximability}, we obtain \((2)\Rightarrow(1)\).
    For the converse, we use Kalton's scaling argument \cite{Kalton2012UniformStructure}.

    Assume that \(T\) is \(\omega\)-approximable for some vanishing gauge \(\omega\).
    Since \(\omega\) is subadditive and nondecreasing, there exist constants \(L,C>0\) such that
    \( \omega(t)\le Lt+C\) for all \( t\ge 0\). 
    Let \(K\subseteq X\) be a compact set, let \(\delta>0\), and let \(m>0\).
    Choose \(\lambda>0\) so that \(C/\lambda<m/2\).
    Since \(\lambda K\) is compact and \(T\) is \(\omega\)-approximable, there exists a continuous map \(S:X\to Y\) such that \(S(X)\) is relatively compact and contained in a finite-dimensional subspace of \(Y\), \(\sup_{z\in \lambda K}\|Tz-Sz\|<\lambda\delta\), and \(\omega_S\le \omega+\lambda m/2.\)
    Define
    \( \phi(x)=\frac1\lambda S(\lambda x)\) for \(x\in X.\)
    Then \(\phi(X)\) is relatively compact and contained in a finite-dimensional subspace of \(Y\).
    Since \(T\) is linear, for \(x\in K\) we have \(\|Tx-\phi(x)\| = \frac1\lambda\|T(\lambda x)-S(\lambda x)\| <\delta.\)
    Moreover, for \(t\ge0\), we have
    \[ \omega_\phi(t) = \frac1\lambda \omega_S(\lambda t) \le \frac1\lambda\omega(\lambda t)+\frac{m}{2} \le Lt+\frac{C}{\lambda}+\frac{m}{2} < Lt+m. \]
    Thus \(T\) is \((L,m)\)-almost Lipschitz approximable.
    Since \(m>0\) was arbitrary, we conclude that \(T\) is \(L\)-almost Lipschitz approximable.
    Therefore, we obtain \((1)\Rightarrow(2)\).

    We next prove \((1)\Rightarrow(3)\).
    Assume that \(T\) is \(\omega\)-approximable.
    Fix \(\eta>0\), and let \(\{u_j:j\geq1\}\) be a dense subset of \(X\).
    By approximating the two-point compact set \(\{x,y\}\) and then letting both error parameters tend to zero, we obtain \(\omega_T\leq\omega\).
    Put \(\theta_n=\eta2^{-n-2}\).
    We apply the definition on \(\{0,u_1,\ldots,u_n\}\) with error \(1/(2n)\), and subtract the value at zero.
    Thus we may choose a continuous map \(S_n:X\to Y\), with relatively compact image in a finite-dimensional subspace of \(Y\), such that \(S_n(0)=0\),
    \[ \norm{S_nu_j-Tu_j}<\frac1n\quad(1\leq j\leq n), \qquad \omega_{S_n}\leq\omega+\theta_n. \]
    For \(x\in X\) and every \(\sigma>0\), choose \(u_j\) so that \(2\omega(\norm{x-u_j}) < \sigma\).
    Then for all \(n\geq j\), we have
    \[ \norm{S_nx-Tx} \leq \omega_{S_n}(\norm{x-u_j})+\norm{S_nu_j-Tu_j} +\omega_T(\norm{x-u_j}) \leq \sigma+\theta_n+\frac1n. \]
    Thus \(S_nx\to Tx\) for every \(x\in X\).
    For \(K\Subset X\) and \(r>0\), take a finite \(r\)-net \(z_1,\ldots,z_M\) in \(K\).
    Since \(\omega_T\leq\omega\), we have
    \[ \sup_{x\in K}\norm{S_nx-Tx}\leq2\omega(r)+\theta_n+\max_{1\leq j\leq M}\norm{S_nz_j-Tz_j}. \]
    By first letting \(n\to\infty\) and then \(r\to0\), we obtain uniform convergence on \(K\).

    The sequence \((S_n)\) is now fixed on \(X\).
    We next write each difference \(S_n-S_{n-1}\) in scalar coordinates and keep the gauge below \(\omega+\eta\) for every partial sum.
    Let \(S_0=0\), \(A_n=S_n-S_{n-1}\), and \(E_n=\Span\{A_n(X)\}\).
    Then each \(E_n\) is finite dimensional and \(A_n(X)\) is relatively compact.
    Put \(d_n=\dim E_n\).
    It is enough to assume \(d_n>0\). Choose an Auerbach basis \(\{y_1^n,y_2^n,\dots,y_{d_n}^n\}\) of \(E_n\) with biorthogonal functionals \(\{y_1^{n*},y_2^{n*},\dots,y_{d_n}^{n*}\}\), and put \(C_n=\max_{1\leq j\leq d_n}\norm{\sum_{k=1}^j y_k^{n*}\otimes y_k^n}\).
    Since \(A_n(X)\) is bounded, let
    \( B_n=\sup_{x,y\in X}\norm{A_n(x)-A_n(y)}<\infty . \)
    Choose \(M_n\in\N_+\) such that
    \(\frac{C_nB_n}{M_n}\leq\min\{\eta/2,2^{-n}\}.\)
    Define
    \[ y_i=y_j^n \quad\text{ and } \quad f_i=\frac{y_j^{n*}\circ A_n}{M_n}\qquad \text{ if } \quad i=\sum_{k=1}^{n-1}M_kd_k+rd_n+j, \]
    where \(r=0,\dots,M_n-1\) and \(j=1,\dots,d_n\).

    Since \(\sum_{k=1}^{d_n}y_k^{n*}(A_nx)y_k^n=A_nx\), for \(x,y\in X\) and \(1\leq j\leq d_n\), we have
    \begin{align*}
        \norm{\frac{1}{M_n}\sum_{k=1}^{j}y_k^{n*}(A_nx-A_ny)y_k^n}
        \leq\frac{C_n}{M_n}\norm{A_nx-A_ny}
        \leq\frac{C_nB_n}{M_n}\leq\min\{\eta/2,2^{-n}\}.
    \end{align*}
    Thus the choice of \(M_n\) bounds the gauge of every unfinished coordinate sum; taking \(y=0\) also bounds its norm at each point.
    For \(s=\sum_{k=1}^{n-1}M_kd_k+rd_n+j\), the corresponding partial sum is
    \[ \sum_{i=1}^s f_i\otimes y_i= \frac r{M_n}S_n+ \left(1-\frac r{M_n}\right)S_{n-1} + \frac1{M_n}\sum_{k=1}^{j} y_k^{n*}\circ A_n\otimes y_k^n . \]
    Since \(A_n(0)=0\), the last term has norm at most \(C_nB_n/M_n\leq2^{-n}\) for every \(x\in X\).
    Thus, for every \(K\Subset X\), we have
    \[ \sup_{x\in K}\norm{\sum_{i=1}^s f_i(x)y_i-Tx}\leq\max_{k\in\{n-1,n\}}\sup_{x\in K}\norm{S_kx-Tx}+2^{-n}\longrightarrow0. \]
    Hence all partial sums converge to \(T\) uniformly on every compact subset of \(X\).
    Moreover, every finite partial sum has relatively compact image contained in a finite-dimensional subspace of \(Y\).
    By the choice of \(M_n\), for every \(t\geq0\), we have
    \begin{align*}
        \omega_{\sum_{i=1}^s f_i\otimes y_i}(t) &\leq \frac r{M_n}\omega_{S_n}(t) +\left(1-\frac r{M_n}\right)\omega_{S_{n-1}}(t) +\frac{C_n}{M_n}\omega_{A_n}(t) \leq \omega(t)+\frac{\eta}{8}+\frac{C_nB_n}{M_n} \leq \omega(t)+\eta .
    \end{align*}
    Hence \(\{(f_n,y_n)\}_{n=1}^{\infty}\) is an operator \((\omega+\eta)\)-decomposition for \(T\).  For the fixed \(\eta\), all coordinate functions \(f_i\) and vectors \(y_i\) have been chosen independently of any later compact set.
        Taking their first \(s\) terms therefore gives a fixed family of encoder--decoder pairs whose reconstruction maps have gauge at most \(\omega+\eta\).

    To prove \((3)\Rightarrow(1)\), fix a compact set \(K\subseteq X\) and \(\delta,\eta>0\).
    By (3), choose an operator \((\omega+\eta)\)-decomposition with partial sums \(P_N\).
    Since \(P_N\to T\) uniformly on \(K\), we can choose \(N\) such that \(\sup_{x\in K}\norm{P_Nx-Tx}<\delta\).
    The map \(P_N\) has relatively compact finite-dimensional image and satisfies \(\omega_{P_N}\leq\omega+\eta\), so it is the required approximant.

    It remains to prove \((2)\Longleftrightarrow(4)\).
    For \((2)\Rightarrow(4)\), fix \(\eta>0\), put \(\theta_n=\eta2^{-n-2}\), and choose an \((L,\theta_n)\)-almost Lipschitz approximant to \(T\) on \(\{0,u_1,\ldots,u_n\}\) with error less than \(1/(2n)\).
    After subtracting its value at zero, the resulting maps satisfy \(S_nx\to Tx\) and \(\omega_{S_n}(t)\leq Lt+\theta_n\).
    In the preceding coordinate construction, choose \(M_n\) as above.
    For the adjacent block endpoints, we have the bound \(Lt+\eta/8\).
    For the unfinished coordinate term, we have the bound \(\eta/2\).
    By the same finite-net and partial-sum estimates with \(\omega(t)=Lt\), we obtain an operator \((L,\eta)\)-almost Lipschitz decomposition.
    To prove \((4)\Rightarrow(2)\), fix a compact set \(K\subseteq X\) and \(\delta,\eta>0\).
    By (4), choose an operator \((L,\eta)\)-almost Lipschitz decomposition with partial sums \(P_N\).
    Since \(P_N\to T\) uniformly on \(K\), we can choose \(N\) such that \(\sup_{x\in K}\norm{P_Nx-Tx}<\delta\).
    The map \(P_N\) has relatively compact finite-dimensional image and satisfies \(\omega_{P_N}(t)\leq Lt+\eta\), so it is the required approximant.
\end{proof}

For the identity operator, we obtain the following frame characterization.

\begin{corollary}\label{thm:gauge-frame-characterization}
    Let \(X\) be a separable Banach space.
    The following statements are equivalent.
    \begin{enumerate}
        \item \(X\) is \(\omega\)-approximable for some vanishing gauge \(\omega\).
        \item There exists \(L>0\) such that, for every \(m>0\), \(X\) is \((L,m)\)-almost Lipschitz approximable.
        \item There exists a vanishing gauge \(\omega\) such that, for every \(\varepsilon>0\), \(X\) has an \((\omega+\varepsilon)\)-frame.
        \item There exists \(L>0\) such that, for every \(\eta>0\), \(X\) has an \((L,\eta)\)-almost Lipschitz frame.
    \end{enumerate}
\end{corollary}
\begin{proof}
    By \cite[Theorem~4.4]{Kalton2012UniformStructure}, we have \((1)\Longleftrightarrow(2)\).
    By Lemma~\ref{lem:gauge-sequential-approximation} and the coordinate construction in the proof of Theorem~\ref{thm:operator-omega-decomposition}, applied to \(I_X\), we obtain \((1)\Rightarrow(3)\), with uniformly continuous partial sums.
    Since the sequence in the lemma is equi-uniformly continuous and converges pointwise, it converges uniformly on compact sets, and the same holds for the partial sums by the estimate in the theorem.
    Theorem~\ref{thm:operator-omega-decomposition} also gives \((3)\Rightarrow(4)\).
    And under \((4)\), continuous almost Lipschitz approximants to \(I_X\) have relatively compact images.
    The argument for \((ii)\Rightarrow(v)\) in \cite[Theorem~4.4]{Kalton2012UniformStructure} uses only approximation on finite sets, relative compactness, and the almost Lipschitz bounds, so it applies to these maps.
    The implication \((v)\Rightarrow(ii)\) in \cite[Theorem~4.4]{Kalton2012UniformStructure} then gives uniformly continuous almost Lipschitz approximants. So \((4)\Rightarrow(2)\) follows.
\end{proof}

The decompositions above retain a uniform stability bound in finite coordinates.
They provide approximation on compact sets, but do not bound the number of coordinates in terms of the width or entropy of the compact set.

    \section[Stable recovery and the uniform approximation property]{Stable recovery and the uniform approximation property}\label{sec:UAP}

We now ask whether a model dimension chosen from widths determines the rank of a linear recovery with a fixed norm bound.
The recovery may depend on the compact set, but the rank bound must apply to every compact set with the same width bound.
We show that the UAP is equivalent to this uniform recovery estimate.
We first define the \((\omega,f)\)-UAP by nonlinear recovery with a fixed global vanishing gauge \(\omega\) and a fixed rank function \(f\).
The existence of such a pair \((\omega,f)\) gives the Gauge-UAP, which we prove to be equivalent to the UAP.
We also obtain linear recoveries with the same rank bound.

\subsection[Uniform recovery and the UAP]{Uniform recovery and the UAP}\label{subsec:widths-stable-linear-recovery}

\begin{definition}\label{def:stable-linear-coder-error}
    Let \(K\Subset X\), let \(m\in\N\), and let \(\Lambda\geq1\).
    Define
    \[
        \operatorname{slc}_{m}^{\Lambda}(K;X) = \inf_{\substack{A\in\cF(X)\\
                \operatorname{rank}A\leq m\\
                \norm{A}\leq\Lambda}} \sup_{x\in K}\norm{x-Ax}.
    \]
\end{definition}

The letters \(\operatorname{slc}\) stand for stable linear coder.
When the ambient space is clear, we suppress it and write \(d_m(K)\) for \(d_m(K;X)\) and \(\operatorname{slc}_{m}^{\Lambda}(K)\) for \(\operatorname{slc}_{m}^{\Lambda}(K;X)\).
We use the same convention for image widths below.
For \(\varepsilon>0\), we also define \(r_K(\varepsilon)=\min\{n\in\N:d_n(K)\leq\varepsilon\}\).
We always assume that a minimum or an infimum over an empty set is \(+\infty\).
The function \(r_K\) records the width scale.

By considering the range of an admissible coder, we obtain the following lower bound in terms of Kolmogorov widths.

\begin{proposition}\label{prop:slc-width-lower-bound}
    Let \(K\Subset X\), let \(m\in\N\), and let \(\Lambda\geq1\).
    Then \(d_m(K)\leq\operatorname{slc}_m^\Lambda(K).\)
\end{proposition}

\begin{proof}
    Let \(A\in\cF(X)\) be admissible.
    Since \(\dim\operatorname{ran}A\leq m\) and \(Ax\in\operatorname{ran}A\) for every \(x\in K\), we have \(\dist(x,\operatorname{ran}A)\leq\norm{x-Ax}\) for all \(x\in K\).
    Thus we obtain \(d_m(K) \leq \sup_{x\in K}\dist(x,\operatorname{ran}A) \leq \sup_{x\in K}\norm{x-Ax}.\)
    By taking the infimum over all admissible \(A\), we obtain the required inequality.
\end{proof}

If \(X\) is a Hilbert space, orthogonal projections give
\[ \inf_{\mathcal C_m,\mathcal R_m}\sup_{x\in K}\norm{x-\mathcal R_m\mathcal C_mx}=\operatorname{slc}_m^1(K;X)=d_m(K;X), \]
where the first infimum is over all encoders \(\mathcal C_m:X\to\R^m\) and linear decoders \(\mathcal R_m:\R^m\to X\).
Every such recovery takes values in a subspace of dimension at most \(m\), so its error is bounded below by the width.
Orthogonal projections give the reverse inequality by taking the infimum over subspaces of dimension at most \(m\).
Thus linear encoding gives the same optimal error with recovery norm at most one.

We next recall the UAP and define its rank function.

\begin{definition}\label{def:UAP-rank-function}
    We say that a Banach space \(X\) has the UAP if there are \(\Lambda\geq1\) and \(f:\N\to\N\) such that every finite-dimensional subspace \(E\subseteq X\) admits \(A\in\cF(X)\) satisfying
    \[ A|_E=I_E,\quad \norm{A}\leq\Lambda,\quad \operatorname{rank}A\leq f(\dim E). \]
    For fixed \(\Lambda>0\) and \(n\in\N\), define
    \[
        k_X(\Lambda,n) = \sup_{\substack{E\subseteq X\\
                \dim E\leq n}} \inf\left\{ \operatorname{rank}A: A\in\cF(X),\ \norm{A}\leq\Lambda,\ A|_E=I_E \right\}.
    \]
\end{definition}

Since admissible ranks are integers, every finite inner infimum is actually a minimum.
We may replace a UAP rank function by its monotone envelope.
Hence we may assume that \(f\) is nondecreasing and \(f(0)=0\).

Finite-dimensional perturbations are used in approximation-property arguments; see Casazza \cite{Casazza2001Approximation} and Bao, Liu, and Shen \cite[Lemma~2.3]{BaoLiuShen2027}.
We use the following elementary correction.

\begin{lemma}\label{lem:finite-dimensional-correction}
    Let \(T\in\cB(X,Y)\), let \(E\subseteq X\) be \(n\)-dimensional with \(n\geq1\), let \(m\in\N\), let \(C,\varepsilon\geq0\), and let \(S\in\cF(X,Y)\) satisfy
    \[ \norm{(T-S)|_E}\leq\varepsilon,\qquad \operatorname{rank}S\leq m,\qquad \norm{S}\leq C. \]
    Then there is \(A\in\cF(X,Y)\) such that
    \[ A|_E=T|_E,\qquad \operatorname{rank}A\leq m+n,\qquad \norm{A}\leq C+n\varepsilon. \]
\end{lemma}

\begin{proof}
    Let \(R=(T-S)|_E:E\to Y\), and choose an Auerbach basis \((e_i,e_i^*)_{i=1}^n\) for \(E\).
    By the Hahn--Banach theorem, we can extend each \(e_i^*\) to \(\widetilde e_i^*\in X^*\) with \(\norm{\widetilde e_i^*}=1\).
    Define \(\widetilde R = \sum_{i=1}^n\widetilde e_i^*\otimes Re_i.\)
    Then \(\widetilde R|_E=R\), \(\operatorname{rank}\widetilde R\leq n\), and \(\norm{\widetilde R}\leq\sum_{i=1}^n\norm{Re_i}\leq n\varepsilon\).
    Therefore, \(A=S+\widetilde R\) has the required properties.
\end{proof}

By the UAP, we obtain the following comparison between the stable recovery error and Kolmogorov widths.

\begin{proposition}\label{prop:UAP-coder-realization}
    Let \(X\) be a Banach space, let \(\Lambda\geq1\), let \(n\in\N\), and suppose that \(k_X(\Lambda,n)<\infty\).
    Then every \(K\Subset X\) satisfies
    \[ d_{k_X(\Lambda,n)}(K) \leq \operatorname{slc}_{k_X(\Lambda,n)}^\Lambda(K) \leq (1+\Lambda)d_n(K). \]
\end{proposition}

\begin{proof}
    By Proposition \ref{prop:slc-width-lower-bound}, we obtain the first inequality.
    To prove the second inequality, let \(\eta>0\), and choose \(E\subseteq X\) with \(\dim E\leq n\) such that \(\sup_{x\in K}\dist(x,E)<d_n(K)+\eta.\)
    Choose \(A\in\cF(X)\) with
    \begin{equation}\label{eq:uap-extension-conditions}
        A|_E=I_E,\qquad \norm{A}\leq\Lambda,\qquad \operatorname{rank}A\leq k_X(\Lambda,n).
    \end{equation}
    For every \(x\in K\), choose \(y\in E\) with \(\norm{x-y}<d_n(K)+\eta\).
    By \eqref{eq:uap-extension-conditions}, we have
    \[ \norm{x-Ax} \leq \norm{x-y}+\norm{A(y-x)} \leq (1+\Lambda)(d_n(K)+\eta). \]
    By taking the supremum over \(K\) and then letting \(\eta\downarrow0\), we obtain the upper estimate.
\end{proof}

The estimate in Proposition~\ref{prop:UAP-coder-realization} records both the error factor \(1+\Lambda\) and the recovery rank bound \(k_X(\Lambda,n)\) for a model of dimension \(n\).
For the converse, we require the recovery estimate to hold for every compact set.
Since the unit ball of an \(n\)-dimensional subspace has zero \(n\)-th width, we can approximate the identity on this ball arbitrarily well under the given rank and norm bounds.
An invertible change on the ambient space gives exact agreement on the subspace and preserves the rank.

\begin{theorem}\label{prop:stable-coder-characterization-UAP}
    For a Banach space \(X\), the following are equivalent.
    \begin{enumerate}
        \item \(X\) has the UAP.
        \item There are \(\Lambda\geq1\), \(C\geq0\), and \(f:\N\to\N\) such that for every \(K\Subset X\) and every \(n\in\N\), we have \(\operatorname{slc}_{f(n)}^\Lambda(K) \leq C d_n(K).\)
    \end{enumerate}
    If the second condition holds, then, for every \(\eta>0\) and \(n\in\N\), we have \(k_X(\Lambda+\eta,n)\leq f(n)\).
\end{theorem}

\begin{proof}
    Suppose first that \(X\) has the UAP.
    Choose \(\Lambda\geq1\) such that \(k_X(\Lambda,n)<\infty\) for every \(n\), and put \(f(n)=k_X(\Lambda,n)\).
    By Proposition~\ref{prop:UAP-coder-realization}, we obtain the second condition with \(C=1+\Lambda\).

    Conversely, assume the second condition, and fix \(\eta>0\), \(n\geq1\), and \(E\subseteq X\) with \(1\leq\dim E\leq n\).
    By choosing an Auerbach basis and applying the Hahn--Banach theorem, we obtain a bounded projection \(Q:X\to E\) with \(\norm Q\leq n\).
    Since \(B_E\) is compact and \(d_n(B_E)=0\), we have \(\operatorname{slc}_{f(n)}^\Lambda(B_E)=0\).

    For every admissible \(S\), the recovery error on \(B_E\) equals \(\norm{(I_X-S)|_E}\).
    Thus this operator norm can be made arbitrarily small while the rank and norm bounds remain \(f(n)\) and \(\Lambda\).
    Choose \(0<\delta<\eta/[n(\Lambda+\eta)]\) and \(S\in\cF(X)\) with \(\operatorname{rank}S\leq f(n)\), \(\norm S\leq\Lambda\), and \(\norm{(S-I_X)|_E}<\delta\).
    Put \(V=I_X+(S-I_X)Q\).
    Since \(Q(X)\subseteq E\), we have \(\norm{V-I_X}\leq\norm{(S-I_X)|_E}\norm Q<n\delta\).
    By \(\norm{V-I_X}<n\delta<1\), we have that \(V\) is invertible and \(\norm{V^{-1}}\leq(1-n\delta)^{-1}\).
    The choice of \(\delta\) then gives \(1-n\delta>\Lambda/(\Lambda+\eta)\).
    
    Define \(A=V^{-1}S\).
    For \(e\in E\), the identity \(Qe=e\) gives \(Ve=e+(S-I_X)e\).
    Also, \(\operatorname{ran}A=V^{-1}(\operatorname{ran}S)\), and \(V^{-1}\) preserves the dimension of this finite-dimensional range.
    Since \(Ve=Se\) for every \(e\in E\), we have
    \[ A|_E=I_E,\qquad \operatorname{rank}A=\operatorname{rank}S\leq f(n),\qquad \norm A\leq\frac{\Lambda}{1-n\delta}<\Lambda+\eta. \]
    Since these bounds hold for every \(E\) of dimension at most \(n\), we obtain \(k_X(\Lambda+\eta,n)\leq f(n)\).
    Therefore, \(X\) has the UAP.
\end{proof}

The converse shows that compact-set geometry alone cannot supply this uniform recovery estimate.

We next define the corresponding property for nonlinear recovery maps.
The gauge is imposed on the whole space, and the rank is measured by the dimension of the linear span of the whole image.

\begin{definition}\label{def:omega-f-UAP}\label{def:omega-UAP}\label{def:gauge-UAP}\label{def:common-gauge-recovery-error}
    Let \(K\Subset X\), let \(m\in\N\), and let \(\omega\) be a vanishing gauge.
    Define
    \[
        \mathcal E_m^\omega(K;X)=\inf\left\{\sup_{x\in K}\norm{x-Px}:\begin{array}{l}
            P:X\to X\text{ uniformly continuous},\\
            P(X)\text{ relatively compact},\ \dim\Span P(X)\leq m,\ \omega_P\leq\omega
        \end{array}\right\}.
    \]
    We write \(\mathcal E_m^\omega(K)\) when the ambient space is clear.

    Let \(f:\N\to\N\) be nondecreasing with \(f(0)=0\).
    We say that \(X\) has the \((\omega,f)\)-uniform approximation property (\((\omega,f)\)-UAP) if \(\mathcal E_{f(\dim E)}^\omega(RB_E;X)=0\) for every finite-dimensional \(E\subseteq X\) and every \(R>0\).
    We say that \(X\) has the \(\omega\)-UAP if it has the \((\omega,f)\)-UAP for some such \(f\).
    It has the gauge uniform approximation property (Gauge-UAP) if it has the \(\omega\)-UAP for some vanishing gauge \(\omega\).
\end{definition}

Here \(\omega\) and \(f\) are independent of \(E,R\), and the approximation accuracy, while the recovery maps may vary.
By considering \(\Span P(X)\), we have \(d_m(K)\leq\mathcal E_m^\omega(K)\).
For a vanishing gauge \(\omega\), by subadditivity and monotonicity, we have
\begin{equation}\label{eq:gauge-asymptotic-slope}
    L=\lim_{t\to\infty}\frac{\omega(t)}{t}=\inf_{t>0}\frac{\omega(t)}{t}<\infty.
\end{equation}
In particular, \(Lt\leq\omega(t)\) for every \(t\geq0\).

The following lemma is a standard consequence of Lipschitz
linearization and local reflexivity.

\begin{lemma}\label{lem:finite-range-linearization}
    Let \(T\in\cB(X,Y)\), let \(E\subseteq X\) be finite-dimensional, identify \(Y\) with its canonical image in \(Y^{**}\), and let \(F\subseteq Y^{**}\) be finite-dimensional with \(T(E)\subseteq F\).
    Suppose that \(Q:X\to F\) is \(\lambda\)-Lipschitz with \(Q|_E=T|_E\), where \(\lambda>0\).
    Then, for every \(\eta>0\), there is \(A\in\cF(X,Y)\) such that
    \[ A|_E=T|_E,\qquad \operatorname{rank}A\leq\dim F,\qquad \norm A\leq\lambda+\eta. \]
\end{lemma}

\begin{proof}
    By the Lipschitz linearization theorem \cite[Theorem~7.2]{BenyaminiLindenstrauss2000}, there is a bounded linear map \(B:X\to F\) with \(B|_E=T|_E\) and \(\norm B\leq\lambda\).
    By local reflexivity \cite{Ryan2002}, we can choose a linear map \(J:F\to Y\) with \(J|_{F\cap Y}=I_{F\cap Y}\) and \(\norm J\leq1+\eta/\lambda\).
    Then \(A=JB\) is the desired map.
\end{proof}

\Needspace{12\baselineskip}
\begin{theorem}\label{thm:common-gauge-recovery-UAP}
    A Banach space \(X\) has the Gauge-UAP if and only if it has the UAP.
    More precisely, let \(L\) be given by \eqref{eq:gauge-asymptotic-slope}.
    Then \(X\) has the \((\omega,f)\)-UAP if and only if
    \[ k_X(L+\eta,n)\leq f(n)\qquad(\eta>0,\ n\in\N). \]
    In this case, every \(K\Subset X\) and every \(n\in\N\) satisfy
    \[ d_{f(n)}(K)\leq\mathcal E_{f(n)}^\omega(K)\leq\operatorname{slc}_{f(n)}^L(K)\leq(1+L)d_n(K). \]
\end{theorem}

\begin{proof}
    Suppose that \(X\) has the \((\omega,f)\)-UAP.
    Fix \(n\geq1\) and \(E\subseteq X\) with \(\dim E\leq n\), and put \(m=f(n)\).
    For each \(N\in\N_+\), choose an admissible \(P_N\) with \(\dim\Span P_N(X)\leq f(\dim E)\leq m\) and \(\sup_{e\in N^2B_E}\norm{P_Ne-e}<1\).
    Define
    \[ Q_N(x)=\frac{P_N(Nx)-P_N(0)}{N}. \]
    Since \(P_N(0)\in\Span P_N(X)\), we have \(\dim\Span Q_N(X)\leq m\).
    By subadditivity, we have\( \norm{Q_Nx}\leq\omega(\norm x)\), and \(\norm{Q_Nx-Q_Ny}\leq N^{-1}\omega(N\norm{x-y}).\)
    By the same compactness argument as in the proof of the Banach--Alaoglu theorem,
    we obtain a subnet converging pointwise weak-star to \(Q:X\to X^{**}\), which is \(L\)-Lipschitz by \eqref{eq:gauge-asymptotic-slope}.
    For every  \(e\in E\), we have \(Ne\in N^2B_E\) for large enough \(N\), so \(\norm{Q_Ne-e}<2/N\) and \(Q|_E=I_E\).
    For any \(x_1,\ldots,x_{m+1}\in X\) and \(x_1^*,\ldots,x_{m+1}^*\in X^*\), we have
    \( \det\bigl[\langle Q_Nx_j,x_i^*\rangle\bigr]_{i,j=1}^{m+1}=0. \)
    Thus we obtain \( \det\bigl[\langle Q x_j,x_i^*\rangle\bigr]_{i,j=1}^{m+1}=0. \)
    Since \(X^*\) separates points on every finite-dimensional subspace of \(X^{**}\), we know  \(\dim F= \dim \Span Q(X) \leq m\).
    For every \(\eta>0\), by Lemma~\ref{lem:finite-range-linearization} applied with \(T=I_X\) and \(Y=X\), we can choose \(A\in\cF(X)\) with \(A|_E=I_E\), \(\operatorname{rank}A\leq f(n)\), and \(\norm A\leq L+\eta\).
    Thus \(k_X(L+\eta,n)\leq f(n)\) and \(X\) has the UAP.

    Conversely, assume the stated rank bounds, and let \(K\Subset X\), \(n\in\N\), and \(R_K=\sup_{x\in K}\norm x\).
    For all \(\zeta>0\), choose \(E\subseteq X\), \(\dim E\leq n\), with \(\sup_{x\in K}\dist(x,E)<\delta:=d_n(K)+\zeta\).
    For \(\tau>0\), the assumed rank bound gives \(A_\tau\in\cF(X)\) with \(A_\tau|_E=I_E\), \(\operatorname{rank}A_\tau\leq f(n)\), and \(\norm{A_\tau}\leq L+\tau\).
    Put \(B_\tau=L A_\tau/(L+\tau)\).
    Then \(\norm{B_\tau}\leq L\) and \(\operatorname{rank}B_\tau\leq f(n)\).
    Since \(A_\tau|_E=I_E\), we have \(\norm{x-A_\tau x}\leq(1+L+\tau)\dist(x,E)\) for every \(x\in X\).
    Then we obtain
    \[
        \begin{aligned}
            \operatorname{slc}_{f(n)}^L(K)
            \leq\sup_{x\in K}\norm{x-B_\tau x}\leq\frac{L}{L+\tau}\sup_{x\in K}\norm{x-A_\tau x}+\frac{\tau}{L+\tau}R_K\leq\frac{L(1+L+\tau)}{L+\tau}\delta+\frac{\tau}{L+\tau}R_K.
        \end{aligned}
    \]
    By letting \(\tau\downarrow0\) and then  \(\zeta\downarrow0\), we have
    \( \operatorname{slc}_{f(n)}^L(K)\leq(1+L)d_n(K). \)

    To compare the nonlinear and linear recovery errors, take \(B\in\cF(X)\) with \(\norm B\leq L\) and \(\operatorname{rank}B\leq f(n)\), and choose \(M>LR_K\).
    For \(a>0\) and \(v\in\operatorname{ran}B\), define
    \[
        \rho_a(r)=
        \begin{cases}
            1,&0\leq r\leq M,\\
            \max\{1-a\log(r/M),0\},&r>M,
        \end{cases}
        \qquad h_a(v)=\rho_a(\norm v)v.
    \]
    For \(0<s\leq r\), the bound \(s[\rho_a(s)-\rho_a(r)]\leq as\log(r/s)\leq a(r-s)\) gives \(\Lip{h_a}\leq1+a\).
    Thus \(P_a=(1+a)^{-1}h_aB\) is uniformly continuous, \(\omega_{P_a}(t)\leq Lt\leq\omega(t)\), and \(\dim\Span P_a(X)\leq f(n)\).
    Since \(h_a(v)=0\) for \(\norm v\geq M e^{1/a}\), the image \(P_a(X)\) is bounded in \(\operatorname{ran}B\) and hence relatively compact.
    On \(K\), we have \(P_ax=Bx/(1+a)\), so
    \[ \mathcal E_{f(n)}^\omega(K)\leq\sup_{x\in K}\norm{x-P_ax}\leq\sup_{x\in K}\norm{x-Bx}+\frac{aLR_K}{1+a}. \]
    By letting \(a\downarrow0\) and taking the infimum over \(B\), we obtain \(\mathcal E_{f(n)}^\omega(K)\leq\operatorname{slc}_{f(n)}^L(K)\).

    For every \(P\) admissible in the definition of \(\mathcal E_{f(n)}^\omega(K)\), put \(F=\Span P(X)\).
    Since \(\dim F\leq f(n)\) and \(Px\in F\) for all \(x\in K\), we have\( d_{f(n)}(K)\leq\sup_{x\in K}\dist(x,F)\leq\sup_{x\in K}\norm{x-Px} \).
    Taking the infimum over all such \(P\) gives \(d_{f(n)}(K)\leq\mathcal E_{f(n)}^\omega(K)\).
    For every finite-dimensional \(E\subseteq X\) and \(R>0\), we have \(d_{\dim E}(RB_E)=0\), since \(RB_E\subseteq E\).
    The upper bound therefore gives \( 0\leq\mathcal E_{f(\dim E)}^\omega(RB_E)\leq(1+L)d_{\dim E}(RB_E)=0. \)
    Thus \(X\) has the \((\omega,f)\)-UAP.
    If \(X\) has the UAP, choose \(\Lambda\geq1\) with \(k_X(\Lambda,n)<\infty\) for every \(n\), and take \(f(n)=k_X(\Lambda,n)\) and \(\omega(t)=\Lambda t\).
    The rank condition then holds, so \(X\) has the Gauge-UAP.
\end{proof}

\subsection[Examples and rank growth]{Examples and rank growth}\label{subsec:uap-examples-rank-growth}

The results below show when \(k_X(\Lambda,n)\) is finite and give bounds for its growth, and hence for the rank of stable recovery maps.
If \(X\) fails the AP, then there are \(K\Subset X\) and \(\varepsilon_0>0\) such that
\( \inf_{A\in\cF(X)}\sup_{x\in K}\norm{x-Ax}\geq\varepsilon_0. \)
Thus we can approximate a compact set by finite-dimensional subspaces, but we cannot in general obtain such approximations from bounded finite-rank operators on the whole space.

We next recall an example involving the standard universal spaces.

\begin{example}\label{ex:universal-fd-geometry}
    For \(1\leq p<\infty\), Casazza's survey defines standard universal spaces \(\mathcal C_p\) and \(\mathcal U_p\) from \(\ell_p\)-sums of finite-dimensional spaces; see \cite{Casazza2001Approximation}.
    Their finite-dimensional blocks range over Banach--Mazur dense families.
    Thus they provide fixed spaces in which we can test arbitrary finite-dimensional geometry.
    However, this local universality does not imply the UAP.
    In fact, \(\mathcal C_1\) and \(\mathcal U_1\) fail the UAP.
\end{example}

The next results recall the classical Hilbert and weak-Hilbert rank characterizations \cite{LindenstraussTzafriri1971,JohnsonPisier1992}.
Their role here is to identify spaces in which stable realization requires no increase, or only a proportional increase, in the model dimension.

\begin{proposition}\label{prop:hilbert-rank}
    Let \(X\) be an infinite-dimensional Banach space.
    Then \(X\) is isomorphic to a Hilbert space if and only if there is \(\Lambda\geq1\) such that \(k_X(\Lambda,n)=n\), for each \(n \in \N\).
\end{proposition}

\begin{definition}\label{def:weak-Hilbert}
    A Banach space \(X\) is a weak Hilbert space if there are constants \(C\geq1\) and \(0<\delta\leq1\) such that every finite-dimensional subspace \(E\subseteq X\) contains a subspace \(F\subseteq E\) for which
    \[ \dim F\geq\delta\dim E,\quad \dist_{BM}(F,\ell_2^{\dim F})\leq C, \]
    and there is a projection \(P:X\to F\) with \(\norm{P}\leq C\).
    Here \(\dist_{BM}\) denotes the Banach--Mazur distance.
\end{definition}

\begin{remark}
    The same proportional Euclidean-subspace condition without the uniformly bounded projection is the weak cotype \(2\) condition of Milman and Pisier \cite{MilmanPisier1986WeakCotype}.
    The projection requirement is therefore essential in Definition \ref{def:weak-Hilbert}; see \cite{Pisier1988WeakHilbert} for equivalent formulations of the weak Hilbert property.
\end{remark}

Johnson and Pisier \cite{JohnsonPisier1992} proved that proportional rank growth characterizes weak Hilbert spaces.

\begin{theorem}\label{prop:weak-Hilbert-rank}
    A Banach space \(X\) is a weak Hilbert space if and only if there are constants \(\Lambda\geq1\) and \(C<\infty\) such that \(k_X(\Lambda,n)\leq Cn\), for each \(n \in \N\).
\end{theorem}

Also, the following example shows that a weak Hilbert space need not be isomorphic to a Hilbert space.

\begin{example}\label{ex:convexified-Tsirelson}
    Let \(T^{(2)}\) denote the \(2\)-convexification of the classical Tsirelson space.
    It is a weak Hilbert space but is not isomorphic to a Hilbert space \cite{Pisier1988WeakHilbert,CasazzaNielsen2003WeakCotype}.
    Hence Theorem \ref{prop:weak-Hilbert-rank} gives constants \(\Lambda\geq1\) and \(C<\infty\) such that \(k_{T^{(2)}}(\Lambda,n)\leq Cn\) for all \(n\), while Proposition \ref{prop:hilbert-rank} shows that there is no fixed \(\Lambda\) for which \(k_{T^{(2)}}(\Lambda,n)=n\) for every \(n\).
\end{example}

We use the following disjointization theorem to obtain the UAP for \(L_p\) and \(\ell_p\) \cite{LindenstraussTzafriri1976UAP,CasazzaNielsen2003WeakCotype}.

\begin{theorem}\label{thm:lattice-disjoint-model}
    Let \(0<\varepsilon<1\), let \(L\) be an order-complete Banach lattice, and let \(E\subseteq L\) be \(k\)-dimensional, where \(k\geq1\).
    There are pairwise disjoint elements \(x_1,\ldots,x_N\in L\), with \(N\leq\left(2k^2/\varepsilon\right)^k,\) and a linear map \(V:E\to\Span\{x_1,\ldots,x_N\}\) such that \(\norm{Vx-x}\leq\varepsilon\norm{x}\), for all \(x \in E\).
\end{theorem}

For \(1\leq p\leq\infty\) and a \(\sigma\)-finite measure \(\mu\), finite-dimensional disjoint spans in \(L_p(\mu)\) and \(\ell_p\) are \(1\)-complemented.
By Theorem~\ref{thm:lattice-disjoint-model} and Lemma~\ref{lem:finite-dimensional-correction}, we obtain the metric UAP for these spaces.

The following theorem gives the ultrapower characterization and self-duality of the UAP \cite{Heinrich1980,Mascioni1990UAP,Mascioni1991Duality,Casazza2001Approximation}.
These permanence results extend the recovery estimate to further function spaces used as domains and ranges of operators.

\Needspace{6\baselineskip}
\begin{theorem}\label{thm:Heinrich-UAP-stability}
    Let \(X\) be a Banach space.
    \begin{enumerate}
        \item \(X\) has the UAP if and only if every ultrapower \(X_{\mathcal U}\) has the BAP.
        \item \(X\) has the UAP if and only if \(X^*\) has the UAP.
    \end{enumerate}
    Here \(X_{\mathcal U}\) denotes the ultrapower of \(X\) with respect to \(\mathcal U\).
\end{theorem}

For \(1\leq p<\infty\), we have isometric identifications \((L_p(\mu))_{\mathcal U}\cong L_p(\nu)\) and \((L_p(\mu;M))_{\mathcal U}\cong L_p(\nu;M)\) for a suitable measure \(\nu\) and every finite-dimensional Banach space \(M\) \cite[Theorem~3.3(ii) and Lemma~9.5]{Heinrich1980}.
By Theorem~\ref{thm:Heinrich-UAP-stability}(1) and a finite-dimensional reduction, we obtain the UAP for \(L_p(\mu;X)\), \(1\leq p\leq\infty\), whenever \(X\) has the UAP; see also \cite[proof of Theorem~9.4]{Heinrich1980}.
For a compact Hausdorff space \(K\), \(C(K)\) has the UAP \cite[Theorems~3.3(i) and 9.1]{Heinrich1980}.
By Theorem~\ref{thm:Heinrich-UAP-stability}(2), we obtain the UAP for \(M(K)=C(K)^*\), the space of finite regular scalar Borel measures on \(K\) with the total variation norm.
Thus the stable recovery estimate in Proposition~\ref{prop:UAP-coder-realization} applies to norm-compact subsets of all these spaces.

\subsection[Stable recovery of linear factors]{Stable recovery of linear factors}\label{subsec:stable-recovery-linear-factors}

To approximate \(T_F=LFS\), we seek finite-rank approximations to \(S\) on the input set and to \(L\) on the corresponding image set.
The norm and rank conditions in the next definition concern the operator itself.
These quantities can be finite even when neither the domain nor the codomain has the UAP.
In Section~\ref{subsec:operator-to-functional-reduction}, we use them to bound the approximation error for \(T_F\).

\begin{definition}\label{def:operator-UAP-rank}
    Let \(T\in\cB(X,Y)\), let \(\Lambda\geq1\), and let \(n,m\in\N\).
    Define
    \[
        k_T(\Lambda,n) = \sup_{\substack{E\subseteq X\\
                \dim E\leq n}} \inf\left\{ \operatorname{rank}A: A\in\cF(X,Y),\ A|_E=T|_E,\ \norm{A}\leq\Lambda\norm{T} \right\},
    \]
    and, for \(K\Subset X\),
    \[
        \sigma_m^\Lambda(T;K) = \inf_{\substack{A\in\cF(X,Y)\\
                \operatorname{rank}A\leq m\\
                \norm{A}\leq\Lambda\norm{T}}} \sup_{x\in K}\norm{Tx-Ax}.
    \]

\end{definition}
The first quantity is the normalized exact rank function.
The second is the stable finite-rank error on \(K\).
By considering the image of an admissible operator, we obtain the following width lower bound.

\begin{proposition}\label{prop:operator-image-width-lower-bound}
    Let \(T\in\cB(X,Y)\), \(K\Subset X\), \(m\in\N\), and \(\Lambda\geq1\).
    Then \(d_m(T(K))\leq\sigma_m^\Lambda(T;K).\)
\end{proposition}

\begin{proof}
    Let \(A\in\cF(X,Y)\) be admissible.
    Then we have \(\dim\operatorname{ran}A\leq m\) and \(\dist(Tx,\operatorname{ran}A)\leq\norm{Tx-Ax}\) for all \(x\in K\).
    By taking the supremum over \(K\), followed by the infimum over \(A\), we obtain the required inequality.
\end{proof}

The next proposition bounds the operator rank function by the rank function of either the domain or the codomain.

\begin{proposition}\label{prop:operator-UAP-endpoint-domination}
    Let \(T\in\cB(X,Y)\), let \(\Lambda\geq1\), and let \(n\in\N\).
    Then
    \[ k_T(\Lambda,n) \leq \min\{k_X(\Lambda,n),k_Y(\Lambda,n)\}. \]
\end{proposition}

\begin{proof}
    Fix \(E\subseteq X\) with \(\dim E\leq n\).
    For \(k_X(\Lambda,n)<\infty\), choose \(R\in\cF(X)\) satisfying \eqref{eq:uap-extension-conditions} with \(A=R\).
    Then \(TR\) agrees with \(T\) on \(E\), has norm at most \(\Lambda\norm{T}\), and has no larger rank.
    Since the subspace was arbitrary, we obtain \(k_T(\Lambda,n)\leq k_X(\Lambda,n)\).

    For \(k_Y(\Lambda,n)<\infty\), since \(\dim T(E)\leq n\), choose \(P\in\cF(Y,Y)\) satisfying \eqref{eq:uap-extension-conditions} with \(X,E,A\) replaced by \(Y,T(E),P\), respectively.
    Then \(PT\) agrees with \(T\) on \(E\), has norm at most \(\Lambda\norm{T}\), and has no larger rank.
    Since the subspace was arbitrary, we obtain the second estimate.
\end{proof}

\begin{example}\label{ex:operator-uap-without-space-uap}
    Let \(Z\) be a Banach space without the UAP, put \(X=Z\oplus_2\ell_2\), and define \(T:X\to X\) by \(T(z,h)=(0,h)\).
    Then \(X\) does not have the UAP.
    The operator \(T\) has norm one, infinite rank, and \( k_T(1,n)=n \), for all \(n\).
\end{example}

By the correction lemma, we can pass from approximate agreement to exact agreement, with the following rank and norm bounds.

\begin{corollary}\label{cor:approximate-to-exact-rank}
    Let \(T\in\cB(X,Y)\), let \(\Gamma\geq0\), let \(n\in\N_+\), and let \(\delta>0\).
    Suppose there is \(M(n,\delta)\in\N\) such that every \(E\subseteq X\) with \(\dim E\leq n\) admits \(S_{E,\delta}\in\cF(X,Y)\) satisfying
    \[ \norm{(T-S_{E,\delta})|_E}\leq\delta\norm{T},\qquad \operatorname{rank}S_{E,\delta}\leq M(n,\delta),\qquad \norm{S_{E,\delta}}\leq\Gamma\norm{T}. \]
    Then \(k_T(\max\{1,\Gamma+n\delta\},n) \leq M(n,\delta)+n.\)
    In particular,
    \[ k_T(\Gamma+1,n)\leq M(n,1/n)+n. \]
\end{corollary}

\begin{proof}
    Let \(E\subseteq X\) with \(1\leq q=\dim E\leq n\).
    By Lemma~\ref{lem:finite-dimensional-correction} with \(\varepsilon=\delta\norm T\) and \(C=\Gamma\norm T\), we obtain \(A_E\) which agrees with \(T\) on \(E\) and satisfies \(\operatorname{rank}A_E\leq M(n,\delta)+q\leq M(n,\delta)+n\) and \(\norm{A_E}\leq(\Gamma+q\delta)\norm T\leq(\Gamma+n\delta)\norm T\).
    Since these bounds hold for every \(E\), we have the first estimate.
    By taking \(\delta=1/n\), we get the second estimate.
\end{proof}

The correction may change the original measurement structure.
For a prescribed rank bound \(f\), we define \(\mathbf u_f(T)\) to record the admissible norm bounds.
Here we use an absolute norm bound, whereas the bound in \(k_T\) is normalized by the norm of the operator.

\begin{definition}\label{def:operator-ideal-gauge}\label{def:operator-gauge-UAP}
    Let \(f:\N\to\N\) be nondecreasing with \(f(0)=0\).
    For \(T\in\cB(X,Y)\), define
    \[
        \mathbf u_f(T) = \inf\left\{ C\geq0: \begin{array}{l}
            \text{for every finite-dimensional }E\subseteq X,\text{ there is }A_E\in\cF(X,Y), \text{ so that} \\
            A_E|_E=T|_E,\quad \operatorname{rank}A_E\leq f(\dim E),\quad \norm{A_E}\leq C
        \end{array}
        \right\}.
    \]
    Let \(\omega\) be a vanishing gauge and let \(m\in\N\).
    For \(K\Subset X\), define
    \[
        \mathcal E_m^\omega(T;K)=\inf\left\{\sup_{x\in K}\norm{Tx-Px}:\begin{array}{l}
            P:X\to Y\text{ uniformly continuous},\\
            P(X)\text{ relatively compact},\ \dim\Span P(X)\leq m,\ \omega_P\leq\omega
        \end{array}\right\}.
    \]
   We say that \(T\) has the \((\omega,f)\)-UAP if \(\mathcal E_{f(\dim E)}^\omega(T;RB_E)=0\) for every finite-dimensional \(E\subseteq X\) and every \(R>0\).
    It has the \(\omega\)-UAP if this holds for some such \(f\), and the Gauge-UAP if it has the \(\omega\)-UAP for some vanishing gauge \(\omega\).
\end{definition}

By the same arguments as in Proposition~\ref{prop:UAP-coder-realization} and Theorem~\ref{thm:common-gauge-recovery-UAP}, we obtain the following operator version.

\begin{corollary}\label{prop:operator-UAP-input-bound}\label{cor:operator-gauge-UAP}
    Let \(T\in\cB(X,Y)\).
    \begin{enumerate}
        \item For \(\Lambda\geq1\), \(K\Subset X\), and \(n\in\N\) with \(k_T(\Lambda,n)<\infty\), we have
        \[ d_{k_T(\Lambda,n)}(T(K))\leq\sigma_{k_T(\Lambda,n)}^\Lambda(T;K)\leq(1+\Lambda)\norm T\,d_n(K). \]
        \item Let \(L\) be given by \eqref{eq:gauge-asymptotic-slope}.
        Then \(T\) has the \((\omega,f)\)-UAP if and only if \(\mathbf u_f(T)\leq L\).
        If \(T\neq0\) and these conditions hold, every \(K\Subset X\) and every \(n\in\N\) satisfy
        \[ d_{f(n)}(T(K))\leq\mathcal E_{f(n)}^\omega(T;K)\leq\sigma_{f(n)}^{L/\norm T}(T;K)\leq(\norm T+L)d_n(K). \]
        Consequently, \(T\) has the Gauge-UAP if and only if \(\mathbf u_f(T)<\infty\) for some such \(f\).
    \end{enumerate}
\end{corollary}

The family \((\mathbf u_f)_f\) has the following ideal properties.

\begin{proposition}\label{prop:operator-ideal-gauge}
    Let \(f:\N\to\N\) be nondecreasing with \(f(0)=0\).
    Let \(A\in\cB(X_0,X)\), \(T\in\cB(X,Y)\), and \(B\in\cB(Y,Y_0)\), and suppose that \(\mathbf u_f(T)<\infty\).
    Then
    \[ \mathbf u_f(BTA) \leq \norm{B}\,\mathbf u_f(T)\,\norm{A}. \]
    If \(S,T\in\cB(X,Y)\), \(f,g:\N\to\N\) are nondecreasing with \(f(0)=g(0)=0\), and both \(\mathbf u_f(S)\) and \(\mathbf u_g(T)\) are finite, then
    \[ \mathbf u_{f+g}(S+T) \leq \mathbf u_f(S)+\mathbf u_g(T). \]
\end{proposition}

\begin{proof}
    Fix \(C>\mathbf u_f(T)\), and let \(E\subseteq X_0\) be finite-dimensional.
    Put \(F=A(E)\).
    Choose \(R_F\in\cF(X,Y)\) such that \(R_F|_F=T|_F\), \(\operatorname{rank}R_F\leq f(\dim F)\leq f(\dim E)\), and \(\norm{R_F}\leq C.\)
    Then \(BR_FA\) agrees with \(BTA\) on \(E\), has the required rank bound, and has norm at most \(\norm{B}C\norm{A}\).
    By the definition of \(\mathbf u_f\) and by letting \(C\downarrow\mathbf u_f(T)\), we obtain the composition estimate.

    For the addition estimate, fix \(C_S>\mathbf u_f(S)\) and \(C_T>\mathbf u_g(T)\).
    On a finite-dimensional \(E\subseteq X\), we can choose finite-rank operators \(R_E^S\) and \(R_E^T\) with these bounds.
    Their sum agrees with \(S+T\) on \(E\), has rank at most \(f(\dim E)+g(\dim E)\), and has norm at most \(C_S+C_T\).
    By letting both constants decrease to their respective infima, we obtain the addition estimate.
\end{proof}

We have related geometric model dimensions to stable recovery ranks, both for spaces and for the linear factors of a target operator.
The next section describes the encoding and recovery maps in classical sampling and truncation schemes, and uses their approximation estimates to bound the rank functions.

    \section{Stable recovery from classical approximation theory}\label{sec:unit-ball-UAP-coders}

Point sampling, cell averages, and wavelet and Gabor truncations define finite-dimensional encoders and reconstruction maps.
We describe these maps and their norm and approximation bounds in the notation of Section~\ref{sec:UAP}.
By finite-rank correction, operator-norm approximation estimates then give bounds for the operator rank functions.
The table summarizes the constructions in Appendix~\ref{app:classical-realizations} and the wavelet construction below under the hypotheses stated there.
The error rates are measured in the target norm on the unit ball of the source space.
For point sampling, \(E\subseteq C(K)\), \(n=\dim E\), \(\Gamma>1\) is fixed, and \(C_\Gamma\) depends only on \(\Gamma\).
\begin{center}
    \small
    \begin{tabular}{@{}llll@{}}
        \hline
        Construction & Spaces & Code dimension & Approximation \\
        \hline
        Point sampling & \(C(K)\to C(K)\) & \(\leq C_\Gamma^n\) & Exact on \(E\) \\
        Cell averages & \(\mathcal A\to L_p(\Omega)\) & \(O(2^{Jd})\) & Error \(O(2^{-Js})\) \\
        Gabor truncation & \(M^{s_0}_{p,q}\to M^{s_1}_{p,q}\) & \(O(R^{2d})\) & Error \(O(R^{-(s_0-s_1)})\) \\
        Wavelet truncation & \(H^{s_0}\to H^{s_1}\) & \(O(2^{Jd})\) & Error \(O(2^{-J(s_0-s_1)})\) \\
        \hline
    \end{tabular}
\end{center}
Here \(\mathcal A=B^s_{p,\infty}(\Omega)\) or \(W^{s,p}(\Omega)\), under the corresponding averaging assumptions.
The modulation spaces are defined on \(\R^d\), and the Sobolev spaces in the wavelet row are defined on \(\mathbb T^d\).
Appendix~\ref{app:detailed-operator-learning-examples} uses cell averages to compute recovery errors and neural parameter bounds for a nonlinear integral operator.
For cell averages, the encoder consists of normalized integrals and the decoder forms a piecewise constant function; their composition is a conditional expectation.
For Gabor and wavelet truncations, the encoder gives the retained coefficients and the decoder forms the corresponding finite sum.
With the coefficient norms used in these constructions, both maps are bounded uniformly in the truncation parameter.
We consider exact recovery on a prescribed model and approximation of the target operator.
Finite-rank correction gives exact agreement on an arbitrary finite-dimensional subspace, with coordinate maps determined by the corrected operator.

\begin{definition}\label{def:three-coder-notions}
    Let \(T:X_0\to X_1\) be bounded and linear, and let \(\mathcal T\) be a directed set.
    For each \(\tau\in\mathcal T\), let \(Z_\tau\) be finite dimensional and let \(C_\tau:X_0\to Z_\tau\) and \(D_\tau:Z_\tau\to X_1\) be bounded linear maps.
    Let \(S_\tau=D_\tau C_\tau\).
    \begin{enumerate}
        \item The reduction is exact on the finite-dimensional model space \(V_\tau\subseteq X_0\) if \(S_\tau|_{V_\tau}=T|_{V_\tau}\).
        \item The family gives \emph{operator-norm approximation} if \(\norm{T-S_\tau}\to0\).
        For \(K\Subset X_0\), it gives \emph{approximation on \(K\)} if \(\sup_{x\in K}\norm{Tx-S_\tau x}\to0\).
        \item A finite-rank operator \(A_E\) agrees with \(T\) on a finite-dimensional space \(E\subseteq X_0\) if \(A_E|_E=T|_E\).
    \end{enumerate}
\end{definition}

The finite-dimensional correction lemma gives encoders and decoders, together with a uniform rank bound for a family of operator approximations.
\begin{proposition}\label{prop:correction-to-exact-coder}\label{prop:stable-family-exactification}
    Let \(T\in\cB(X_0,X_1)\).
    \begin{enumerate}
        \item Let \(E\subseteq X_0\) have dimension \(q\geq1\), let \(S\in\cF(X_0,X_1)\), and let \(m\in\N\) and \(\delta,\Gamma_S\geq0\).
        Suppose that \(\norm{(T-S)|_E}\leq\delta\), \(\operatorname{rank}S\leq m\), and \(\norm S\leq\Gamma_S\).
        Then there are a finite-dimensional space \(Z_E\) and bounded linear maps \(C_E:X_0\to Z_E\) and \(D_E:Z_E\to X_1\) such that
        \[ D_EC_E|_E=T|_E, \qquad \dim Z_E\leq m+q, \qquad \norm{D_EC_E}\leq\Gamma_S+q\delta. \]
        \item Let \(\mathcal T\) be an index set, let \(\Gamma\geq0\), and let \(N:\mathcal T\to\N\) and \(\varepsilon:\mathcal T\to[0,\infty)\).
        Suppose that \(S_\tau=D_\tau C_\tau\) factors through a finite-dimensional space \(Z_\tau\) for each \(\tau\in\mathcal T\), with \(\dim Z_\tau\leq N(\tau)\), \(\norm{S_\tau}\leq\Gamma\norm T\), and \(\norm{T-S_\tau}\leq\varepsilon(\tau)\norm T\).
        If \(n\geq1\) and \(\varepsilon(\tau)\leq1/n\), then \(k_T(\Gamma+1,n)\leq N(\tau)+n\).
    \end{enumerate}
\end{proposition}

\begin{proof}
    For the first assertion, by Lemma~\ref{lem:finite-dimensional-correction}, we obtain an operator \(A_E\) with the required agreement, rank, and norm bounds.
    Let \(Z_E=A_E(X_0)\) with its norm inherited from \(X_1\), let \(C_Ex=A_Ex\), and let \(D_E\) be the inclusion.
    Since \(D_EC_E=A_E\), we obtain the required encoder and decoder.
    For the second assertion, let \(E\subseteq X_0\) have dimension \(1\leq q\leq n\).
    By the first assertion with \(S=S_\tau\), \(\delta=\norm T/n\), \(m=N(\tau)\), and \(\Gamma_S=\Gamma\norm T\), we obtain \(A_E|_E=T|_E\), \(\operatorname{rank}A_E\leq N(\tau)+q\), and \(\norm{A_E}\leq(\Gamma+q/n)\norm T\).
    Therefore, by the definition of \(k_T\), we obtain the asserted rank bound.
\end{proof}

The next corollary converts an algebraic operator-norm approximation rate into a bound for \(k_T\).
\begin{corollary}\label{cor:algebraic-stable-family-exactification}
    Let \(T\in\cB(X_0,X_1)\).
    Suppose that \((S_h)_{0<h\leq1}\), where \(S_h=D_hC_h\), satisfies
    \[ \dim Z_h\leq C_{\mathrm{app}}h^{-a},\qquad \norm{S_h}\leq\Gamma_{\mathrm{app}}\norm{T},\qquad \norm{T-S_h}\leq C_{\mathrm{err}}h^s\norm{T}, \]
    where \(C_{\mathrm{app}},C_{\mathrm{err}},a,s>0\).
    Then, for every \(n\geq1\), we have
    \[ k_T(\Gamma_{\mathrm{app}}+1,n) \leq C_{\mathrm{app}}2^a(1+C_{\mathrm{err}}n)^{a/s}+n. \]
\end{corollary}

\begin{proof}
    We choose \(h_n=2^{-J}\) such that \(h_n\leq(1+C_{\mathrm{err}}n)^{-1/s}<2h_n\).
    Thus \(C_{\mathrm{err}}h_n^s\leq1/n\) and \(\operatorname{rank}S_{h_n}\leq\dim Z_{h_n}\leq C_{\mathrm{app}}2^a(1+C_{\mathrm{err}}n)^{a/s}\).
    By applying Proposition~\ref{prop:stable-family-exactification} with \(\tau=h_n\), we obtain the result.
\end{proof}

\begin{remark}
    If \(T\ne0\) and \(\norm{A_E}\leq M\) uniformly in \(E\), we may take \(\max\{1,M/\norm T\}\) as the normalized constant in \(k_T\) of Definition~\ref{def:operator-UAP-rank}.
\end{remark}

For wavelet truncation, the coefficient estimates give uniform bounds for the encoder and decoder and relate the truncation level to the approximation error.
By Corollary~\ref{cor:algebraic-stable-family-exactification}, these estimates give an operator rank bound.
Let \(T:H^{s_0}(\mathbb T^d)\hookrightarrow H^{s_1}(\mathbb T^d)\), where \(0\leq s_1<s_0<\infty\), and let \((\psi_\lambda)\) be the periodized tensor-product Meyer orthonormal basis, including the constant function at the coarse level.
For every finite \(s\geq0\), by \cite[Sections 2--4, Theorems 3 and 4]{Meyer1987}, we have
\[ \norm f_{H^s}^2\asymp\sum_\lambda2^{2s|\lambda|}|c_\lambda(f)|^2, \qquad c_\lambda(f)=\langle f,\psi_\lambda\rangle_{L_2}, \]
where \(|\lambda|\) is the dyadic level and is zero for the coarse term.
Put \(C_Jf=(c_\lambda(f))_{|\lambda|\leq J}\), \(D_Ja=\sum_{|\lambda|\leq J}a_\lambda\psi_\lambda\), and \(P_J=D_JC_J\).
Equip the coefficient space \(Z_J\) with \(\norm a_{Z_J}^2=\sum_{|\lambda|\leq J}2^{2s_0|\lambda|}|a_\lambda|^2\).
By the norm equivalence above and \(s_1<s_0\), the maps \(C_J:H^{s_0}\to Z_J\) and \(D_J:Z_J\to H^{s_1}\) are bounded uniformly in \(J\).
Since there are \(O(2^{jd})\) wavelets at level \(j\), and \(2^{2s_1|\lambda|}\leq2^{-2J(s_0-s_1)}2^{2s_0|\lambda|}\) for \(|\lambda|>J\), we obtain \(\operatorname{rank}P_J\lesssim2^{Jd}\), and \(\norm{T-P_J}_{H^{s_0}\to H^{s_1}}\lesssim2^{-J(s_0-s_1)}.\)

These constructions give finite-dimensional encoders and decoders with norm, dimension, and error bounds.
In Section~\ref{sec:operator-learning-nn-complexity}, we use these bounds together with compact-set complexity and the target gauge to estimate the error of an encoder--network--decoder approximation.

    \section{Stable reductions and neural approximation of operators}\label{sec:operator-learning-nn-complexity}

We now use compact-set complexity, the target gauge, and stable recovery to construct an encoder--network--decoder approximation with a prescribed error bound.
The widths determine the model dimensions, the gauge determines the required input accuracy, and the UAP rank functions bound the ranks of the linear reductions.
We then approximate the reduced map by a finite-dimensional neural network.
For \(T_F=LFS\), we impose the rank conditions on \(S,L\), so the construction does not require the UAP for all the spaces in the factorization.
A Barron representation supplies input functionals and output vectors directly, and its sampling bounds give a further neural approximation estimate.

\subsection{Error bounds for fixed encoders and decoders}\label{subsec:operator-to-functional-reduction}

We first fix the coordinate maps and separate input recovery, output recovery, and the decoded network error.
We then use widths and UAP rank bounds to choose the linear maps with the required recovery errors.

Let \(K\Subset X\), and let \(T:X\to Y\) be continuous.
Fix input maps \(C_m^X:X\to\ell_\infty^m\), \(R_m^X:\ell_\infty^m\to X\), and output maps \(C_q^Y:Y\to\ell_\infty^q\), \(R_q^Y:\ell_\infty^q\to Y\).
We use the notation \(D_m=C_m^XK\), \(K_m=K\cup R_m^XC_m^XK\), and \(\Phi_{m,q}=C_q^YTR_m^X\).
The two reconstruction errors are
\[ \gamma_X(K,m)=\sup_{x\in K}\norm{x-R_m^XC_m^Xx}_X\quad \text{and}\quad \gamma_Y(T,K;m,q)=\sup_{x\in K} \norm{T(R_m^XC_m^Xx)-R_q^YC_q^YT(R_m^XC_m^Xx)}_Y. \]
Since the set \(K_m\) is compact, we know that the local modulus \(\omega_{T,K_m}\) vanishes at zero.

Put \(\kappa_{q,2}=\norm{R_q^Y:\ell_2^q\to Y}\).
Since \(B_{\ell_2^q}\subseteq B_{\ell_\infty^q}\), we have \(\kappa_{q,2}\leq\norm{R_q^Y:\ell_\infty^q\to Y}\).

\begin{proposition}\label{prop:operator-vector-reduction}\label{prop:master-coordinate-complexity}
    For every map \(N:\R^m\to\R^q\), we have
    \begin{align*}
        \sup_{x\in K}\norm{T(x)-R_q^YN(C_m^Xx)}_Y \leq \omega_{T,K_m}(\gamma_X(K,m))+\gamma_Y(T,K;m,q)+\sup_{z\in D_m}\norm{R_q^Y(\Phi_{m,q}(z)-N(z))}_Y.
    \end{align*}
    If \(T\) has a global gauge \(\omega_T\), the first term may be replaced by \(\omega_T(\gamma_X(K,m))\).
    Write \(\Phi_{m,q}=(\phi_1,\ldots,\phi_q)\).
    If \(N=(N_1,\ldots,N_q)\) satisfies \(\sup_{z\in D_m}\abs{\phi_j(z)-N_j(z)}\leq\eta_j\) for \(1\leq j\leq q\), then the last term is at most \(\kappa_{q,2}(\sum_{j=1}^q\eta_j^2)^{1/2}\).
\end{proposition}

\begin{proof}
    Let \(x\in K\), and put \(z=C_m^Xx\) and \(x_m=R_m^Xz\).
    Then we have
    \[ T(x)-R_q^YN(z)=T(x)-T(x_m)+T(x_m)-R_q^YC_q^YT(x_m) +R_q^Y(\Phi_{m,q}(z)-N(z)). \]
    By the definition of the local modulus, we obtain the stated bound for the first term.
    By the definition of the output reconstruction error, we also obtain the bound for the second term.
    Thus, by the triangle inequality and by taking the supremum over \(x\in K\), we obtain the first estimate.
    Since \(\norm{R_q^Yv}_Y\leq\kappa_{q,2}\norm v_2\), we also obtain the coordinate estimate.
\end{proof}

The last term of this error estimate is measured in the output norm.
For \(q\geq1\) and \(\kappa_{q,2}>0\), by choosing \(\eta_j=\eta/(\kappa_{q,2}\sqrt q)\), we make this term at most \(\eta\).

Let \(T_F=LFS:E\to X\), where \(S:E\to Z\) and \(L:Y\to X\) are bounded linear maps, and \(F:Z\to Y\) has a global vanishing gauge \(\omega_F\).
Fix \(K\Subset E\).
Since \(Su=Sv\) implies \(T_F(u)=T_F(v)\), the target depends on the input through the linear factor.
More generally, \(\norm{T_F(u)-T_F(v)}_X\leq\norm L\,\omega_F(\norm{S(u-v)}_Z)\).
Thus we approximate \(S\) on \(K\) and \(L\) on the compact image \(F(SK)\), so the recovery assumptions concern these two factors.
Write finite-rank approximations of the linear factors as \(\widetilde S=UC\) and \(\widetilde L=RV\), through \(\R^m\) and \(\R^q\), respectively, and put \(\Phi_F=VFU\).
Let \(N_F:\R^m\to\R^q\).
Put \(H=F(SK)\).
By the identity
\( T_F(u)-\widetilde L F(\widetilde Su)=(L-\widetilde L)F(Su)+\widetilde L\bigl(F(Su)-F(\widetilde Su)\bigr), \)
we have
\begin{align*}
    &\sup_{u\in K}\norm{T_F(u)-RN_F(Cu)}_X\\*
    \leq&\norm{\widetilde L}\omega_F\!\left(\sup_{u\in K}\norm{Su-\widetilde Su}_Z\right) +\sup_{y\in H}\norm{Ly-\widetilde Ly}_X +\sup_{z\in CK}\norm{R(\Phi_F(z)-N_F(z))}_X.
\end{align*}
Let \(\Lambda_S,\Lambda_L\geq1\), and fix \(n,r\in\N\) such that \(k_S(\Lambda_S,n)\) and \(k_L(\Lambda_L,r)\) are finite.
By Corollary~\ref{prop:operator-UAP-input-bound}(1), for every \(\varepsilon>0\), we can choose these maps so that
\begin{align*}
    \sup_{u\in K}\norm{Su-\widetilde Su}_Z &\leq(1+\Lambda_S)\norm S\,d_n(K;E)+\varepsilon,\\
    \sup_{y\in H}\norm{Ly-\widetilde Ly}_X &\leq(1+\Lambda_L)\norm L\,d_r(H;Y)+\varepsilon.
\end{align*}
We may take \(m\leq k_S(\Lambda_S,n)\) and \(q\leq k_L(\Lambda_L,r)\), with \(\norm{\widetilde S}\leq\Lambda_S\norm S\) and \(\norm{\widetilde L}\leq\Lambda_L\norm L\).

\begin{remark}\label{rem:factor-ideal-estimates}
    Let \(f,g\) be as in Definition~\ref{def:operator-ideal-gauge}, and suppose that \(\mathbf u_f(S),\mathbf u_g(L)<\infty\).
    For bounded linear maps \(A:E_0\to E\) and \(B:X\to X_0\), we have \(BT_FA=(BL)F(SA)\).
    By Proposition~\ref{prop:operator-ideal-gauge} and the definition of approximation numbers, we obtain
    \begin{align*}
        \mathbf u_f(SA)\leq\norm A\mathbf u_f(S),&\qquad \mathbf u_g(BL)\leq\norm B\,\mathbf u_g(L),\\
        a_j(SA)\leq\norm Aa_j(S),&\qquad  a_j(BL)\leq\norm B a_j(L).
    \end{align*}
    Thus the same functions \(f,g\) control the ranks after bounded linear changes of input and output.
    Polynomial upper bounds for the approximation numbers retain the same exponents.
\end{remark}

\begin{remark}\label{rem:factor-approximation-numbers}
    By Definition~\ref{def:operator-image-realization}, for \(m,q\in\N_+\) and \(0<\delta\leq1\), we can choose \(\widetilde S,\widetilde L\) of ranks at most \(m,q\) so that, on \(B_E\), the input and output reduction errors have sum at most
    \[ \norm L\,\omega_F(a_{m+1}(S)+\delta) +\bigl[\norm{F(0)}_Y+\omega_F(2\norm S+1)\bigr]\bigl[a_{q+1}(L)+\delta\bigr]. \]
    If \(\omega_F(t)\leq Bt^\beta\), with \(B>0\), \(0<\beta\leq1\), and \(a_j(S)\lesssim j^{-\alpha}\), \(a_j(L)\lesssim j^{-\gamma}\), with \(\alpha,\gamma>0\), this sum is \(O(m^{-\alpha\beta}+q^{-\gamma})\) for \(0<\delta\leq\min\{m^{-\alpha},q^{-\gamma},1\}\).

    Under the same H\"older condition, if \(S\) is compact, we have \(e_j(\overline{T_F(B_E)};X)\leq\norm L\,B\,e_j(S)^\beta\).
    By Corollary~\ref{cor:carl-polynomial}, a bound \(s_j(S)\lesssim j^{-s}\), \(s>0\), with \(s_j\in\{a_j,c_j,d_j\}\), therefore gives \(e_j(\overline{T_F(B_E)};X)\lesssim j^{-s\beta}\).
\end{remark}

\subsection{Recovery ranks for the linear factors}\label{subsec:compact-widths-uap-selected}\label{subsec:operator-level-summary}

A global vanishing gauge determines the required input accuracy for the rank selection below.
For a vanishing gauge \(\omega\) and \(\eta>0\), put
\[
    \delta_\omega(\eta)= \begin{cases}
        \frac12\sup\{t>0:\omega(t)<\eta\},&\text{if the supremum is finite},\\
        1,&\text{if the supremum is infinite}.
    \end{cases}
\]
Since \(\omega\) is nondecreasing and vanishes at zero, we have \(\delta_\omega(\eta)>0\) and \(\omega(\delta_\omega(\eta))<\eta\).

For \(K\Subset X\), \(\Lambda\geq1\), and \(\delta>0\), put \(n=r_K(\delta/(2(1+\Lambda)))\) and \(m=k_X(\Lambda,n)\), and suppose that \(m<\infty\).
By Proposition~\ref{prop:UAP-coder-realization} and \(d_n(K)\leq\delta/(2(1+\Lambda))\), there is \(A\in\cF(X)\) such that
\[ \operatorname{rank}A\leq m,\qquad \norm A\leq\Lambda,\qquad \sup_{x\in K}\norm{x-Ax}<\delta. \]

For any \(A\in\cF(X,Y)\) with \(d=\dim A(X)\leq m\), choose an Auerbach system \((u_i,u_i^*)_{i=1}^d\) in \(A(X)\).
Put \(Cx=(u_1^*(Ax),\ldots,u_d^*(Ax),0,\ldots,0)\) and \(Ra=\sum_{i=1}^d a_i u_i\).
Then \(RC=A\), \(\norm{C:X\to\ell_\infty^m}\leq\norm A\), and \(\norm{R:\ell_\infty^m\to Y}\leq d\leq m\).
By the Cauchy--Schwarz inequality, \(\norm{R:\ell_2^m\to Y}\leq\sqrt d\leq\sqrt m\).

We now choose finite-rank approximations of \(S\) and \(L\) in \(T_F=LFS\) for the prescribed error tolerances.
We first choose the output recovery on \(F(SK)\), then use its norm bound and the gauge to determine the input accuracy.
\begin{theorem}\label{thm:UAP-profile-operator-template}
    Let \(K\Subset E\) be nonempty, let \(S:E\to Z\) and \(L:Y\to X\) be nonzero bounded linear maps, and let \(F:Z\to Y\) have a global vanishing gauge \(\omega_F\).
    Let \(\Lambda_S,\Lambda_L\geq1\), and assume that \(k_S(\Lambda_S,n)\) and \(k_L(\Lambda_L,n)\) are finite for every \(n\in\N\).
    Fix positive errors \(\varepsilon_{\rm in},\varepsilon_{\rm out},\eta\), and put \(H=F(SK)\), \(n_L=r_H(\varepsilon_{\rm out}/[2(1+\Lambda_L)\norm L])\), and \(q=k_L(\Lambda_L,n_L)\).
    There is \(\widetilde L=RV\), through \(\ell_\infty^q\), with \(\norm{\widetilde L}\leq\Lambda_L\norm L\) and
    \[ \norm V\leq\Lambda_L\norm L,\qquad \norm R\leq q,\qquad \sup_{y\in H}\norm{Ly-\widetilde Ly}_X<\varepsilon_{\rm out}. \]
    Moreover, \(\norm{R:\ell_2^q\to X}\leq\sqrt q\).
    Put \(\delta_S=\delta_{\omega_F}(\varepsilon_{\rm in}/[\Lambda_L\norm L])\), \(n_S=r_K(\delta_S/[2(1+\Lambda_S)\norm S])\), and \(m=k_S(\Lambda_S,n_S)\).
    There is \(\widetilde S=UC\), through \(\ell_\infty^m\), with \(\norm{\widetilde S}\leq\Lambda_S\norm S\) and
    \[ \norm C\leq\Lambda_S\norm S,\qquad \norm U\leq m,\qquad \sup_{x\in K}\norm{Sx-\widetilde Sx}_Z<\delta_S. \]

    Put \(\Phi_F=VFU\).
    If \(N:\R^m\to\R^q\) satisfies \(\sup_{z\in CK}\norm{R(\Phi_F(z)-N(z))}_X\leq\eta\), then
    \[ \sup_{x\in K}\norm{T_F(x)-RN(Cx)}_X <\varepsilon_{\rm in}+\varepsilon_{\rm out}+\eta. \]
\end{theorem}

\begin{proof}
    Since \(H=F(SK)\) is compact, Corollary~\ref{prop:operator-UAP-input-bound}(1) and the choice of \(n_L\) give \(\widetilde L\) with the stated norm and rank bounds and error less than \(\varepsilon_{\rm out}\).
    By the same corollary and the choice of \(n_S\), we can select \(\widetilde S\) with the stated norm and rank bounds and error less than \(\delta_S\).
    By the preceding Auerbach construction, we obtain the coordinate maps and their norm bounds.
    Since \(\norm{\widetilde L}\omega_F(\delta_S)\leq\Lambda_L\norm L\,\omega_F(\delta_S)<\varepsilon_{\rm in}\), the error decomposition in Section~\ref{subsec:operator-to-functional-reduction} gives the result.
\end{proof}

For \(S=I_E\) and \(L=I_X\), this theorem gives the corresponding space UAP estimate.
By a direct calculation and Proposition~\ref{prop:polynomial-rate-inversion}, we obtain the following rank bounds.
\begin{corollary}\label{cor:uap-operator-polynomial-ranks}
    Assume the hypotheses of Theorem~\ref{thm:UAP-profile-operator-template} with \(\omega_F(t)=B_Ft^\beta\), where \(B_F>0\) and \(0<\beta\leq1\).
    Suppose that, for every \(j\in\N_+\), we have
    \[ d_j(K)\leq A_Kj^{-\sigma_S},\quad k_S(\Lambda_S,j)\leq B_Sj^{\gamma_S},\quad d_j(H)\leq A_Hj^{-\sigma_L},\quad k_L(\Lambda_L,j)\leq B_Lj^{\gamma_L}, \]
    where \(H=F(SK)\) and all displayed constants are positive.
    Put
    \begin{align*}
        \overline n_S=\left\lceil\left[4(1+\Lambda_S)\norm S A_K\left(\frac{\Lambda_L\norm L B_F}{\varepsilon_{\rm in}}\right)^{1/\beta}\right]^{1/\sigma_S}\right\rceil,\qquad
        \overline n_L=\left\lceil\left[\frac{2(1+\Lambda_L)\norm L A_H}{\varepsilon_{\rm out}}\right]^{1/\sigma_L}\right\rceil.
    \end{align*}
    Then \(n_S\leq\overline n_S\), \(n_L\leq\overline n_L\), and \(m\leq B_S\overline n_S^{\gamma_S}\), \(q\leq B_L\overline n_L^{\gamma_L}\).
\end{corollary}

The compact set \(H=F(SK)\) and the output rank bound are independent of the input recovery.

\subsection{Neural parameter bounds for the reduced map}\label{subsec:neural-parameter-bounds}\label{subsec:gauge-baseline-functional-learning}

We use the factorizations of \(\widetilde S\) and \(\widetilde L\) to choose coordinates for the network.
Since \(\Phi_F=VFU\), we have \(\norm{\Phi_F(z)-\Phi_F(w)}_\infty\leq\norm V\,\omega_F(\norm U\,\norm{z-w}_\infty)\).
The coordinate norms may depend on the dimensions; they control the radius of \(CK\), the gauge of \(\Phi_F\), and the decoded error.
We apply the gauge-preserving extension and quantitative ReLU estimate in Appendix~\ref{app:reduced-map-modules} to approximate each coordinate.

Write \(\Phi_F=(\phi_1,\ldots,\phi_q)\), put \(D_m=CK\), and let \(\kappa_{q,2}=\norm{R:\ell_2^q\to X}\leq\sqrt q\).
Assume \(m,q\geq1\) and \(\kappa_{q,2}>0\) in the parameter formulas below.
Each coordinate of \(\Phi_F\) has the global gauge \(\Omega(t)=\Lambda_L\norm L\,\omega_F(mt)\).
If \(\omega_F(t)=B_Ft^\beta\), we may use the common H\"older constant \(B_\Phi=\Lambda_L\norm L\,B_Fm^\beta\).
In the formulas below, \(Q_m(B)=[-B,B]^m\), \(R_K=\sup_{x\in K}\norm x_E\), and \(a_m,c_m,W_m\) are the constants from Proposition~\ref{prop:quantitative-continuous-relu}.

Let \(\eta_j\) denote the scalar network errors.
By Proposition~\ref{prop:master-coordinate-complexity}, we obtain a bound for the decoded error.
We sum the parameter counts of the scalar networks.

The parameter counts below include all scalar weight and bias slots in the inner ReLU networks and their combination into a vector-valued network, including slots whose chosen values are zero.
For prescribed parameter bounds for each coordinate, the inner graph is fixed independently of the target functions.
We count the parameters used to represent \(C,U,V,R\) only when stated.

We use \(\mathcal R(t)=\min\{w\in\N_+:w\log(e+w)\geq t\}\) from \eqref{eq:relu-budget-inversion}.
Suppose \(D_m\subseteq Q_m(B)\), and let \(\eta_j>0\) be the coordinate errors.
Let \(\Omega_j\) be a vanishing gauge of \(\phi_j\) on \(Q_m(B)\).
By \eqref{eq:relu-modulus-estimate}, we have the total parameter bound
\[ \sum_{j=1}^q \max\left\{ W_m,\, \mathcal R\!\left( \left( \frac{2Bc_m}{\delta_{\Omega_j}(\eta_j/a_m)} \right)^m \right) \right\}. \]
For each coordinate, this parameter bound makes the argument of \(\Omega_j\) in \eqref{eq:relu-modulus-estimate} at most \(\delta_{\Omega_j}(\eta_j/a_m)\).
Hence the coordinate error is smaller than \(\eta_j\).
If the coordinates are \(\beta\)-H\"older with common constant \(B_\Phi\) and errors \(\eta_j=\eta/(\kappa_{q,2}\sqrt q)\), then by \eqref{eq:relu-modulus-estimate} it is enough to use
\[ q\max\left\{ W_m,\, \mathcal R\!\left( \left( \frac{ a_mB_\Phi(2Bc_m)^\beta\kappa_{q,2}\sqrt q }{\eta} \right)^{m/\beta} \right) \right\} \]
parameters.
Put \(t=(\kappa_{q,2}\sqrt q/\eta)^{m/\beta}\).
By \eqref{eq:relu-budget-inversion}, the parameter bound for H\"older coordinates is \(O_m(q(1+t/\log(e+t)))\) when \(B,B_\Phi,\beta\) are fixed.

In Theorem~\ref{thm:UAP-profile-operator-template}, we take \(B=\Lambda_S\norm S R_K\) and select \(m\) and \(q\) from widths and UAP rank functions.
To obtain a bound involving only the final accuracy, we also need to control the dimension-dependent constants in the finite-dimensional approximation theorem.

In Proposition~\ref{prop:master-coordinate-complexity}, we can also use a scalar Barron approximation bound, provided that the chosen scalar coordinate of the reduced map has the required representation and the bound controls its uniform error.
Section~\ref{subsec:barron-profiles-vector-operators} treats direct representations of \(T\).

\subsection{Barron representations and sampling bounds}\label{subsec:barron-profiles-vector-operators}

For a target with a Barron representation, we construct an encoder--network--decoder approximation by sampling its input functionals and output vectors.
This construction does not require reconstructing the input compact set first.
We first describe the representation space and its norm, then estimate the number of samples and the resulting error.

Classical Barron and variation spaces are studied in \cite{Barron1993,MaWuE2022,SiegelXu2023Variation}.
Korolev~\cite{Korolev2021} considers Banach-valued networks with lattice positive-part activation.
Dummer, Heeringa, and Iglesias~\cite{DummerHeeringaIglesias2025} develop vector-valued integral representations and representer theorems.
Yang and Shi~\cite{YangShi2026} study approximation and generalization for encoder--decoder variation spaces between Hilbert spaces.

Let \(\sigma(t)=\max\{t,0\}\) and \(\mathcal P_X=B_{X^*}^{w^*}\times[-1,1]\).
\begin{definition}\label{def:intrinsic-vector-barron}
    An operator \(T:X\to Y\) is a vector-valued Barron operator with a Bochner-integrable density if, for a finite positive Radon measure \(\nu\) on \(\mathcal P_X\), we have \(T(x)=\int_{\mathcal P_X}\sigma(\ell(x)+b)y(\ell,b)\,d\nu(\ell,b)\) for \(x\in X\), where \(y:\mathcal P_X\to Y\) is Bochner measurable and \(\int_{\mathcal P_X}\norm{y(\theta)}_Y\,d\nu(\theta)<\infty\).
    We denote the class of these operators by \(\mathcal B_X(Y)\), and define
    \[ \norm{T}_{\mathcal B_X(Y)} =\inf_{\nu,y}\left\{ \int_{\mathcal P_X}\norm{y(\theta)}_Y\,d\nu(\theta): T(x)=\int_{\mathcal P_X}\sigma(\ell(x)+b)y(\ell,b)\,d\nu(\ell,b),\ x\in X \right\}. \]
\end{definition}

The nuclear representation below identifies the linear structure and completeness of \(\mathcal B_X(Y)\).
It also gives a vector-measure description when the output space has the Radon--Nikod\'ym property.
For \(x\in X\), put \(\varphi_x(\ell,b)=\sigma(\ell(x)+b)\).
By the representation theorem for nuclear operators \cite[Proposition~5.30]{Ryan2002}, \(T\in\mathcal B_X(Y)\) if and only if \(T(x)=Q\varphi_x\) for some nuclear operator \(Q:C(\mathcal P_X)\to Y\), and \(\norm{T}_{\mathcal B_X(Y)}=\inf\{\norm Q_{\mathrm{nuc}}:Q\varphi_x=T(x)\text{ for all }x\in X\}.\)
Since \(\norm{\varphi_x}_\infty\leq1+\norm x_X\), the map \(Q\mapsto(x\mapsto Q\varphi_x)\) has closed kernel.
Hence \(\mathcal B_X(Y)\) is a Banach quotient of the nuclear operator space.
If \(Y\) has the Radon--Nikod\'ym property, the same norm is \(\inf|\mu|(\mathcal P_X)\) over all \(Y\)-valued Radon measures of finite variation with \(T(x)=\int_{\mathcal P_X}\varphi_x\,d\mu\).

Proposition~\ref{prop:scalar-barron-finite-restrictions} describes the scalar representation norm through finite restrictions and its behaviour under linear reductions.
These properties allow finite-dimensional restrictions to identify a global representation with a controlled norm.
For a finite neural realization, the sampling estimate below supplies a bound on the number of hidden neurons and norm bounds for the resulting encoder and decoder.

The space \(Y\) has type \(p_Y\in(1,2]\) with constant \(T_{p_Y}(Y)\) if \[\mathbb E_\varepsilon\norm{\sum_{i=1}^N\varepsilon_i v_i}_Y^{p_Y} \leq T_{p_Y}(Y)^{p_Y}\sum_{i=1}^N\norm{v_i}_Y^{p_Y}.\]
The following proposition applies the standard sampling argument in spaces of nontrivial type; see~\cite{DonahueGurvitsDarkenSontag1997} and~\cite[Theorem~4.5]{NeufeldSchmocker2023}.
The uniform estimate also uses the bounded differences inequality~\cite{McDiarmid1989}.
\begin{proposition}\label{thm:type-p-vector-barron-sampling}
    Let \(K\Subset X\) be nonempty, and let \(Y\) have type \(p_Y\) with type constant \(T_{p_Y}(Y)\).
    Let \((\nu,y)\) be a representation of \(T:X\to Y\) with a Bochner-integrable density, and put \(B_0=\int_{\mathcal P_X}\norm{y(\theta)}_Y\,d\nu(\theta)>0\) and \(R_K=\sup_{x\in K}\norm x_X\).
    Define \(d\pi(\theta)=B_0^{-1}\norm{y(\theta)}_Y\,d\nu(\theta)\) and \(u(\theta)=y(\theta)/\norm{y(\theta)}_Y\) on \(\{y\ne0\}\).
    For independent samples \(\theta_i=(\ell_i,b_i)\) from \(\pi\), set
    \[ T_N(x)=\frac{B_0}{N}\sum_{i=1}^N\sigma(\ell_i(x)+b_i)u(\theta_i). \]
    With \(C_Nx=(\ell_i(x))_{i=1}^N\), \(\mathcal N_N(z)=(\sigma(z_i+b_i))_{i=1}^N\), and \(D_Na=(B_0/N)\sum_{i=1}^Na_i u(\theta_i)\), we have almost surely \(T_N=D_N\mathcal N_N C_N\), \(\norm{C_N:X\to\ell_\infty^N}\leq1\), and \(\norm{D_N:\ell_\infty^N\to Y}\leq B_0.\)
    For every Borel probability measure \(\lambda\) on \(K\), we have
    \[ \left(\mathbb E\norm{T-T_N}_{L^{p_Y}(\lambda;Y)}^{p_Y}\right)^{1/p_Y} \leq2T_{p_Y}(Y)B_0(R_K+1)N^{-(1-1/p_Y)}. \]
    For a fixed input, take \(\lambda=\delta_x\).
    For every \(\delta>0\), one deterministic sample satisfies
    \begin{align*}
        \sup_{x\in K}\norm{T(x)-T_N(x)}_Y \leq B_0(R_K+1)\left[ 2T_{p_Y}(Y)N^{-(1-1/p_Y)} +\sqrt{\frac{2\log\bigl(2N(K,(\delta/2)B_X)\bigr)}{N}}\right]+2B_0\delta.
    \end{align*}
\end{proposition}

The proof is given in Appendix~\ref{app:barron-representations}.
For uniform error at most \(\eta>0\), we choose \(\delta=\eta/(6B_0)\) in the preceding bound.
It is enough that
\[ N\geq\max\left\{ \left(\frac{6T_{p_Y}(Y)B_0(R_K+1)}{\eta}\right)^{p_Y/(p_Y-1)}, \frac{18B_0^2(R_K+1)^2}{\eta^2} \log\left(2N\left(K,\frac{\eta}{12B_0}B_X\right)\right) \right\}. \]
Here \(N\) is the number of hidden neurons.
The scalar parameter count also depends on how we represent the sampled functionals and output vectors.

The next section gives common reductions for operator classes and parameter lower bounds for their finite-dimensional neural approximations.

    \section{Common reductions and parameter bounds for operator classes}\label{sec:compact-factorization-lower-bound-gauges}

We study the operators \(T_F=LFS:E\to X\), where \(S:E\to Z\) and \(L:Y\to X\) are bounded linear maps between real Banach spaces and \(F\) belongs to a class \(\mathcal F\subseteq C(Z,Y)\).
For fixed \(K\) and accuracy, we require the maps \(C,U,V,R\) and the network graph to be the same for every \(F\in\mathcal F\), with weights and biases depending on \(F\).
We first choose common reductions using a common gauge and a compact set containing \(F(SK)\) for every \(F\in\mathcal F\).
We then use scalar approximation estimates to choose the common network graph.
Input packing and independent H\"older coordinates then give parameter lower bounds for the class.
Throughout, \(K\subseteq E\) is nonempty and bounded; compactness is assumed when needed.
\subsection{Common linear reductions for operator classes}\label{subsec:common-class-reductions}

The next corollary follows from the construction in Theorem~\ref{thm:UAP-profile-operator-template}.
\begin{corollary}\label{prop:factorized-uap-upper}
    Let \(K\Subset E\), and suppose that the maps in \(\mathcal F\) admit a common global vanishing gauge \(\omega\).
    Assume that \(H=\overline{\bigcup_{F\in\mathcal F}F(SK)}\) is compact.
    Let \(\Lambda_S,\Lambda_L\geq1\), \(n,r\in\N_+\), and \(\delta>0\), with \(k_S(\Lambda_S,n)<\infty\) and \(k_L(\Lambda_L,r)<\infty\).
    There is a finite-rank map \(\widetilde L=RV\) factoring through \(\R^q\), such that \(q\leq k_L(\Lambda_L,r)\), \(\norm{\widetilde L}\leq\Lambda_L\norm L\), and
    \[ \sup_{y\in H}\norm{Ly-\widetilde Ly}_X <(1+\Lambda_L)\norm L\,d_r(H;Y)+\delta. \]
    There is a finite-rank map \(\widetilde S=UC\) factoring through \(\R^m\), such that \(m\leq k_S(\Lambda_S,n)\), \(\norm{\widetilde S}\leq\Lambda_S\norm S\), and
    \[ \sup_{x\in K}\norm{Sx-\widetilde Sx}_Z <(1+\Lambda_S)\norm S\,d_n(K;E)+\delta. \]
    The coordinate maps can be chosen so that \(\norm C\leq\Lambda_S\norm S\), \(\norm U\leq m\), \(\norm V\leq\Lambda_L\norm L\), and \(\norm R\leq q\), where the coordinate spaces have their \(\ell_\infty\)-norms.
    For \(F\in\mathcal F\), put \(\Phi_F=VFU\).
    Then every map \(N:\R^m\to\R^q\) satisfies
    \begin{align*}
        &\sup_{x\in K}\norm{T_F(x)-RN(Cx)}_X\\
        \leq&\Lambda_L\norm L\,\omega\!\left((1+\Lambda_S)\norm S\,d_n(K;E)+\delta\right)
        +(1+\Lambda_L)\norm L\,d_r(H;Y)+\delta+\sup_{z\in CK}\norm{R(\Phi_F(z)-N(z))}_X.
    \end{align*}
\end{corollary}
The proof is given in Appendix~\ref{app:lower-bound-machinery}.

For \(m,q\geq1\), each scalar coordinate of \(\Phi_F\) has the common gauge \(t\mapsto\Lambda_L\norm L\,\omega(mt)\), and \(CK\subseteq[-\Lambda_S\norm S R_K,\Lambda_S\norm S R_K]^m\), where \(R_K=\sup_{x\in K}\norm x_E\).
Thus the scalar estimates in Section~\ref{sec:operator-learning-nn-complexity} give one network graph for the reduced maps at each accuracy, with only its weights and biases depending on \(F\).
An error bound \(\varepsilon>0\) follows if each of the three error terms is at most \(\varepsilon/3\).

The same argument applies to the absolute norm bounds in Definition~\ref{def:operator-ideal-gauge}.
For \(C_S>\mathbf u_f(S)\) and \(C_L>\mathbf u_g(L)\), we can choose \(m\leq f(n)\), \(q\leq g(r)\), \(\norm{\widetilde S}\leq C_S\), and \(\norm{\widetilde L}\leq C_L\).
The resulting reduction error satisfies
\begin{align*}
    \sup_{x\in K}\norm{T_F(x)-\widetilde L F(\widetilde Sx)}_X
    &\leq C_L\,\omega\!\left((\norm S+C_S)d_n(K;E)+\delta\right)+(\norm L+C_L)d_r(H;Y)+\delta.
\end{align*}

\subsection{Parameter lower bounds from input packing}\label{subsec:factorized-parameter-lower}

For bit-complexity lower bounds for the full class of Lipschitz operators, see~\cite{Lanthaler2024Information}.

\begin{definition}\label{def:covering-packing-entropy}
    Let \((M,d)\) be a metric space and let \(A\subseteq M\).
    We denote by \(\mathsf N(A,d,\varepsilon)\) the least cardinality of closed \(d\)-balls of radius \(\varepsilon\) needed to cover \(A\).
    We denote by \(\mathsf P(A,d,\varepsilon)\) the largest cardinality of an \(\varepsilon\)-separated subset of \(A\).
\end{definition}
Trivially, for any compact \(A\), \(\mathsf P(A,d,3\varepsilon) \leq \mathsf N(A,d,\varepsilon) \leq \mathsf P(A,d,\varepsilon).\)
\begin{definition}\label{def:pseudo-dimension}
    Let \(\mathcal F\subseteq\R^\Omega\) be a real-valued function class.
    The class \(\mathcal F\) pseudo-shatters \(x_1,\ldots,x_n\in\Omega\) if there are thresholds \(b_1,\ldots,b_n\in\R\) such that, for every \(\theta\in\{0,1\}^n\), there is \(f_\theta\in\mathcal F\) satisfying
    \[ f_\theta(x_j)>b_j\ (\theta_j=1),\quad \text{while} \quad f_\theta(x_j)<b_j\ (\theta_j=0). \]
    The pseudo-dimension \(\operatorname{Pdim}(\mathcal F)\) is the supremum of such \(n\).
\end{definition}

We use ReLU networks with vector-valued affine outputs \cite[Definition~3.1]{KutyniokPetersenRaslanSchneider2022} and the fixed graph convention of \cite[Section~1.3]{BartlettHarveyLiawMehrabian2019}.
Let \(\mathcal{NN}_{W,D}\) consist of maps \(N_\theta=R\Phi_\theta C:K\to X\), where \(C:E\to\R^a\) and \(R:\R^b\to X\) are fixed bounded linear maps.
The maps \(\Phi_\theta:\R^a\to\R^b\) have one fixed graph of depth at most \(D\) and at most \(W\) trainable weights and biases, including those assigned the value zero.
For the lower bounds, we impose the following condition.

\begin{assumption}\label{ass:scalar-readout-realization}
    Let \(r_1,\ldots,r_m\in X^*\) be fixed, and put
    \[ \mathcal{NN}_{W,D}^{\mathrm{scal}} =\{(i,u)\mapsto r_i(N(u)):N\in\mathcal{NN}_{W,D}\} \]
    on \(\{1,\ldots,m\}\times K\).
    There exists a constant \(C_{\rm arch}\geq1\), common to all network classes under consideration and independent of \(m,W,D,\tau,\varepsilon\), with the following property.
    After a fixed finite-dimensional encoding of \(u\), this class admits a representation by one scalar ReLU graph with at most \(C_{\rm arch}W\) weights and biases in total, of which at most \(W\) are trainable.
    The graph has at most \(C_{\rm arch}W\) computation units and depth at most \(C_{\rm arch}D\).
    The graph and encoding are independent of \(\theta\).
\end{assumption}

Since \(r_i\circ R\) is a fixed linear functional on \(\R^b\), we can compute \(r_i(N_\theta(u))\) from the finite-dimensional network output by a fixed linear map.
This assumption requires one scalar ReLU graph to do so with \(i\) as an additional input, with size \(O(W)\) and depth \(O(D)\).

By \cite{BartlettHarveyLiawMehrabian2019}, we obtain the following estimate.
\begin{lemma}\label{lem:scalarized-readout-pdim}
    Under Assumption~\ref{ass:scalar-readout-realization}, we have
    \[ \operatorname{Pdim}(\mathcal{NN}_{W,D}^{\mathrm{scal}}) \leq C_{\rm pd}WD(1+\log W), \]
    where \(C_{\rm pd}\) depends only on \(C_{\rm arch}\).
\end{lemma}
The proof is given in Appendix~\ref{app:lower-bound-machinery}.

The lower estimate uses targets with independently varying values on separated inputs.
The output factor and the scalar readouts preserve this independence in the parameter-counting argument.

For a metric space \((A,d)\), a vanishing gauge \(\omega\), and \(B>0\), put
\[ \mathcal H_{\omega,B}(A)=\{h\in C(A):\norm{h}_\infty\leq B,\ \omega_h\leq\omega\}. \]

For the lower bounds below, fix \(0<\beta\leq1\) and take \(\omega(t)=Bt^\beta\).
We choose a totally bounded set \(A\subseteq S(K)\), with the metric inherited from \(Z\).
Fix \(q\in\N_+\), a bounded linear map \(J:\ell_\infty^q\to Y\), and \(\tau,C_+,C_->0\) such that
\begin{equation}\label{eq:factor-output-bounds}
    \frac{\tau}{C_-}\norm v_\infty\leq\norm{LJv}_X \leq C_+\tau\norm v_\infty,\qquad v\in\ell_\infty^q.
\end{equation}
By the Hahn--Banach theorem applied to the inverse coordinate functionals on \(LJ(\ell_\infty^q)\), we can fix \(\Pi:X\to\ell_\infty^q\) with \(\Pi LJ=I\) and \(\norm\Pi\leq C_-/\tau\).

\begin{theorem}\label{thm:factorized-operator-lower-bound}
    Suppose that, for every \(h_1,\ldots,h_q\in\mathcal H_{\omega,B}(A)\), there is \(F_h\in\mathcal F\) satisfying
    \[ F_h(z)=J(h_1(z),\ldots,h_q(z)) \qquad\text{for }z\in A. \]
    Let \(\mathcal{NN}_{W,D}\) satisfy Assumption~\ref{ass:scalar-readout-realization} for \(r_i=e_i^*\Pi\), \(1\leq i\leq q\), and uniformly \(\varepsilon\)-approximate \(\{T_F:F\in\mathcal F\}\) on \(K\).
    Let \(c_\beta>0\) be the constant in Lemma~\ref{lem:metric-holder-lower-gauge}.
    If
    \[ 0<\rho(\varepsilon):= \left(\frac{2C_-\varepsilon}{c_\beta B\tau}\right)^{1/\beta}\leq1, \]
    then
    \[ WD(1+\log W)\geq\frac{q}{2C_{\rm pd}} \mathsf P\left(A,\norm{\cdot}_Z,\rho(\varepsilon)\right). \]
\end{theorem}
The proof is given in Appendix~\ref{app:lower-bound-machinery}.

\subsection{Matching parameter bounds for fixed factors}\label{subsec:fixed-factor-rates}

Here \(S\) and \(LJ\) give a fixed finite-dimensional representation of every target in the class.
We estimate the number of parameters needed for the full class of H\"older nonlinearities with this representation.
Under the stated factor and architecture assumptions, the upper and lower bounds have the same order at any fixed sufficiently large depth.

\begin{assumption}\label{ass:factorized-holder-cube}
    Let \(d\in\N_+\), \(\ell>0\), and \(Z=\ell_\infty^d\).
    \begin{enumerate}
        \item[(i)] The input image satisfies \(S(K)=Q_d(\ell)=[-\ell,\ell]^d\).
        \item[(ii)] The output map \(LJ\) satisfies \eqref{eq:factor-output-bounds}.
        \item[(iii)] For every \(g=(g_1,\ldots,g_q)\in\mathcal H_{\omega,B}(Q_d(\ell))^q\), extend each \(g_i\) by Lemma~\ref{lem:scalar-infimal-extension} with \(\omega(t)=Bt^\beta\).
        Clip these extensions to \([-B,B]\) and denote the resulting functions by \(\widetilde g_i\).
        Put \(F_g(z)=J(\widetilde g_1(z),\ldots,\widetilde g_q(z))\).
        The target class is \(\mathfrak T_{S,L,J}^{\beta,B,\square}=\{T_{F_g}:g\in\mathcal H_{\omega,B}(Q_d(\ell))^q\}\), where \(T_{F_g}(x)=LJg(Sx)\) on \(K\).
        \item[(iv)] The encoder \(S\), decoder \(LJ\), and readouts \(e_i^*\Pi\) are fixed.
        The networks satisfy Assumption~\ref{ass:scalar-readout-realization} with one fixed \(C_{\rm arch}\), chosen large enough to include parallel combinations of the scalar networks in Proposition~\ref{prop:quantitative-continuous-relu}, clipped to \([-B,B]\).

    \end{enumerate}
\end{assumption}

Put \(\widehat L_d=L_d+6\).
For \(D\in\N_+\), let \(W_D^\square(\varepsilon)\) be the least positive integer \(W\) for which there is a fixed graph in this network class, with at most \(W\) trainable weights and biases and depth at most \(D\), such that every \(T\in\mathfrak T_{S,L,J}^{\beta,B,\square}\) is uniformly \(\varepsilon\)-approximated on \(K\) by one of its realizations.
\begin{theorem}\label{thm:factorized-holder-cube-sharp}
    Under Assumption~\ref{ass:factorized-holder-cube}, put \(\mathcal B(\varepsilon)=q[1+\ell^d(B\tau/\varepsilon)^{d/\beta}]\).
    If \(D\geq\widehat L_d\) and \(0<\varepsilon\leq c_\beta B\tau/(2C_-)\), then
    \begin{equation}\label{eq:factorized-holder-cube-bounds}
        c\frac{\mathcal B(\varepsilon)}{D(1+\log(e+\mathcal B(\varepsilon)))} \leq W_D^\square(\varepsilon) \leq C\frac{\mathcal B(\varepsilon)}{\log(e+\mathcal B(\varepsilon)/q)}.
    \end{equation}
    The constants depend only on \(d,\beta,C_+,C_-,C_{\rm pd}\), the fixed architecture convention, and the constants in Proposition~\ref{prop:quantitative-continuous-relu} for dimension \(d\).
\end{theorem}
The proof is given in Appendix~\ref{app:lower-bound-machinery}.

For fixed factors, \(\beta,B\), and \(D\geq\widehat L_d\), we have
\[ W_D^\square(\varepsilon)\asymp \frac{\varepsilon^{-d/\beta}}{\log(1/\varepsilon)} \qquad(\varepsilon\downarrow0). \]

In Appendix~\ref{app:detailed-operator-learning-examples}, we apply the same error decomposition to \(T_F=LFS\) with \(F(v)=\varphi\circ v\).
The regularity of \(S,L\) determines the recovery scales, and the Lipschitz constant of \(\varphi\) bounds the effect of the recovery errors.
Conditional expectations provide stable linear measurements and reconstruction.
The pointwise form of \(F\) allows us to obtain the network bound from parallel one-dimensional approximations.

    \appendix

\section{Point sampling, cell averages, and Gabor truncation}\label{app:classical-realizations}

The constructions below make the recovery estimates of Section~\ref{sec:unit-ball-UAP-coders} available through specified measurements.
For each scheme, we track the number of coordinates, the reconstruction norm, and the approximation error needed in the operator-learning estimates.

\subsection{Stable recovery by point sampling in \texorpdfstring{\(C(K)\)}{C(K)}}\label{subsec:unit-ball-coder-tools}
Here the measurements are restricted to point evaluations.
We ask how many samples are needed to fix a prescribed finite-dimensional space while keeping the reconstruction norm below a chosen bound.

\begin{theorem}\label{thm:finite-sampling-retraction}
    For every \(\Gamma>1\), there is a function \(N_\Gamma:\N\to\N\), with \(N_\Gamma(0)=0\), such that the following holds for every compact Hausdorff space \(K\).
    For every subspace \(E\subseteq C(K)\) of dimension \(d<\infty\), there are points \(t_1,\ldots,t_m\in K\), with \(m\leq N_\Gamma(d)\), and a bounded linear map \(D_E:\ell_\infty^m\to C(K)\) such that
    \[ D_EC_E|_E=I_E,\qquad \norm{D_E}\leq\Gamma,\qquad C_Ef=(f(t_1),\ldots,f(t_m)). \]
    Moreover, we can take \(N_\Gamma(d)\leq C_\Gamma^d\), where \(C_\Gamma>1\) depends only on \(\Gamma\).

    For fixed \(K\), such a finite sampling factorization with \(\norm{D_E}\leq1\) exists for every finite-dimensional \(E\subseteq C(K)\) if and only if \(K\) is finite.
\end{theorem}

\begin{proof}
    Let \(\Gamma>1\).
    Suppose that \(\dim E=d\geq1\), and let \(0<\delta<1\).
    By taking a maximal \(\delta\)-separated subset \(\{f_1,\ldots,f_m\}\) of the unit sphere, we obtain a \(\delta\)-net.
    By the disjoint-ball volume estimate, we have \(m\leq(1+2/\delta)^d\).
    Choose \(t_j\in K\) with \(\abs{f_j(t_j)}=1\).
    For every \(f\in E\) with \(\norm{f}_\infty=1\), we can choose \(j\) with \(\norm{f-f_j}_\infty\leq\delta\).
    Then \(\abs{f(t_j)}\geq1-\delta\).
    By homogeneity, we have
    \[ \norm{f}_\infty\leq(1-\delta)^{-1}\max_{1\leq j\leq m}\abs{f(t_j)},\qquad f\in E. \]
    Thus \(C_E|_E\) is an isomorphic embedding.
    Put \(Y=C_E(E)\subseteq\ell_\infty^m\) and \(U=(C_E|_E)^{-1}:Y\to C(K)\).
    Then \(\norm{U}\leq(1-\delta)^{-1}\).

    Let \(r:\ell_1^m=(\ell_\infty^m)^*\to Y^*\) be restriction.
    By the Hahn--Banach theorem, we have \(r(B_{\ell_1^m})=B_{Y^*}\), for closed unit balls.
    Hence \(Y^*\) is isometric to \(\ell_1^m/\ker r\) with its quotient norm.
    By the Bartle--Graves theorem \cite{BartleGraves1952,BenyaminiLindenstrauss2000}, for every \(\varepsilon>0\), we can choose a continuous positively homogeneous map \(s:Y^*\to\ell_1^m\) such that \(rs=I_{Y^*}\) and \(\norm{s}\leq(1+\varepsilon).\)
    Since \(Y\) is finite-dimensional, we know that \(t\mapsto U^*\delta_t\) is norm continuous, because its coordinates in any basis of \(Y\) are continuous functions of \(t\).
    We write \(s(U^*\delta_t)=(w_1(t),\ldots,w_m(t))\) and define
    \[ (D_Ea)(t)=\sum_{j=1}^m a_jw_j(t),\qquad a\in\ell_\infty^m,\quad t\in K. \]
    Since the functions \(w_j\) are independent of \(a\), \(D_E\) is linear.
    Moreover, we have
    \[ \norm{D_E}\leq\sup_{t\in K}\sum_{j=1}^m\abs{w_j(t)} \leq(1+\varepsilon)\norm{U}\leq\frac{1+\varepsilon}{1-\delta}. \]
    For every \(f\in E\), by restriction to \(Y\), we have \((D_EC_Ef)(t)=\langle C_Ef,rs(U^*\delta_t)\rangle =\langle UC_Ef,\delta_t\rangle=f(t).\)
    For the same operator on the whole ambient space, we have
    \[ (D_EC_Ef)(t)=\sum_{j=1}^m f(t_j)w_j(t),\qquad f\in C(K), \]
    and \(\norm{D_EC_E}\leq\norm{D_E}\), since \(\norm{C_E}\leq1\).
    Let \(\delta=(\Gamma-1)/(2\Gamma)\) and \(\varepsilon=(\Gamma-1)/2\).
    Then \((1+\varepsilon)/(1-\delta)=\Gamma\).
    By the bound for the net, we obtain the claimed estimate with \(C_\Gamma=1+4\Gamma/(\Gamma-1)\), \(N_\Gamma(d)=\lfloor C_\Gamma^d\rfloor\) for \(d\geq1\), and \(N_\Gamma(0)=0\).

    Suppose that the norm-one assertion holds and that \(K\) is infinite.
    For \(n\geq1\), let \(F_n\subseteq C(K)\) consist of functions taking at most \(n\) values.
    Each \(F_n\) is closed and has empty interior, since the Tietze extension theorem allows arbitrarily small perturbations with distinct values at \(n+1\) prescribed distinct points.
    By the Baire category theorem, we can choose \(h\in C(K)\) with infinite range.

    Apply the assertion to \(E=\Span\{1,h,h^2\}\), and put \(P=D_EC_E\) and \(p_j(t)=(D_Ee_j)(t)\).
    Since \(P1=1\) and \(\norm{D_E}\leq1\), we have \(\sum_jp_j(t)=1\) and \(\sum_j\abs{p_j(t)}=\norm{\delta_tD_E}\leq1\), so \(p_j(t)\geq0\).
    Since \(P\) fixes \(h\) and \(h^2\), we obtain
    \[ 0=P\bigl((h-h(t))^2\bigr)(t) =\sum_{j=1}^m p_j(t)(h(t_j)-h(t))^2. \]
    Hence \(h(K)\subseteq\{h(t_1),\ldots,h(t_m)\}\), a contradiction.
    Conversely, if \(K\) is finite, evaluation at all its points identifies \(C(K)\) isometrically with \(\ell_\infty^{|K|}\), and the inverse map gives the required decoder.
\end{proof}

Since \((I-D_EC_E)|_E=0\) and \(\norm{D_EC_E}\leq\Gamma\), we have \(\norm{f-D_EC_Ef}_\infty\leq(1+\Gamma)\operatorname{dist}(f,E)\) for every \(f\in C(K)\).

\begin{remark}
    For every \(\Gamma>1\), we use the operators \(T_E=D_EC_E\) to obtain the UAP of real \(C(K)\) in sampling form: \(T_E|_E=I_E\), \(\norm{T_E}\leq\Gamma\), and \(\operatorname{rank}T_E\leq N_\Gamma(\dim E)\).
    Their ranges contain \(E\) but need not equal \(E\).
\end{remark}

\subsection{Finite-dimensional recovery by conditional expectations}\label{subsec:Lp-conditional-expectation-coders}

Conditional expectations fix piecewise constant functions, and we can correct them to fix a prescribed finite-dimensional \(E\).
A smoothness bound will give uniform control of the number of partition atoms.

We use cell averages to define the basic conditional-expectation coder on \(L_p\).
We write \(p'\) for the conjugate exponent, that is, \(p^{-1}+p'^{-1}=1\).
When \(p=1\), we use \(\mu(A)^{-1/p'}=1\).

\begin{proposition}\label{prop:Lp-conditional-expectation-coder}
    Let \((\Omega,\mu)\) be a finite measure space, let \(1\leq p<\infty\), and let \(\mathcal P=\{A_1,\dots,A_m\}\) be a finite measurable partition of \(\Omega\) with \(0<\mu(A_j)<\infty\).
    Define
    \[ C_{\mathcal P}f = \left(\mu(A_j)^{-1/p'}\int_{A_j}f\,\rmd\mu\right)_{j=1}^m, \qquad D_{\mathcal P}a = \sum_{j=1}^m\mu(A_j)^{-1/p}a_j\mathbf 1_{A_j}. \]
    Then \(C_{\mathcal P}:L_p(\Omega)\to\ell_p^m\) and \(D_{\mathcal P}:\ell_p^m\to L_p(\Omega)\) satisfy \(\norm{C_{\mathcal P}}\leq1\text{ and } \norm{D_{\mathcal P}}=1,\) and
    \[ D_{\mathcal P}C_{\mathcal P}f = \sum_{j=1}^m \frac{1}{\mu(A_j)} \left(\int_{A_j}f\,\rmd\mu\right)\mathbf 1_{A_j} = \mathbb E[f\mid\sigma(\mathcal P)]. \]
    If \(S_{\mathcal P}=\Span\{\mathbf 1_{A_1},\dots,\mathbf 1_{A_m}\},\) then \(D_{\mathcal P}C_{\mathcal P}f=f\) for every \(f\in S_{\mathcal P}\).
\end{proposition}

\begin{proof}
    By H\"older's inequality on each atom, we have \(\norm{C_{\mathcal P}f}_{\ell_p^m}^p\leq\sum_j\int_{A_j}\abs f^p\,\rmd\mu=\norm f_{L_p}^p\).
    Since the supports are disjoint, we have \(\norm{D_{\mathcal P}a}_{L_p}^p=\sum_j\abs{a_j}^p\).
    By the definitions, \(D_{\mathcal P}C_{\mathcal P}\) is the conditional expectation, and thus fixes \(S_{\mathcal P}\).
\end{proof}

We write \(\mathbb E[f\mid\mathcal P]\) for the finite conditional expectation \(D_{\mathcal P}C_{\mathcal P}f\).
By the density of simple functions, we obtain approximation on a fixed finite-dimensional subspace.
We then apply the finite-dimensional correction to obtain exact agreement on that subspace.

\begin{proposition}\label{prop:Lp-fixed-subspace-conditional-expectation}
    Let \((\Omega,\mu)\) be a finite measure space, let \(1\leq p<\infty\), and let \(E\subseteq L_p(\Omega)\) be \(n\)-dimensional, where \(n\geq1\).
    For every \(\varepsilon>0\), there are a finite partition \(\mathcal P\), with \(m\) atoms, and \(A_E\in\cF(L_p(\Omega))\) such that
    \[ \norm{(I-D_{\mathcal P}C_{\mathcal P})|_E}<\varepsilon, \qquad A_E|_E=I_E,\qquad \operatorname{rank}A_E\leq m+n, \qquad \norm{A_E}\leq1+n\varepsilon. \]
\end{proposition}

\begin{proof}
    Let \((e_i,e_i^*)_{i=1}^n\) be an Auerbach basis for \(E\).
    Since simple functions are dense in \(L_p(\Omega)\), we can choose simple functions \(s_i\), all measurable with respect to a common finite partition \(\mathcal P\), such that \(\norm{e_i-s_i}_{L_p}<\eta\text{ for }1\leq i\leq n.\)
    Since \(D_{\mathcal P}C_{\mathcal P}s_i=s_i\) and \(\norm{D_{\mathcal P}C_{\mathcal P}}\leq1\), we have \(\norm{e_i-D_{\mathcal P}C_{\mathcal P}e_i}_{L_p}\leq2\eta.\)
    If \(x\in B_E\) and \(x=\sum_{i=1}^n a_ie_i\), then \(\abs{a_i}=\abs{e_i^*(x)}\leq1\).
    Thus we have \(\norm{x-D_{\mathcal P}C_{\mathcal P}x}_{L_p}\leq2n\eta.\)
    By taking \(\eta<\varepsilon/(2n)\), we obtain the approximation estimate.
    Put \(S=D_{\mathcal P}C_{\mathcal P}\).
    By Proposition~\ref{prop:correction-to-exact-coder}, with \(T=I_{L_p(\Omega)}\), \(\delta=\norm{(I-S)|_E}<\varepsilon\), \(\Gamma_S=1\), and \(\operatorname{rank}S\leq m\), we obtain the asserted \(A_E=D_EC_E\).
\end{proof}

\begin{remark}
    The number of atoms in \(\mathcal P\) of Proposition~\ref{prop:Lp-fixed-subspace-conditional-expectation} may depend on the particular subspace \(E\), as well as on \(\dim E\) and \(\varepsilon\).
    A UAP rank bound therefore requires an atom count uniform over subspaces of a given dimension.
    By taking \(\varepsilon=1/n\), we have \(\norm{A_E}\leq2\).
    The operator \(A_E\) factors through \(Z_E=A_E(L_p(\Omega))\), but it need not factor through the original cell-average measurements.
\end{remark}

\begin{definition}\label{def:averaging-doubling}
    Fix \(1\leq p<\infty\) and \(0<s\leq1\).
    Let \(\Omega\subseteq\R^d\) be a bounded domain with Lebesgue measure.
    In each of the two definitions below, we require constants \(C_{\#,\Omega},C_{\mathrm{av},\Omega,p,s}>0\) and finite measurable partitions \(\mathcal P_J=\{A_{J,1},\ldots,A_{J,N_J}\}\), \(J\geq0\), with \(0<|A_{J,\nu}|<\infty\) and \(N_J\leq C_{\#,\Omega}2^{Jd}\).
    The constants and partitions may be chosen independently in the two definitions.
    Put \(P_J^\Omega=\mathbb E[\,\cdot\,\mid\mathcal P_J]\).

    The domain satisfies the \(B^s_{p,\infty}\) averaging estimates if such constants and partitions satisfy
    \[ \norm{f-P_J^\Omega f}_{L_p(\Omega)} \leq C_{\mathrm{av},\Omega,p,s}2^{-Js}\norm{f}_{B^s_{p,\infty}(\Omega)},\qquad f\in B^s_{p,\infty}(\Omega). \]
    It satisfies the \(W^{s,p}\) averaging estimates if such constants and partitions satisfy
    \begin{equation}\label{eq:sobolev-averaging-estimate}
        \norm{f-P_J^\Omega f}_{L_p(\Omega)} \leq C_{\mathrm{av},\Omega,p,s}2^{-Js}\norm{f}_{W^{s,p}(\Omega)},\qquad f\in W^{s,p}(\Omega).
    \end{equation}
\end{definition}

For the cube \(Q=[0,1]^d\) with dyadic partitions, the Besov averaging estimate holds for \(1\leq p<\infty\) and \(0<s<1\); see \cite{DeVoreLorentz1993,Triebel1983}.
The Sobolev estimate holds for \(0<s\leq1\), including \(s=1\) by the cellwise Poincar\'e inequality.
We first use Besov smoothness to obtain a uniform atom count.

\begin{proposition}\label{prop:Lp-smooth-dyadic-conditional-expectation}
    Let \(\Omega\subseteq\R^d\) be a bounded domain with Lebesgue measure, let \(1\leq p<\infty\), and let \(0<s\leq1\).
    Assume that \(\Omega\) satisfies the \(B^s_{p,\infty}\) averaging estimates, and let \((\mathcal P_J)\) be the corresponding partitions.
    For a finite-dimensional subspace \(E\subseteq L_p(\Omega)\), set
    \[ M_s^\Omega(E) = \sup_{f\in B_E}\norm{f}_{B^s_{p,\infty}(\Omega)}. \]
    If \(M_s^\Omega(E)\leq M\), then for every \(\varepsilon>0\) there is a conditional expectation \(P_J^\Omega\) such that
    \[ \norm{(I-P_J^\Omega)|_E}\leq\varepsilon,\qquad \operatorname{rank}P_J^\Omega \leq 2^dC_{\#,\Omega} \left(1+\frac{C_{\mathrm{av},\Omega,p,s}M}{\varepsilon}\right)^{d/s}. \]
    If \(q=\dim E\geq1\), then there are bounded linear maps \(L_p(\Omega)\xrightarrow{C_E}Z_E\xrightarrow{D_E}L_p(\Omega)\) such that
    \[ D_EC_E|_E=I_E, \qquad \norm{D_EC_E}\leq2, \qquad \dim Z_E\leq 2^dC_{\#,\Omega}(1+C_{\mathrm{av},\Omega,p,s}Mq)^{d/s}+q. \]
\end{proposition}

\begin{proof}
    Since \(P_J^\Omega\) is the conditional expectation associated with \(\mathcal P_J\), we have \(\norm{P_J^\Omega}_{L_p\to L_p}\leq1\), and its range is the space of functions which are constant on the atoms of \(\mathcal P_J\).
    Therefore, we have \(\operatorname{rank}P_J^\Omega\leq N_J\leq C_{\#,\Omega}2^{Jd}\).
    If \(M_s^\Omega(E)\leq M\), then, by the Besov averaging hypothesis, we have \(\norm{(I-P_J^\Omega)|_E} \leq C_{\mathrm{av},\Omega,p,s}M2^{-Js}\).
    Choose the least \(J\geq0\) for which the right side is at most \(\varepsilon\).
    Then we have
    \[ 2^{Jd} \leq 2^d\left(1+\frac{C_{\mathrm{av},\Omega,p,s}M}{\varepsilon}\right)^{d/s}. \]
    For \(q=\dim E\geq1\), we make this choice with \(\varepsilon=1/q\).
    By Proposition~\ref{prop:correction-to-exact-coder}, with \(T=I_{L_p(\Omega)}\), \(S=P_J^\Omega\), \(\delta=1/q\), \(m=\operatorname{rank}P_J^\Omega\), and \(\Gamma_S=1\), we obtain \(D_EC_E|_E=I_E\), \(\norm{D_EC_E}\leq2\), and \(\dim Z_E\leq\operatorname{rank}P_J^\Omega+q\).
    By the preceding rank bound, we obtain the assertion.
\end{proof}

For \(\dim E\leq n\) and \(M_s^\Omega(E)\leq M\), the last bound is uniform.
The same argument applies when the averaging estimate is stated with a Sobolev norm.

\begin{proposition}\label{prop:Lp-regular-domain-sobolev-version}
    Let \(\Omega\subseteq\R^d\) be a bounded domain with Lebesgue measure, let \(1\leq p<\infty\), and let \(0<s\leq1\).
    Assume that \(\Omega\) satisfies the \(W^{s,p}\) averaging estimates, and let \((\mathcal P_J)\) be the corresponding partitions.
    For a finite-dimensional subspace \(E\subseteq L_p(\Omega)\), set
    \[ M_{s,W}^\Omega(E) = \sup_{f\in B_E}\norm{f}_{W^{s,p}(\Omega)}. \]
    If \(M_{s,W}^\Omega(E)\leq M\), then for every \(\varepsilon>0\) there is a conditional expectation \(P_J^\Omega\) such that
    \[ \norm{(I-P_J^\Omega)|_E}\leq\varepsilon,\qquad \operatorname{rank}P_J^\Omega \leq 2^dC_{\#,\Omega} \left(1+\frac{C_{\mathrm{av},\Omega,p,s}M}{\varepsilon}\right)^{d/s}. \]
    If \(q=\dim E\geq1\), then there are bounded linear maps \(L_p(\Omega)\xrightarrow{C_E}Z_E\xrightarrow{D_E}L_p(\Omega)\) such that
    \[ D_EC_E|_E=I_E, \qquad \norm{D_EC_E}\leq2, \qquad \dim Z_E\leq 2^dC_{\#,\Omega}(1+C_{\mathrm{av},\Omega,p,s}Mq)^{d/s}+q. \]
\end{proposition}

\begin{proof}
    By the Sobolev averaging hypothesis, we have \(\norm{P_J^\Omega}\leq1\), \(\operatorname{rank}P_J^\Omega\leq C_{\#,\Omega}2^{Jd}\), and \(\norm{(I-P_J^\Omega)|_E}\leq C_{\mathrm{av},\Omega,p,s}M2^{-Js}\).
    Thus, by the scale choice and correction in the proof of Proposition~\ref{prop:Lp-smooth-dyadic-conditional-expectation}, we obtain the assertion.
\end{proof}

For \(\dim E\leq n\) and \(M_{s,W}^\Omega(E)\leq M\), we obtain a uniform rank bound.

In Appendix~\ref{app:detailed-operator-learning-examples}, we apply these conditional expectations to the linear factors of a nonlinear integral operator.
The averaging estimate and Proposition~\ref{prop:correction-to-exact-coder} give explicit operator UAP rank bounds for those factors.

\subsection{Gabor truncation in modulation spaces}\label{subsec:gabor-phase-space-coder}

The phase-space weight controls the Gabor truncation error, and the lattice count controls the rank.
The frame identities and inner products in this subsection use the standard complex convention.
For the coding and rank estimates, the function and coefficient spaces are regarded as real Banach spaces.

Fix \(a,b>0\), let \(\Lambda=a\Z^d\times b\Z^d\), and write \(z=(x,\xi)\in\R^{2d}\).
The time-frequency shift is 
\[\pi(z)f(t)=e^{2\pi i\xi\cdot t}f(t-x).\]
For \(g,\gamma\) and \(\lambda\in\Lambda\), put \(g_\lambda=\pi(\lambda)g\) and \(\gamma_\lambda=\pi(\lambda)\gamma\).
We use \(M^s_{p,q}(\R^d)\) for the modulation space with the full phase-space weight \(m_s(z)=\langle z\rangle^s=(1+\norm{z}_2^2)^{s/2}\).
Let \(1\leq p,q<\infty\) and \(s_0>s_1\), write \(\beta=s_0-s_1\), and let \(T:M^{s_0}_{p,q}(\R^d)\hookrightarrow M^{s_1}_{p,q}(\R^d)\).
For \(c=(c_{k,n})_{k,n\in\Z^d}\), where \(c_{k,n}=c_{(ak,bn)}\), set
\[ \norm{c}_{\ell^{p,q}_{m_s}(\Lambda)} = \left( \sum_{n\in\Z^d} \left( \sum_{k\in\Z^d} \abs{c_{k,n}}^p\langle(ak,bn)\rangle^{sp} \right)^{q/p} \right)^{1/q}. \]
For \(R\geq1\), define \(\Lambda_R=\{\lambda\in\Lambda:\norm{\lambda}_2\leq R\}\) and \(m_R=\abs{\Lambda_R}.\)

Put \(v_\rho(z)=(1+\norm{z}_2)^\rho\).
We write \(M^1_{v_\rho}\) for the modulation space with both mixed exponents equal to one and weight \(v_\rho\).

\begin{proposition}\label{prop:gabor-banach-frame-input}
    Let \(\rho\geq\max\{\abs{s_0},\abs{s_1}\}\).
    Suppose that \(g,\gamma\in M^1_{v_\rho}(\R^d)\), and that \((\pi(\lambda)g)_{\lambda\in\Lambda}\) and \((\pi(\lambda)\gamma)_{\lambda\in\Lambda}\) are dual Gabor frames for \(L_2(\R^d)\).
    Then the maps \(C_gf=(\inp{f}{g_\lambda})_{\lambda\in\Lambda}\) and \(D_\gamma c=\sum_{\lambda\in\Lambda}c_\lambda\gamma_\lambda\) are bounded between \(M^{s_i}_{p,q}(\R^d)\) and \(\ell^{p,q}_{m_{s_i}}(\Lambda)\) for \(i=0,1\).
    Thus there are constants \(B_{s_i},S_{s_i}>0\) such that
    \[ \norm{C_gf}_{\ell^{p,q}_{m_{s_i}}}\leq B_{s_i}\norm{f}_{M^{s_i}_{p,q}}, \qquad \norm{D_\gamma c}_{M^{s_i}_{p,q}} \leq S_{s_i}\norm{c}_{\ell^{p,q}_{m_{s_i}}}. \]
    Moreover, \(D_\gamma C_gf=f\) in \(M^{s_i}_{p,q}(\R^d)\).
    The synthesis series converges unconditionally in that norm.
\end{proposition}

Here the coefficients use the standard extension of the \(L_2\) pairing.
The proposition is the standard Banach-frame and atomic-decomposition theorem \cite{Grochenig2001,Heil2011}.
The restriction on \(\rho\) gives the required \(v_\rho\)-moderateness, \(m_{s_i}(z+w)\lesssim v_\rho(z)m_{s_i}(w)\), and the finite exponents give norm convergence.
We set \(C_{\mathrm{Gab}}=S_{s_1}B_{s_0}\).
Let \(r_Rc=(c_\lambda)_{\lambda\in\Lambda_R}\), \(e_R\) be zero extension to \(\Lambda\), and \(Q_R=e_Rr_R\).

\begin{lemma}\label{lem:gabor-tail-counting}
    For \(c\in\ell^{p,q}_{m_{s_0}}(\Lambda)\), we have
    \[ \norm{(I-Q_R)c}_{\ell^{p,q}_{m_{s_1}}} \leq \langle R\rangle^{-\beta} \norm{c}_{\ell^{p,q}_{m_{s_0}}}. \]
    There is a lattice counting constant \(C_{\#,\Lambda,d}>0\) such that \(m_R\leq C_{\#,\Lambda,d}R^{2d}\) for \(R\geq1\).
\end{lemma}

\begin{proof}
    For \(\lambda\notin\Lambda_R\), we have \(\langle\lambda\rangle^{s_1}=\langle\lambda\rangle^{-\beta}\langle\lambda\rangle^{s_0} \leq\langle R\rangle^{-\beta}\langle\lambda\rangle^{s_0}\).
    Taking the mixed \(\ell^{p,q}\) norm gives the tail estimate with constant \(1\).
    And by counting lattice cells meeting the ball, we have the second estimate with a constant depending on the lattice covolume and cell diameter.
\end{proof}

\begin{remark}
    The finite Gabor sum \(P_R=D_\gamma Q_RC_g\) is a bounded finite-rank recovery map and need not satisfy \(P_R^2=P_R\).
    Corollary~\ref{cor:GaborExactUAP} supplies exact agreement on any prescribed finite-dimensional subspace by correction.
    In \(L_2(\R^d)\), the orthogonal projection onto \(V_R=\Span_{\R}\{\gamma_\lambda,i\gamma_\lambda:\lambda\in\Lambda_R\}\) also gives a norm-one recovery map exact on \(V_R\).
    It is computed from the finite Gram system, using the pseudoinverse when the atoms are linearly dependent.
\end{remark}

\begin{proposition}\label{prop:gabor-approximate-coder}
    Under the hypotheses of Proposition~\ref{prop:gabor-banach-frame-input}, for \(R\geq1\), define the Gabor encoder--decoder maps
    \[ M^{s_0}_{p,q}(\R^d)\xrightarrow{C_R}\ell^{p,q}_{m_{s_0}}(\Lambda_R) \xrightarrow{D_R}M^{s_1}_{p,q}(\R^d) \]
    given by
    \[ C_Rf=r_RC_gf=(\inp{f}{g_\lambda})_{\lambda\in\Lambda_R},\qquad D_Ra=D_\gamma e_Ra=\sum_{\lambda\in\Lambda_R}a_\lambda\gamma_\lambda. \]
    Then \(\norm{C_R}\leq B_{s_0}\) and \(\norm{D_R}\leq S_{s_1}\).
    The operator \(P_R=D_RC_R\) satisfies
    \begin{align*}
        \operatorname{rank}P_R\leq2m_R\leq2C_{\#,\Lambda,d}R^{2d}, \quad \norm{P_R}_{M^{s_0}_{p,q}\to M^{s_1}_{p,q}}\leq C_{\mathrm{Gab}},\quad\norm{T-D_RC_R}_{M^{s_0}_{p,q}\to M^{s_1}_{p,q}} \leq C_{\mathrm{Gab}}\langle R\rangle^{-\beta}.
    \end{align*}
\end{proposition}

\begin{proof}
    Since the coefficient space has real dimension \(2m_R\), we have \(\operatorname{rank}P_R\leq2m_R\).
    By Lemma~\ref{lem:gabor-tail-counting}, we obtain the lattice count.
    Since restriction, zero extension, and the inclusion from the \(s_0\)-weighted coefficient space into the \(s_1\)-weighted space are contractive, the analysis and synthesis bounds give the stated map norms and \(\norm{P_R}\leq S_{s_1}B_{s_0}=C_{\mathrm{Gab}}\).
    By the tail estimate applied to \(T-P_R=D_\gamma(I-Q_R)C_g\), we have \(\norm{T-P_R}_{M^{s_0}_{p,q}\to M^{s_1}_{p,q}}\leq S_{s_1}\langle R\rangle^{-\beta}B_{s_0}\).
\end{proof}

The lattice count also gives \(\norm{T-D_RC_R}_{M^{s_0}_{p,q}\to M^{s_1}_{p,q}} \leq C_{\mathrm{Gab}}C_{\#,\Lambda,d}^{\beta/(2d)}m_R^{-\beta/(2d)}.\)

\begin{corollary}\label{cor:GaborExactUAP}
    Let \(E\subseteq M^{s_0}_{p,q}(\R^d)\) be \(n\)-dimensional, \(n\geq1\).
    Under the hypotheses of Proposition \ref{prop:gabor-banach-frame-input}, there is a finite-rank operator \(A_E:M^{s_0}_{p,q}(\R^d)\to M^{s_1}_{p,q}(\R^d)\) such that
    \[ A_E|_E=T|_E,\quad \norm{A_E}_{M^{s_0}_{p,q}\to M^{s_1}_{p,q}}\leq C_{\mathrm{Gab}}+1, \quad \operatorname{rank}A_E \leq C_{\#,\Lambda,d}2^{2d+1} (1+C_{\mathrm{Gab}}n)^{2d/(s_0-s_1)}+n. \]
\end{corollary}

\begin{proof}
    Let \(R_n=\max\{1,\lceil(C_{\mathrm{Gab}}n)^{1/\beta}\rceil\}\).
    Then \(C_{\mathrm{Gab}}\langle R_n\rangle^{-\beta}\leq1/n\) and \(R_n^{2d}\leq2^{2d}(1+C_{\mathrm{Gab}}n)^{2d/\beta}\).
    Proposition~\ref{prop:gabor-approximate-coder} gives, for \(S=P_{R_n}\), \(\norm{S}\leq C_{\mathrm{Gab}}\), \(\norm{(T-S)|_E}\leq1/n\), and \(\operatorname{rank}S\leq2C_{\#,\Lambda,d}R_n^{2d}\).
    Proposition~\ref{prop:correction-to-exact-coder} gives exact agreement on \(E\), adding at most \(1\) to the norm bound and \(n\) to the rank bound.
\end{proof}

For \(T\ne0\), Corollary~\ref{cor:GaborExactUAP} gives the normalized operator-UAP constant \(\max\{1,(C_{\mathrm{Gab}}+1)/\norm{T}\}\).

\section{Neural approximation and Barron sampling}\label{app:reduced-map-modules}

This appendix supplies the neural approximation bounds used after linear reduction and the direct sampling argument for Barron operators.
Gauge-preserving extension allows us to apply finite-dimensional network estimates to coordinates initially defined only on an encoded compact set.

\subsection{Gauge extensions and ReLU parameter bounds}

We first extend the scalar coordinates on the encoded compact set without changing their gauge and then prove the ReLU estimate used in the main text.

\begin{lemma}\label{lem:scalar-infimal-extension}
    Let \(D\) be a nonempty subset of a normed space \(Z\).
    Let \(g:D\to\R\) be a function, and let \(L\geq0\).
    Let \(\omega\) be a gauge, and assume \(\abs{g(z)-g(w)}\leq\omega(L\norm{z-w})\) for \(z,w\in D\).
    Then
    \[ \widetilde g(z)=\inf_{w\in D}\{g(w)+\omega(L\norm{z-w})\},\qquad z\in Z, \]
    extends \(g\) and satisfies \(\abs{\widetilde g(z)-\widetilde g(z')}\leq\omega(L\norm{z-z'})\) for \(z,z'\in Z\).
\end{lemma}

\begin{proof}
    Let \(w_0\in D\).
    By the gauge assumption and subadditivity, we have \(g(w)+\omega(L\norm{z-w}) \geq g(w_0)-\omega(L\norm{z-w_0})\), so the infimum is finite.
    If \(z\in D\), by choosing \(w=z\), we have \(\widetilde g(z)\leq g(z)\).
    By the assumed estimate, we have \(g(w)+\omega(L\norm{z-w})\geq g(z)\) for every \(w\in D\).
    Hence \(\widetilde g=g\) on \(D\).

    By subadditivity and monotonicity of \(\omega\), we have \(\omega(L\norm{z-w}) \leq\omega(L\norm{z'-w})+\omega(L\norm{z-z'})\).
    By taking the infimum over \(w\in D\) and then interchanging \(z\) and \(z'\), we obtain the required estimate.
\end{proof}

Let \(Q_m(B)=[-B,B]^m\), with the \(\ell_\infty^m\)-norm.
A ReLU network below is a finite feedforward network with activation \(\sigma(t)=\max\{t,0\}\).
We count all trainable scalar weight and bias slots in the chosen architecture, including slots whose value is zero.

For finite-dimensional smooth functions, Lu, Shen, Yang, and Zhang \cite{LuShenYangZhang2021} give ReLU approximation bounds with explicit width and depth dependence.
Jiao, Wang, and Yang \cite{JiaoWangYang2023} study approximation under network weight-norm constraints, with applications to regression and GANs.
We use Theorem~1.1 of Shen, Yang, and Zhang \cite{ShenYangZhang2022Optimal} to obtain a constant-depth approximation bound in terms of the modulus of continuity.

\begin{proposition}\label{prop:quantitative-continuous-relu}
    For every \(m\in\N_+\), there are \(a_m,c_m>0\) and \(W_m,L_m\in\N_+\) with the following property.
    For each \(B>0\) and integer \(W\geq W_m\), there is a fixed ReLU architecture on \(\R^m\), with at most \(W\) scalar weight and bias slots and depth at most \(L_m\), such that every \(g\in C(Q_m(B))\) has a realization \(N:\R^m\to\R\) satisfying
    \begin{equation}\label{eq:relu-modulus-estimate}
        \norm{g-N}_\infty \leq a_m\omega_g\!\left(2Bc_m[W\log(e+W)]^{-1/m}\right).
    \end{equation}
\end{proposition}

\begin{proof}
    Let \(f(u)=g(2Bu-B\mathbf1)\) on \([0,1]^m\).
    If \(\norm{u-v}_2\leq r\), then \(\norm{2B(u-v)}_\infty\leq2Br\).
    Thus the Euclidean modulus of \(f\) satisfies \(\omega_{f,2}(r)\leq\omega_g(2Br)\).
    By \cite[Theorem~1.1]{ShenYangZhang2022Optimal}, with \(p=\infty\) and the depth parameter equal to \(1\), for every \(N_0\in\N_+\), we have a ReLU approximation of \(f\) with at most \(29+2m\) hidden layers, width at most \(3^{m+3}\max\{m\lfloor N_0^{1/m}\rfloor,N_0+2\}\), and error at most
    \[ 131\sqrt m\,\omega_g\!\left(2B[N_0^2\log_3(N_0+2)]^{-1/m}\right). \]
    Since this width is at most \(A_mN_0\) for a constant \(A_m\), we can add neurons with zero output and identity maps between hidden layers to obtain one fixed fully connected graph.
    The hidden-layer outputs are nonnegative, so the added identity maps are realized by ReLU.
    If we count every weight and bias slot in this graph, the total is at most \(K_mN_0^2\), for a constant \(K_m\) depending only on \(m\).

    For \(W\) sufficiently large, choose \(N_0=\lfloor\sqrt{W/K_m}\rfloor\).
    Then
    \[ K_mN_0^2\leq W,\qquad N_0^2\log_3(N_0+2)\geq b_mW\log(e+W) \]
    for some \(b_m>0\).
    By the monotonicity of \(\omega_g\), we obtain the stated estimate after choosing \(a_m,c_m,W_m\) accordingly.
    The substitution \(u=(z+B\mathbf1)/(2B)\) changes only the weights and biases of the first layer, so the number of parameter slots is unchanged.
    We include the final affine output layer in the depth bound \(L_m\).
\end{proof}

For \(t\geq0\), put
\begin{equation}\label{eq:relu-budget-inversion}
    \mathcal R(t)=\min\{w\in\N_+:w\log(e+w)\geq t\}, \qquad \mathcal R(t)\asymp 1+\frac{t}{\log(e+t)}.
\end{equation}
Indeed, for \(t\geq1\), we have
\[ \frac{t}{\log(e+t+1)} \leq\mathcal R(t) \leq\left\lceil\frac{2t}{\log(e+t)}\right\rceil. \]
For the lower estimate, use \(\mathcal R(t)\leq\lceil t\rceil\).
For the upper estimate, the integer on the right is at least \(\sqrt t\), since \(\log(e+t)\leq2\sqrt t\).
Its logarithmic factor is therefore at least \(\frac{1}{2}\log(e+t)\).
Since \(\mathcal R(t)=1\) for \(0\leq t\leq1\), we obtain the stated comparison for all \(t\geq0\).

By \eqref{eq:relu-modulus-estimate} and the definition of \(\delta_\omega\) in Section~\ref{subsec:compact-widths-uap-selected}, we obtain a uniform error smaller than \(\eta\) whenever
\[ W\geq\max\left\{W_m,\, \mathcal R\!\left(\left[\frac{2Bc_m}{\delta_{\omega_g}(\eta/a_m)}\right]^m\right)\right\}. \]
Indeed, the argument of \(\omega_g\) in \eqref{eq:relu-modulus-estimate} is then at most \(\delta_{\omega_g}(\eta/a_m)\).
The constants \(a_m,c_m,W_m,L_m\) may grow with \(m\).
We retain them when the dimension selected by the UAP depends on the target accuracy.

\subsection{Barron representation norms and sampling estimates}\label{app:barron-representations}

The scalar representation norm below is determined by restrictions to finite sets and finite-dimensional subspaces.
This gives a criterion for the existence of a global representation and its norm.
Realization through finite-dimensional coders also requires estimates for their stability, dimension, and recovery error.
The converse statement records the encoder norm needed to pass from pointwise-convergent reduced representations to a representation on the original set.
For a real Banach space \(Z\), \(S\subseteq Z\), and \(f:S\to\R\), let \(\norm f_{\mathcal B_Z(S)}\) be the infimum of \(\norm\mu_{\mathrm{TV}}\) over finite signed Radon measures on \(\mathcal P_Z=B_{Z^*}^{w^*}\times[-1,1]\) satisfying \(f(z)=\int_{\mathcal P_Z}\sigma(\ell(z)+b)\,d\mu(\ell,b)\) for \(z\in S\).
An empty infimum is \(+\infty\).
We write \(\mathcal B_Z\) when \(S=Z\), and \(\mathcal B_m(S)=\mathcal B_{\ell_\infty^m}(S)\).
For the variation norm and its compactness properties, see \cite[Sections~2.1--2.2]{Bach2017} and~\cite{MaWuE2022}.

\begin{proposition}\label{prop:scalar-barron-finite-restrictions}
    Let \(S\subseteq X\) and \(f:S\to\R\).
    The defining infimum is attained whenever \(\norm f_{\mathcal B_X(S)}<\infty\), and
    \[
        \norm f_{\mathcal B_X(S)} =\sup_{\substack{A\subseteq S\\
                A\ \mathrm{finite}}} \norm{f|_A}_{\mathcal B_X(A)}.
    \]
    For every \(F:X\to\R\) and bounded linear map \(R:Z\to X\), we have
    \begin{align*}
        \norm F_{\mathcal B_X} =\sup_{\substack{E\subseteq X\\
                \dim E<\infty}} \norm{F|_E}_{\mathcal B_E},\qquad \norm{F\circ R}_{\mathcal B_Z} \leq\max\{1,\norm R\}\norm F_{\mathcal B_X}.
    \end{align*}
    Conversely, suppose that bounded linear maps \(C_m:X\to\ell_\infty^m\) and \(R_m:\ell_\infty^m\to X\) satisfy \(\sup_m\norm{C_m}\leq\Lambda<\infty\), \(\sup_m\norm{F\circ R_m}_{\mathcal B_m(C_mS)}\leq B<\infty,\) and \(F(R_mC_mx)\to F(x)\) for every \(x\in S\).
    Then we have \(\norm{F|_S}_{\mathcal B_X(S)}\leq\max\{1,\Lambda\}B.\)
\end{proposition}

\begin{proof}
    For each \(x\in X\), the ridge function \((\ell,b)\mapsto\sigma(\ell(x)+b)\) is continuous on \(\mathcal P_X\).
    By weak-star compactness of the total-variation balls, the infimum is attained when finite.
    By the finite intersection property, the norm is determined by its restrictions to finite subsets.
    For finite-dimensional \(E\subseteq X\), the restriction \(\rho_E:\mathcal P_X\to\mathcal P_E\) is a continuous surjection by the Hahn--Banach theorem.
    Thus \(h\mapsto h\circ\rho_E\) is an isometry.
    By the Hahn--Banach theorem and the Riesz representation theorem, every measure on \(\mathcal P_E\) lifts to one on \(\mathcal P_X\) with the same norm.
    By this lift with \(E=\operatorname{span}A\), for finite \(A\subseteq X\), the norm is bounded above by the supremum of the finite-dimensional restriction norms.
    For the reverse inequality, we can push measures forward under \(\rho_E\).

    For the composition estimate, put \(a_R=\max\{1,\norm R\}\).
    By ReLU homogeneity, \(a_R\) times the pushforward of a measure representing \(F\) under \((\ell,b)\mapsto a_R^{-1}(R^*\ell,b)\) represents \(F\circ R\).
    For the converse, choose measures \(\nu_m\) of norm at most \(B\) representing \(F\circ R_m\) on \(C_mS\).
    By applying the same pushforward construction to \(C_m\), we obtain measures \(\mu_m\) on \(\mathcal P_X\) representing \(x\mapsto F(R_mC_mx)\) on \(S\), with
    \[ \norm{\mu_m}_{\mathrm{TV}} \leq\max\{1,\norm{C_m}\}\norm{\nu_m}_{\mathrm{TV}} \leq\max\{1,\Lambda\}B. \]
    By weak-star compactness, a subnet converges to a measure with the same norm bound.
    Then this measure represents \(F|_S\) by the assumed pointwise convergence.
\end{proof}

If \(F\) is continuous and \(R_mC_mx\to x\) on \(S\), then the pointwise convergence assumption follows directly.
Under their respective approximation hypotheses, the reduction results in the main text supply rank and error bounds for the recovery maps.
For a represented vector-valued operator, the following sampling argument controls the number of neurons directly.

\begin{proof}[Proof of Proposition~\ref{thm:type-p-vector-barron-sampling}]
    Put \(A=B_0(R_K+1)\) and \(Z_i(x)=B_0\sigma(\ell_i(x)+b_i)u(\theta_i)\).
    Then \(T(x)=\mathbb EZ_i(x)\), \(\norm{Z_i(x)}_Y\leq A\), and \(T_N(x)=N^{-1}\sum_iZ_i(x)\).
    By \(\norm{\ell_i}\leq1\) and \(\norm{u(\theta_i)}_Y=1\), we obtain the encoder and decoder norm bounds directly.
    By symmetrization and the type inequality, we have
    \[ \mathbb E\norm{T(x)-T_N(x)}_Y^{p_Y} \leq\frac{2^{p_Y}T_{p_Y}(Y)^{p_Y}}{N^{p_Y}} \sum_{i=1}^N\mathbb E\norm{Z_i(x)}_Y^{p_Y} \leq\bigl(2T_{p_Y}(Y)A\bigr)^{p_Y}N^{1-p_Y}. \]
    By integrating against \(\lambda\), we obtain the \(L^{p_Y}(\lambda;Y)\) estimate.

    Choose an internal \(\delta\)-net \(x_1,\ldots,x_M\) from an external \(\delta/2\)-cover, so that \(M\leq N(K,(\delta/2)B_X)\).
    For \(F_j=\norm{T(x_j)-T_N(x_j)}_Y\), a change in one sample changes \(F_j\) by at most \(2A/N\), and \(\mathbb EF_j\leq2T_{p_Y}(Y)A N^{-(1-1/p_Y)}\).
    By the bounded differences inequality and a union bound, we have
    \[ \mathbb P\left\{\max_{j\leq M}F_j> 2T_{p_Y}(Y)A N^{-(1-1/p_Y)}+t\right\} \leq M\exp\left(-\frac{Nt^2}{2A^2}\right). \]
    We choose
    \[ t=A\sqrt{\frac{2\log(2M)}{N}}. \]
    The right-hand side is at most \(1/2\).
    Thus there is a realization of \(T_N\) satisfying the corresponding bound at all net points.
    Since \(T\) and \(T_N\) are \(B_0\)-Lipschitz, we extend this bound to all of \(K\) with an additional error of \(2B_0\delta\).
\end{proof}

\section{Proofs of the bounds for operator classes}\label{app:lower-bound-machinery}

We prove the class estimates by combining common reductions for the upper bound with separated inputs and scalar readouts for the lower bound.
We use the resulting pseudo-dimension lower bound to estimate the number of trainable parameters.

\begin{proof}[Proof of Lemma~\ref{lem:scalarized-readout-pdim}]
    By Assumption~\ref{ass:scalar-readout-realization}, the scalar readouts have a common ReLU representation with \(O(W)\) coefficients and computation units and depth \(O(D)\).
    If we allow its fixed coefficients to vary as well, we obtain a larger network class.
    By \cite[Theorem~7]{BartlettHarveyLiawMehrabian2019} and the pseudo-dimension reduction of the same paper, its pseudo-dimension is at most \(C_{\rm pd}WD(1+\log W)\).
    By composition with the fixed input map, the same bound holds for the original class.
\end{proof}

\begin{proof}[Proof of Corollary~\ref{prop:factorized-uap-upper}]
    By Corollary~\ref{prop:operator-UAP-input-bound}(1), applied first to \(L\) on \(H\) and then to \(S\) on \(K\), we obtain common maps \(\widetilde L,\widetilde S\) with the stated norm, rank, and error bounds.
    By the Auerbach construction in Section~\ref{subsec:compact-widths-uap-selected}, we obtain the coordinate maps and their norm bounds.
    For each \(F\in\mathcal F\) and \(x\in K\), we have
    \begin{align*}
        T_F(x)-RN(Cx)
        &=(L-\widetilde L)F(Sx)+\widetilde L\bigl(F(Sx)-F(\widetilde Sx)\bigr)+R\bigl(\Phi_F(Cx)-N(Cx)\bigr).
    \end{align*}
    Since \(F(Sx)\in H\) and \(\norm{\widetilde L}\leq\Lambda_L\norm L\), we apply the common gauge bound and the recovery estimates, then take the supremum over \(K\) to obtain the result.
\end{proof}

\begin{lemma}\label{lem:metric-holder-lower-gauge}
    Let \(0<\beta\leq1\).
    There is \(c_\beta>0\), depending only on \(\beta\), with the following property.
    For every totally bounded metric space \((A,d)\), \(B>0\), and \(0<r\leq1\), put \(\omega(t)=Bt^\beta\).
    There are an \(r\)-separated set \(\{a_1,\ldots,a_s\}\subseteq A\), with \(s\geq \mathsf P(A,d,r)/2\), thresholds \(b_1,\ldots,b_s\), and functions \(h_\theta\in\mathcal H_{\omega,B}(A)\), \(\theta\in\{0,1\}^s\), such that
    \[ h_\theta(a_j)\leq b_j-c_\beta Br^\beta\ \text{if }\theta_j=0, \qquad h_\theta(a_j)\geq b_j+c_\beta Br^\beta\ \text{if }\theta_j=1. \]
\end{lemma}

\begin{proof}
    Choose the stated \(r\)-separated set and put \(\psi_j(a)=(1-(3d(a,a_j)/r)^\beta)_+\).
    These functions have disjoint supports and H\"older constant at most \((3/r)^\beta\).
    Let \(a_r=Br^\beta/(4\cdot3^\beta)\).
    Then \(h_\theta=a_r\sum_j\theta_j\psi_j\) belongs to \(\mathcal H_{\omega,B}(A)\).
    At \(a_j\), take \(b_j=a_r/2\).
    We obtain the assertion with \(c_\beta=1/(8\cdot3^\beta)\).
\end{proof}

\begin{proof}[Proof of Theorem~\ref{thm:factorized-operator-lower-bound}]
    Put \(r=\rho(\varepsilon)\).
    By Lemma~\ref{lem:metric-holder-lower-gauge}, we obtain an \(r\)-separated set \(\{z_1,\ldots,z_s\}\subseteq A\), with \(s\geq\frac12\mathsf P(A,\norm{\cdot}_Z,r)\), and scalar functions which realize every binary labeling with margin \(c_\beta Br^\beta\).
    Since \(A\subseteq S(K)\), we can choose \(x_j\in K\) such that \(Sx_j=z_j\).
    We assign one binary labeling to each of the \(q\) coordinates.
    By the assumption on \(\mathcal F\), the resulting functions \(h_1,\ldots,h_q\) are realized by some \(F_h\in\mathcal F\).

    For \(\theta=(\theta_{ij})\in\{0,1\}^{q\times s}\), choose \(h_i=h_{(\theta_{i1},\ldots,\theta_{is})}\) from Lemma~\ref{lem:metric-holder-lower-gauge}.
    The points \(x_j\) and thresholds \(b_j\) are independent of \(\theta\).
    Since \(Sx_j=z_j\) and \(\Pi LJ=I\), we have \(e_i^*\Pi T_{F_h}(x_j)=h_i(z_j)\) for every \(i,j\).

    Let \(N\in\mathcal{NN}_{W,D}\) approximate \(T_{F_h}\) with error at most \(\varepsilon\) on \(K\).
    Since \(\Pi LJ=I\) and \(\norm\Pi\leq C_-/\tau\), we have
    \[ \abs{e_i^*\Pi N(x_j)-h_i(z_j)} \leq\frac{C_-\varepsilon}{\tau} =\frac{c_\beta Br^\beta}{2}. \]

    By the choice of \(h_i\), we obtain
    \begin{align*}
        e_i^*\Pi N(x_j)\leq b_j-\frac{c_\beta Br^\beta}{2}<b_j
        \quad (\theta_{ij}=0), \qquad\text{and}\qquad
        e_i^*\Pi N(x_j)\geq b_j+\frac{c_\beta Br^\beta}{2}>b_j
        \quad (\theta_{ij}=1).
    \end{align*}
    The encoder, decoder, readouts, and network graph are fixed as \(\theta\) varies.
    Only the trainable parameter values may change.
    Thus the scalar readout class pseudo-shatters the \(qs\) points \((i,x_j)\).
    By Lemma~\ref{lem:scalarized-readout-pdim}, we have \(qs\leq C_{\rm pd}WD(1+\log W)\).
    The lower bound for \(s\) gives the required estimate.
\end{proof}

For nonempty \(A\subseteq\R^d\), \(\omega(t)=Bt^\beta\), and \(0<\beta\leq1\), Lemma~\ref{lem:scalar-infimal-extension}, followed by clipping to \([-B,B]\), extends each \(g\in\mathcal H_{\omega,B}(A)\) to \(\widetilde g\in\mathcal H_{\omega,B}(\R^d)\) with respect to the \(\ell_\infty^d\)-metric \cite{McShane1934}.

Given scalar networks \(N_j\), \(1\leq j\leq q\), clip their outputs by \(v_j=\sigma(N_j+B)-\sigma(N_j-B)-B\).
Then \(v\in[-B,B]^q\), and the error does not increase because each target coordinate lies in \([-B,B]\).
For the selector input \(t\), put
\begin{align*}
    \chi_j(t)&=\sigma(t-j+1)-2\sigma(t-j)+\sigma(t-j-1),\\
    \operatorname{Sel}(t,v)&=\sum_{j=1}^q \bigl[\sigma(v_j-B(1-\chi_j(t)))-\sigma(-v_j-B(1-\chi_j(t)))\bigr].
\end{align*}
For integers \(1\leq i\leq q\), we have \(\chi_j(i)=\delta_{ij}\).
The inactive terms vanish because \(\abs{v_j}\leq B\), and the active term equals \(v_i\).
Hence \(\operatorname{Sel}(i,v)=v_i\).
Since \(\Pi LJ=I\), the scalar readout of \(x\mapsto LJv(Sx)\) is exactly \(\operatorname{Sel}(i,v(Sx))\).

For the graph implementation, first replace each final affine output \(N_j\) by the two ReLU outputs \(\sigma(N_j)\) and \(\sigma(-N_j)\), counting the copied weights and biases as separate slots.
Their difference is then passed to the clipping formula using fixed coefficients.
If the original scalar graphs have \(W_0\) weight and bias slots in total, the resulting clipped network has a trainable slot count \(W\leq2W_0\).
The fixed clipping block keeps its output in \([-B,B]^q\) for every parameter choice.
The added clipping, hat, and selector operations use \(O(q)\) units and fixed coefficients.
Six additional layers suffice, including the affine combinations and output sum.
The scalar graph therefore has \(O(W+q)\) coefficients and units and depth at most \(L_d+6\).
Its input coordinates come from the fixed encoder \(S\), and the identity \(\Pi LJ=I\) removes the Banach-valued decoder from the scalar graph.
All fixed coefficients added by the clipping and selector blocks depend only on the fixed factors and \(B\).
Since \(q\leq W\), Assumption~\ref{ass:scalar-readout-realization} holds with a fixed absolute choice of \(C_{\rm arch}\).

\begin{proof}[Proof of Theorem~\ref{thm:factorized-holder-cube-sharp}]
    Put \(t=\ell^d(B\tau/\varepsilon)^{d/\beta}\), so that \(\mathcal B(\varepsilon)=q(1+t)\).

    Fix \(g=(g_1,\ldots,g_q)\in\mathcal H_{\omega,B}(Q_d(\ell))^q\).
    Since \(\omega_{g_i}(u)\leq Bu^\beta\), the scalar estimate gives an error at most \(a_dB(2\ell c_d)^\beta[w\log(e+w)]^{-\beta/d}\) for a network with \(w\geq W_d\) slots.
    Thus the error is at most \(\varepsilon/(2C_+\tau)\) whenever \(w\log(e+w)\geq[2a_dC_+\tau B(2\ell c_d)^\beta/\varepsilon]^{d/\beta}\).
    By \eqref{eq:relu-modulus-estimate}, each coordinate \(g_i\) admits an approximation with error at most \(\varepsilon/(2C_+\tau)\) and at most the following number of parameters:
    \begin{align*}
        w=\max\left\{W_d,\, \mathcal R\!\left( \left[\frac{2a_dC_+\tau B(2\ell c_d)^\beta}{\varepsilon}\right]^{d/\beta} \right)\right\}\leq C\,\frac{1+t}{\log(e+1+t)}.
    \end{align*}
    The last inequality follows from \eqref{eq:relu-budget-inversion}.
    The scalar architecture depends on the required accuracy and the fixed parameters, but not on \(g_i\).
    We clip each coordinate to \([-B,B]\) and use the parallel combination and selector above.
    The trainable count is at most \(2qw\), including the copied output slots.
    The clipping and selector blocks add \(O(q)\) units and fixed coefficients.
    The depth is at most \(\widehat L_d\leq D\).

    Let \(v=(v_1,\ldots,v_q)\) be the clipped network output.
    Since clipping does not increase the coordinate errors and \(S(K)=Q_d(\ell)\), we have
    \begin{align*}
        \sup_{x\in K}\norm{T_{F_g}(x)-LJv(Sx)}_X
        \leq C_+\tau\sup_{z\in Q_d(\ell)}\norm{g(z)-v(z)}_\infty
        =C_+\tau\max_{1\leq i\leq q}\norm{g_i-v_i}_\infty\leq\frac{\varepsilon}{2}.
    \end{align*}
    Consequently,
    \[ W_D^\square(\varepsilon) \leq C\,\frac{q(1+t)}{\log(e+1+t)} =C\,\frac{\mathcal B(\varepsilon)}{\log(e+\mathcal B(\varepsilon)/q)}. \]

    For the lower bound, apply Theorem~\ref{thm:factorized-operator-lower-bound} with \(A=Q_d(\ell)=S(K)\).
    By the definition of the class and Lemma~\ref{lem:scalar-infimal-extension}, followed by clipping to \([-B,B]\), the required independent coordinates belong to the class.
    Put \(r=\rho(\varepsilon)\leq1\).
    A product grid of step \(r\) in \(Q_d(\ell)\) gives
    \begin{align*}
        \mathsf P(Q_d(\ell),\norm{\cdot}_\infty,r) \geq\left(1+\left\lfloor\frac{2\ell}{r}\right\rfloor\right)^d\geq c\left[1+\ell^d\left(\frac{B\tau}{\varepsilon}\right)^{d/\beta}\right].
    \end{align*}
    Indeed, \(1+\lfloor2\ell/r\rfloor\geq\max\{1,2\ell/r\}\).

    By the definition of \(r\), we have \(\ell^dr^{-d}=(c_\beta/(2C_-))^{d/\beta}t\).
    Thus the packing number is bounded below by a fixed multiple of \(1+t\).
    Theorem~\ref{thm:factorized-operator-lower-bound} applies to every admissible graph that approximates the whole class.
    We therefore obtain
    \[ W_D^\square(\varepsilon)D\bigl(1+\log W_D^\square(\varepsilon)\bigr) \geq c\mathcal B(\varepsilon). \]
    The upper estimate gives \(W_D^\square(\varepsilon)\leq C\mathcal B(\varepsilon)\).
    Hence the logarithmic factor is at most a fixed multiple of \(1+\log(e+\mathcal B(\varepsilon))\), which gives the stated lower bound.
    Since the factors are fixed, we have \(t\to\infty\) as \(\varepsilon\downarrow0\), and thus \[\frac{\log(e+q(1+t))}{\log(e+1+t)}\longrightarrow1.\]

    Moreover, \(\mathcal B(\varepsilon)\asymp\varepsilon^{-d/\beta}\) and \(\log(e+1+t)\asymp\log(1/\varepsilon)\).
    By \eqref{eq:factorized-holder-cube-bounds}, with \(D\) fixed, we obtain the stated asymptotic estimate.
\end{proof}

\section{Recovery and parameter bounds for a nonlinear integral operator}\label{app:detailed-operator-learning-examples}
We consider a nonlinear integral operator of the form \(T_F=LFS\).
Compositions of linear integral operators and pointwise nonlinearities also occur in neural operator constructions \cite{KovachkiEtAl2023}.
For a prescribed approximation error, we choose the averaging scales and then bound the number of scalar network parameters.
Under the assumptions on the factors below, constant inputs also give a parameter lower bound for the class.

Let \(\Omega=(0,1)^d\), let \(1<p<\infty\), and let \(0<s\leq1\).
For \(A\in\{S,L\}\), let \((Au)(x)=\int_\Omega k_A(x,y)u(y)\,\rmd y\) define a nonzero bounded linear map on the real space \(L_p(\Omega)\).
Assume that \(\norm{Au}_{W^{s,p}(\Omega)}\leq M_A\norm{u}_{L_p(\Omega)}\) for constants \(M_S,M_L>0\).
For \(C^1\) kernels on \(\overline\Omega\times\overline\Omega\), this estimate holds with \(s=1\), by differentiation under the integral and H\"older's inequality.
Let \(\varphi:\R\to\R\) satisfy \(\varphi(0)=0\) and \(\abs{\varphi(a)-\varphi(b)}\leq L_\varphi\abs{a-b}\).
Then \(F(v)=\varphi\circ v\) has global gauge \(\omega_F(t)=L_\varphi t\), and
\[ T_F(u)(x)=\int_\Omega k_L(x,y) \varphi\!\left(\int_\Omega k_S(y,z)u(z)\,\rmd z\right)\,\rmd y. \]
For example, we may take \(k_S(x,y)=k_L(x,y)=e^{-\norm{x-y}_2^2}\) and \(\varphi(t)=\tanh t\), with \(s=1\) and \(L_\varphi=1\).

For \(J\geq0\), let \(\mathcal P_J\) be the dyadic partition of \(\Omega\) into \(m_J=2^{Jd}\) cubes.
Write \(C_J,R_J\) for the normalized encoder and decoder in Proposition~\ref{prop:Lp-conditional-expectation-coder}, and put \(P_J=R_JC_J\).
Thus \(\norm{C_J:L_p\to\ell_p^{m_J}}\leq1\), \(R_J:\ell_p^{m_J}\to L_p\) is an isometry, and \(P_J\) is a norm-one projection.
Put \(a_A=C_{\mathrm{av},\Omega,p,s}M_A\) for \(A\in\{S,L\}\).

\begin{proposition}\label{prop:integral-factor-uap-ranks}
    For \(A\in\{S,L\}\) and \(J\geq0\), we have
    \[ \norm{A-P_JA}\leq a_A2^{-Js},\qquad \norm{P_JA}\leq\norm A,\qquad \operatorname{rank}(P_JA)\leq2^{Jd}. \]
    Consequently, for every \(n\in\N_+\), we have \(k_A(2,n)\leq2^d\left(1+\frac{a_A}{\norm A}n\right)^{d/s}+n\).
\end{proposition}

\begin{proof}
    By \eqref{eq:sobolev-averaging-estimate} and the smoothing assumption, we have \(\norm{(I-P_J)Au}_{L_p}\leq a_A2^{-Js}\norm{u}_{L_p}\).
    The other two bounds follow from the contractivity and rank of \(P_J\).
    Choose the least \(J\geq0\) such that \(a_A2^{-Js}\leq\norm A/n\).
    By minimality, \(2^{Jd}\leq2^d(1+a_An/\norm A)^{d/s}\).
    For each \(E\subseteq L_p(\Omega)\) with \(1\leq\dim E\leq n\), Lemma~\ref{lem:finite-dimensional-correction} gives \(A_E|_E=A|_E\), with \(\norm{A_E}\leq\norm A+(\dim E)\norm A/n\leq2\norm A\) and \(\operatorname{rank}A_E\leq2^{Jd}+n\).
    By Definition~\ref{def:operator-UAP-rank}, we obtain the claimed rank bound.
\end{proof}

Thus Theorem~\ref{thm:UAP-profile-operator-template} applies with \(\Lambda_S=\Lambda_L=2\).
For this example, we can also use the explicit reductions \(\widetilde S_J=P_JS=R_J(C_JS)\) and \(\widetilde L_H=P_HL=R_H(C_HL)\).
Let \(K\Subset L_p(\Omega)\) be nonempty, and put \(R_K=\sup_{u\in K}\norm{u}_{L_p}\), \(m=2^{Jd}\), and \(q=2^{Hd}\).
Put \(\Phi_{J,H}=C_HLFR_J:\ell_p^m\to\ell_p^q\).
Then \(R_H\Phi_{J,H}(C_JSu)=P_HLF(P_JSu)\).
The input measurements \(C_JSu\) are bounded linear functionals of \(u\).

\begin{proposition}\label{prop:integral-operator-recovery}
    Let \(\eta>0\).
    If \(N:\R^m\to\R^q\) satisfies \(\sup_{z\in C_JS(K)}\norm{\Phi_{J,H}(z)-N(z)}_{\ell_p^q}\leq\eta\), then
    \begin{equation}\label{eq:integral-operator-error}
        \sup_{u\in K}\norm{T_F(u)-R_HN(C_JSu)}_{L_p} \leq L_\varphi R_K \bigl(\norm L\,a_S2^{-Js}+\norm S\,a_L2^{-Hs}\bigr)+\eta.
    \end{equation}
    For every \(\varepsilon_{\rm in},\varepsilon_{\rm out}>0\), we can choose \(J,H\) so that the right-hand side is at most \(\varepsilon_{\rm in}+\varepsilon_{\rm out}+\eta\), with
    \begin{equation}\label{eq:integral-coding-ranks}
        \begin{aligned}
            m\leq2^d\left(1+\frac{L_\varphi R_K\norm L\,a_S}{\varepsilon_{\rm in}}\right)^{d/s},\qquad q\leq2^d\left(1+\frac{L_\varphi R_K\norm S\,a_L}{\varepsilon_{\rm out}}\right)^{d/s}.
        \end{aligned}
    \end{equation}
\end{proposition}

\begin{proof}
    Since \(\norm{F(v)}_{L_p}\leq L_\varphi\norm v_{L_p}\), we have
    \begin{align*}
        \norm{T_F(u)-P_HLF(P_JSu)}_{L_p} &\leq\norm{P_HL}\,L_\varphi\norm{(I-P_J)Su}_{L_p} +a_L2^{-Hs}\norm{F(Su)}_{L_p}\\
        &\leq L_\varphi R_K \bigl(\norm L\,a_S2^{-Js}+\norm S\,a_L2^{-Hs}\bigr)
    \end{align*}
    for \(u\in K\).
    The decoder \(R_H:\ell_p^q\to L_p\) is an isometry, so the network contributes at most \(\eta\).
    For \eqref{eq:integral-coding-ranks}, choose the least \(J,H\geq0\) for which the two displayed terms are at most \(\varepsilon_{\rm in}\) and \(\varepsilon_{\rm out}\), respectively.
    The same dyadic count as in Proposition~\ref{prop:integral-factor-uap-ranks} gives the result.
\end{proof}

Put \(g_J(t)=m^{-1/p}\varphi(m^{1/p}t)\), so that \(\omega_{g_J}(t)\leq L_\varphi t\).
Since each cell has measure \(1/m\), we have \(FR_Jz=R_J(g_J(z_1),\ldots,g_J(z_m))\).
Put \(B=\max\{1,\norm S R_K\}\).
The matrix \(M_{H,J}=C_HLR_J:\ell_p^m\to\ell_p^q\) has norm at most \(\norm L\), and \(C_JS(K)\subseteq Q_m(B)\).

We use \(\mathcal R\) from \eqref{eq:relu-budget-inversion} and the constants \(a_1,c_1,W_1,L_1\) from Proposition~\ref{prop:quantitative-continuous-relu} in dimension one.
We write \(W_{\rm train}\) for the slot count of the ReLU network \(\Psi_J\) with fixed decoder \(D_{H,J}=R_HM_{H,J}\), and \(W_{\rm all}\) for the slot count of \(N=M_{H,J}\Psi_J\), including the matrix coefficients and all zero-valued slots.
Both counts exclude the fixed maps \(C_JS,R_H\).

\begin{corollary}\label{cor:integral-operator-parameter-bounds}
    Under the preceding assumptions, put \(r=d/s\) and fix \(L_\varphi\).
    For every \(\varepsilon>0\), each operator \(T_F=LFS\), with \(\varphi(0)=0\) and \(\omega_\varphi(t)\leq L_\varphi t\), admits a ReLU approximation with uniform error at most \(\varepsilon\) on \(K\).
    The graphs and fixed linear maps can be chosen independently of \(\varphi\), with
    \[ W_{\rm train} =O\!\left(\frac{\varepsilon^{-[1+r(1+1/p)]}}{\log(1/\varepsilon)}\right),\qquad W_{\rm all} =O\!\left(\frac{\varepsilon^{-[1+r(1+1/p)]}}{\log(1/\varepsilon)} +\varepsilon^{-2r}\right). \]
    The bounds hold as \(\varepsilon\downarrow0\), with respective depths at most \(L_1\) and \(L_1+1\).
    If \(S:L_p(\Omega)\to L_\infty(\Omega)\) is bounded, the exponent \(1+r(1+1/p)\) in both estimates can be replaced by \(r+1\).
    The implicit constants depend only on the fixed problem data and the one-dimensional approximation constants.
\end{corollary}

\begin{proof}
    Take \(\varepsilon_{\rm in}=\varepsilon_{\rm out}=\eta=\varepsilon/3\) and the least scales in Proposition~\ref{prop:integral-operator-recovery}.
    By \eqref{eq:integral-coding-ranks}, we have \(m,q=O(\varepsilon^{-r})\).
    Put
    \[ w=\max\left\{W_1,\mathcal R\!\left( \frac{2Ba_1c_1L_\varphi\norm L\,m^{1/p}}{\eta}\right)\right\}. \]
    Define \(\Psi_J(z)=(h_J(z_1),\ldots,h_J(z_m))\), using a scalar network \(h_J\) from Proposition~\ref{prop:quantitative-continuous-relu} with at most \(w\) slots and depth at most \(L_1\).
    For \(N=M_{H,J}\Psi_J\), the scalar estimate and \(\norm{x}_{\ell_p^m}\leq m^{1/p}\norm{x}_\infty\) give
    \[ \sup_{z\in C_JS(K)}\norm{\Phi_{J,H}(z)-N(z)}_{\ell_p^q} \leq\frac{2Ba_1c_1L_\varphi\norm L\,m^{1/p}}{w\log(e+w)} \leq\eta. \]
    Proposition~\ref{prop:integral-operator-recovery} gives error at most \(\varepsilon\) for \(D_{H,J}\Psi_J(C_JSu)=R_HN(C_JSu)\).
    Also, \(\norm{D_{H,J}}\leq\norm L\) since \(R_H\) is an isometry.
    Counting the \(m\) disjoint branches separately gives \(W_{\rm train}\leq mw\) and depth at most \(L_1\).

    To include the matrix operation, replace each final scalar output by \(\sigma(h_J)\) and \(\sigma(-h_J)\), and append the affine layer \([M_{H,J},-M_{H,J}]\).
    This at most doubles the branch slots and adds \(2qm\) weights, \(q\) bias slots, and one layer.
    Hence \(W_{\rm all}\leq2mw+2qm+q\), with depth at most \(L_1+1\).
    The graphs are common to the class because the scalar architecture is independent of \(\varphi\).
    Using \(m,q=O(\varepsilon^{-r})\) and \eqref{eq:relu-budget-inversion} in these counts gives the stated parameter bounds.

    If \(S:L_p\to L_\infty\) is bounded, put \(B_\infty=\max\{1,\norm{S:L_p\to L_\infty}R_K\}\).
    The normalized cell averages satisfy \(C_JS(K)\subseteq Q_m(m^{-1/p}B_\infty)\).
    Applying the same scalar estimate to \(g_J\) on \([-m^{-1/p}B_\infty,m^{-1/p}B_\infty]\) allows us to replace \(w\) by
    \[ w_\infty=\max\left\{W_1,\mathcal R\!\left( \frac{2B_\infty a_1c_1L_\varphi\norm L}{\eta}\right)\right\}. \]
    Since \(w_\infty=O(\varepsilon^{-1}/\log(1/\varepsilon))\), the same counts give the improved bounds.
\end{proof}

For the Gaussian kernel, \(\abs{k_S}\leq1\) and \(\abs\Omega=1\), so H\"older's inequality gives \(\norm{S:L_p\to L_\infty}\leq1\) and we may use \(B_\infty=\max\{1,R_K\}\).
By taking \(s=1\) in these bounds and using \(d\geq1\), we obtain
\[ W_{\rm train}=O\!\left(\frac{\varepsilon^{-(d+1)}}{\log(1/\varepsilon)}\right), \qquad W_{\rm all}=O(\varepsilon^{-2d}). \]

Under the factor conditions below, constant inputs give an exact scalar representation of the operator class.
We obtain the parameter lower bound from this scalar restriction.

\begin{proposition}\label{prop:integral-operator-lower-bound}
    Suppose that \(S\mathbf1=\mathbf1\), \(v=L\mathbf1\ne0\), and \(K_b:=\{t\mathbf1:\abs t\leq b\}\subseteq K\), where \(b>0\).
    Let \(\varphi\) range over all functions with \(\varphi(0)=0\) and \(\omega_\varphi(t)\leq L_\varphi t\), for fixed \(L_\varphi>0\).
    Put \(\tau=\norm v_{L_p}\), and choose \(\ell\in L_p(\Omega)^*\) with \(\norm\ell=1\) and \(\ell(v)=\tau\).
    If \(\mathcal{NN}_{W,D}\) satisfies Assumption~\ref{ass:scalar-readout-realization} for this single readout and uniformly \(\varepsilon\)-approximates all corresponding \(T_F\) on \(K\), then
    \begin{equation}\label{eq:integral-parameter-lower-bound}
        WD(1+\log W)\geq\frac{bL_\varphi\tau}{16C_{\rm pd}\varepsilon}, \qquad 0<\varepsilon\leq\frac{bL_\varphi\tau}{16}.
    \end{equation}
\end{proposition}

\begin{proof}
    We have \(T_F(t\mathbf1)=\varphi(t)v\).
    Put \(\delta=4\varepsilon/(L_\varphi\tau)\), \(n=\lfloor b/(2\delta)\rfloor\), and \(t_j=(2j-1)\delta\).
    For \(\theta\in\{0,1\}^n\), define \(\varphi_\theta(t)=L_\varphi\sum_{j=1}^n \theta_j(\delta-\abs{t-t_j})_+.\)
    The triangular bumps have disjoint interiors of their supports in \([0,b]\).
    Thus \(\varphi_\theta(0)=0\) and \(\omega_{\varphi_\theta}(t)\leq L_\varphi t\).
    At the inputs \(t_j\mathbf1\), the readout of the target takes independently the values \(0\) and \(4\varepsilon\).
    Thus the scalar readout class pseudo-shatters these \(n\) inputs with thresholds \(2\varepsilon\).
    By Lemma~\ref{lem:scalarized-readout-pdim}, \(C_{\rm pd}WD(1+\log W)\geq n\geq b/(4\delta)=bL_\varphi\tau/(16\varepsilon)\), where the second inequality uses \(b/(2\delta)\geq2\).
\end{proof}

On \(K=K_b\), this lower bound has a matching upper bound.
The fixed encoder \(C(u)=\int_\Omega u\) and decoder \(R(a)=av\) reduce the problem exactly to scalar approximation on \([-b,b]\).
Let \(W_D^{\rm int}(\varepsilon)\) be the least slot count of one fixed scalar ReLU graph of depth at most \(D\) whose realizations, with these fixed maps, approximate the whole operator class on \(K_b\) within \(\varepsilon\).
The scalar readout condition holds because \(\ell R=\tau I\).
For \(D\geq L_1\), Proposition~\ref{prop:quantitative-continuous-relu} gives
\[ W_D^{\rm int}(\varepsilon)\leq \max\left\{W_1,\mathcal R\!\left( \frac{2ba_1c_1L_\varphi\tau}{\varepsilon}\right)\right\}. \]
By this estimate, \eqref{eq:integral-parameter-lower-bound}, and \eqref{eq:relu-budget-inversion}, for fixed \(D\geq L_1\) and fixed problem data, we obtain \[W_D^{\rm int}(\varepsilon)\asymp\frac{\varepsilon^{-1}}{\log(1/\varepsilon)} \qquad \text{as} \quad \varepsilon\downarrow0.\]

The additional factor conditions hold for the normalized Gaussian kernels
\( k_S(x,y)=k_L(x,y)= e^{-\norm{x-y}_2^2}/\int_\Omega e^{-\norm{x-z}_2^2}\,\rmd z. \)
These kernels are \(C^1\) and satisfy \(S\mathbf1=L\mathbf1=\mathbf1\), so the smoothing assumption holds with \(s=1\) and \(\tau=1\).

    \section*{Acknowledgments}
	Rui Liu and Jie Shen were partially supported by the National Natural Science Foundation of China (Grant No. 12471131).

    \bibliographystyle{elsarticle-num}
    \bibliography{cite}

\end{document}